%% file: VBDP.tex
\documentclass[10pt]{amsart}
\input{amssymbol}
\thanks{The author is partly supported by National Science Council grant 101-2115-M-002-010-MY2}
\title{
Special values of anticyclotomic Rankin-Selberg $L$-functions}
\author[M.-L. Hsieh]{Ming-Lun Hsieh}
\date{August 23, 2012}
\address{ Department of Mathematics~\\National Taiwan University ~ \\
No. 1, Sec. 4, Roosevelt Road, Taipei 10617, Taiwan~
}
\email{mlhsieh@math.ntu.edu.tw}
\begin{document}
\begin{abstract}
In this article, we prove an explicit Waldspurger formula for the toric Hilbert modular forms. As an application, we construct a class of anticyclotomic \padic Rankin-Selberg $L$-functions for Hilbert modular forms, generalizing the construction of Bertolini, Darmon, and Prasanna in the elliptic case. Moreover, building on works of Hida, we give a necessary and sufficient condition when the Iwasawa $\mu$-invariant of this \padic $L$-function vanishes and prove a result on the non-vanishing modulo $p$ of central Rankin-Selberg $L$-values with anticyclotomic twists. \end{abstract}
\maketitle
\tableofcontents

\input{Waldmain}

\input{WaldPHF}

\input{Waldvmu}
\input{Waldnv}
\bibliographystyle{amsalpha}
\bibliography{C:/users/MingLun/texsetting/mybib}
\end{document}

%% file: amssymbol.tex
\usepackage{amsmath}
\usepackage{amscd,amsthm,amssymb,amsfonts}
\usepackage{mathrsfs}
\usepackage{dsfont}
\usepackage{stmaryrd}
\usepackage{euscript}
\usepackage{expdlist}
\usepackage{enumerate}

\input xy
\xyoption{all}
\usepackage[OT2,T1]{fontenc}

\theoremstyle{plain}
\newtheorem{thm}{Theorem}[section]
\newtheorem{thmA}{Theorem}

\newtheorem{hypA}{Hypothesis}

\newtheorem*{thm*}{Theorem}

\newtheorem{lm}[thm]{Lemma}

\newtheorem*{cor*}{Corollary}
\newtheorem{prop}[thm]{Proposition}
\newtheorem*{conj*}{Conjecture}

\theoremstyle{remark}

\theoremstyle{definition}
\newtheorem*{defn*}{Definition}
\newtheorem{Remark}[thm]{Remark}
\newtheorem{defn}[thm]{Definition}

\newcommand{\nc}{\newcommand}

\newcommand{\beq}{\begin{equation}}
\newcommand{\eeq}{\end{equation}}
\newcommand{\bpmx}{\begin{pmatrix}}
\newcommand{\epmx}{\end{pmatrix}}
\newcommand{\bbmx}{\begin{bmatrix}}
\newcommand{\ebmx}{\end{bmatrix}}
\newcommand{\wh}{\widehat}
\newcommand{\wtd}{\widetilde}

\newcommand{\beqcd}[1]{\begin{equation*}\label{#1}\tag{#1}}
\newcommand{\eeqcd}{\end{equation*}}

\numberwithin{equation}{section}
\newenvironment{mylist}{
  \begin{enumerate}{}{%
      \setlength{\itemsep}{5pt} \setlength{\parsep}{0in}
      \setlength{\parskip}{0in} \setlength{\topsep}{0in}
      \setlength{\partopsep}{0in}
      \setlength{\leftmargin}{0.17in}}}{\end{enumerate}}

\def\parref#1{\ref{#1}}
\def\thmref#1{Theorem~\parref{#1}}
\def\propref#1{Prop.~\parref{#1}}
     
\def\secref#1{\S\parref{#1}}

\def\lmref#1{Lemma~\parref{#1}}
\def\subsecref#1{\S\parref{#1}}
\def\defref#1{Def.~\parref{#1}}

\def\hypref#1{Hypothesis~\parref{#1}}

\def\makeop#1{\expandafter\def\csname#1\endcsname
  {\mathop{\rm #1}\nolimits}\ignorespaces}

\makeop{Hom}   \makeop{End}   \makeop{Aut}   
\makeop{Pic} \makeop{Gal}       \makeop{Div} \makeop{Lie}
\makeop{PGL}   \makeop{Corr} \makeop{PSL} \makeop{sgn} \makeop{Spf}
 \makeop{Tr} \makeop{Nr} \makeop{Fr} \makeop{disc}
\makeop{Proj} \makeop{supp} \makeop{ker}   \makeop{Im} \makeop{dom}
\makeop{coker} \makeop{Stab} \makeop{SO} \makeop{SL} \makeop{SL}
\makeop{Cl}    \makeop{cond} \makeop{Br} \makeop{inv} \makeop{rank}
\makeop{id}    \makeop{Fil} \makeop{Frac}  \makeop{GL} \makeop{SU}
\makeop{Trd}   \makeop{Sp} \makeop{Tr}    \makeop{Trd} \makeop{Res}
\makeop{ind} \makeop{depth} \makeop{Tr} \makeop{st} \makeop{Ad}
\makeop{Int} \makeop{tr}    \makeop{Sym} \makeop{can} \makeop{SO}
\makeop{torsion} \makeop{GSp} \makeop{Tor}\makeop{Ker} \makeop{rec}
\makeop{Ind} \makeop{Coker}
 \makeop{vol} \makeop{Ext} \makeop{gr} \makeop{ad}
 \makeop{Gr}\makeop{corank} \makeop{Ann}
\makeop{Hol} 
\makeop{Fitt} \makeop{Mp} \makeop{CAP}

\def\Sel{Sel}

\def\Spec{\mathrm{Spec}\,}

\DeclareMathAlphabet{\mathpzc}{OT1}{pzc}{m}{it}
\DeclareSymbolFont{cyrletters}{OT2}{wncyr}{m}{n}
\DeclareMathSymbol{\SHA}{\mathalpha}{cyrletters}{"58}

\def\makebb#1{\expandafter\def
  \csname bb#1\endcsname{{\mathbb{#1}}}\ignorespaces}
\def\makebf#1{\expandafter\def\csname bf#1\endcsname{{\bf
      #1}}\ignorespaces}
\def\makegr#1{\expandafter\def
  \csname gr#1\endcsname{{\mathfrak{#1}}}\ignorespaces}
\def\makescr#1{\expandafter\def
  \csname scr#1\endcsname{{\EuScript{#1}}}\ignorespaces}
\def\makecal#1{\expandafter\def\csname cal#1\endcsname{{\mathcal
      #1}}\ignorespaces}

\def\doLetters#1{#1A #1B #1C #1D #1E #1F #1G #1H #1I #1J #1K #1L #1M
                 #1N #1O #1P #1Q #1R #1S #1T #1U #1V #1W #1X #1Y #1Z}
\def\doletters#1{#1a #1b #1c #1d #1e #1f #1g #1h #1i #1j #1k #1l #1m
                 #1n #1o #1p #1q #1r #1s #1t #1u #1v #1w #1x #1y #1z}
\doLetters\makebb   \doLetters\makecal  \doLetters\makebf
\doLetters\makescr
\doletters\makebf   \doLetters\makegr   \doletters\makegr

\def\Gm{{\bbG}_{m}}

\normalsize

\makeop{Ram} \makeop{Rep} \makeop{mass}

\makeop{Bl}
\def\abs#1{\left|#1\right|}

\def\Fpbar{\bar{\mathbb F}_p}

\def\Qbarp{\C_p}
\def\Qp{\Q_p}
\def\Qbar{\bar\Q}
\def\Zbar{\bar{\Z}}
\def\Zbarp{\Zbar_p}

\def\Zp{\Z_p}

\def\rmT{{\mathrm T}}
\def\rmN{{\mathrm N}}

\def\cA{{\mathcal A}}  

\def\cD{\mathcal D}
\def\cE{{\mathcal E}}
\def\cF{{\mathcal F}}  

\def\cL{{\mathcal L}}

\def\cI{\mathcal I}

\def\cK{{\mathcal K}}  
\def\cM{\mathcal M}
\def\cR{{\mathcal R}}
\def\cO{\mathcal O}
\def\cS{{\mathcal S}}
\def\cf{{\mathcal f}}
\def\cW{{\mathcal W}}

\def\cP{{\mathcal P}}
\def\cC{\mathcal C}

\def\cT{\mathcal T}

\def\cU{\mathcal U}

\def\EucA{{\EuScript A}}

\def\EucF{{\EuScript F}}  

\def\EucO{{\EuScript O}}

\def\bfc{\mathbf c}

\def\bfM{\mathbf M}

\def\bff{\mathbf f}
\def\bdh{\mathbf h}

\def\bfi{\mathbf i}

\def\bdd{\mathbf d}
\def\bdc{\mathbf c}

\def\bfw{\mathbf w}

\def\bfdelta{\boldsymbol{\delta}}
\def\bftheta{\boldsymbol{\theta}}

\def\sL{\mathscr L}

\def\sS{\mathscr S}

\def\bbI{\mathbb I}

\newcommand{\Z}{\mathbf Z}
\newcommand{\Q}{\mathbf Q}
\newcommand{\R}{\mathbf R}
\newcommand{\C}{\mathbf C}
\newcommand{\A}{\mathbf A}    

\def\bbmu{\boldsymbol{\mu}}

\def\fraka{{\mathfrak a}}
\def\frakb{{\mathfrak b}}
\def\frakc{{\mathfrak c}}

\def\frakp{{\mathfrak p}}
\def\frakP{\mathfak P}

\def\frakr{\mathfrak r}

\def\frakm{\mathfrak m}
\def\frakn{\mathfrak n}
\def\frakl{\mathfrak l}
\def\frakP{\mathfrak P}
\def\frakF{{\mathfrak F}}
\def\frakE{\mathfrak E}
\def\frakL{{\mathfrak L}}

\def\frakA{\mathfrak A}

\def\frakC{{\mathfrak C}}

\def\frakX{\mathfrak X}

\def\frakN{\mathfrak N}
\def\il{\mathfrak i\frakl}

\def\ulA{\ul{A}}

\def\ulz{\ul{z}}

\def\ollam{\bar{\lam}}

\def\Zhat{\hat{\Z}}
\def\wbar{\bar{w}}

\def\wbar{\bar{w}}
\def\zbar{\bar{z}}

\def\etale{{\'{e}tale }}

\def\pont{Pontryagin } 
\def\padic{\text{$p$-adic }}

\def\BS{Bruhat-Schwartz }

\def\Frob{\mathrm{Frob}}

\newcommand{\<}{\langle}   
\renewcommand{\>}{\rangle} 

\def\isoto{\stackrel{\sim}{\to}}

\def\imply{\Rightarrow}
\def\ot{\otimes}

\def\hookto{\hookrightarrow}
\def\longto{\longrightarrow}
\def\ol{\overline}  \nc{\opp}{\mathrm{opp}} \nc{\ul}{\underline}

\newcommand{\pair}[2]{\< #1, #2\>}

\newcommand{\pairing}{\pair{\,}{\,}}

\def\XYmatrix{\xymatrix@M=8pt} 
\def\ncmd{\newcommand}
\ncmd{\xysubset}[1][r]{\ar@<-2.5pt>@{^(-}[#1]\ar@<2.5pt>@{_(-}[#1]}
\ncmd{\XYmatrixc}[1]{\vcenter{\XYmatrix{#1}}}
\ncmd{\xyto}[1][r]{\ar@{->}[#1]}
\ncmd{\xyinj}[1][r]{\ar@{^(->}[#1]}
\ncmd{\xysurj}[1][r]{\ar@{->>}[#1]}
\ncmd{\xyline}[1][r]{\ar@{-}[#1]}
\ncmd{\xydotsto}[1][r]{\ar@{.>}[#1]}
\ncmd{\xydots}[1][r]{\ar@{.}[#1]}
\ncmd{\xyleadsto}[1][r]{\ar@{~>}[#1]}
\ncmd{\xyeq}[1][r]{\ar@{=}[#1]} \ncmd{\xyequal}[1][r]{\ar@{=}[#1]}
\ncmd{\xyequals}[1][r]{\ar@{=}[#1]}
\ncmd{\xymapsto}[1][r]{l\ar@{|->}[#1]}\ncmd{\xyimplies}[1][r]{\ar@{=>}[#1]}
\ncmd{\xyiso}{\ar[r]_-{\sim}}
\def\injxy{\ar@{^(->}}

\newcommand{\MX}[4]{\begin{bmatrix}
{#1}& {#2}\\
{#3}&{#4}\end{bmatrix} }

 \newcommand{\DII}[2]{\begin{bmatrix}{#1}&0
 \\0&{#2}\end{bmatrix}}

\newcommand{\seesaw}[4]{{#1}\ar@{-}[rd]\ar@{-}[d]&{#2}\ar@{-}[d]\\
{#3}\ar@{-}[ru]&{#4}}

\def\ie{i.e. }

\def\cf{\mbox{{\it cf.} }}
\def\loccit{\mbox{{\it loc.cit.} }}

\def\can{{can}}

\def\ENS{\mathfrak E\mathfrak n\mathfrak s}
\def\SCH{\mathfrak S\mathfrak c\mathfrak h}

\def\Ch{\mathrm{char}}

\def\uf{\varpi} 
\def\Abs{{|\!\cdot\!|}} 

\def\Sg{{\varSigma}}  
\def\Sgbar{\Sg^c}
\def\ndivides{\nmid}
\def\ndivide{\nmid}
\def\x{{\times}}

\def\onehalf{{\frac{1}{2}}}
\def\e{\varepsilon} 
\def\al{\alpha}

\def\Lam{\Lambda}

\def\om{\omega}
\def\dirlim{\varinjlim}
\def\prolim{\varprojlim}
\def\iso{\simeq}
\def\con{\equiv}
\def\bksl{\backslash}
\newcommand\stt[1]{\left\{#1\right\}}
\def\ep{\epsilon}

\def\lam{\lambda}
\def\pii{\pi i}

\def\sg{\sigma}
\def\vp{\varphi}
\def\disjoint{\bigsqcup}
\def\bigot{\bigotimes}

\def\dx{d^\x}

\def\AFf{\A_{\cF,f}}
\def\AKf{\A_{\cK,f}}
\def\AF{\A_\cF}
\def\AK{\A_\cK}
\def\setp{{(p)}}

\def\bbox{{(\Box)}}

\newcommand{\powerseries}[1]{\llbracket{#1}\rrbracket}
\newcommand{\formal}[1]{\widehat{#1}}
\renewcommand\pmod[1]{\,(\mbox{mod }{#1})}

\renewcommand\Re{\text{Re}\,}
\newcommand\Dmd[1]{\left<{#1}\right>} 
\def\vphi{\varphi}
\def\Cp{\C_p}

\def\alg{\mathrm{alg}} 

%% file: Waldmain.tex
\def\ZZbox{\Z_{(\Box)\,}}
\def\cAbox{\cA^{(\Box)}_{K,\bdc}}
\def\cAboxn{\cA^{(\Box)}_{K,\bdc,n}}
\def\Zhatbox{\widehat{\Z}^{(\Box)}}
\def\qchKF{\tau_{\cK/\cF}}
\def\opcpt{K}
\def\sh{Sh}
\def\opn{K^n}
\def\lp{j}
\def\lpp{\eta^{(p)}}
\def\lsgN{\opcpt_1^n}
\def\Om{\boldsymbol{\omega}}
\def\wt{k}
\def\skewhf{\vartheta}
\def\Fv{F}
\def\Kv{E}
\def\OFp{\cO_{\cF_p}}
\def\ads{\chi}
\def\tbar{\ol{t}}
\def\wbar{\ol{w}}
\def\Section{\vp}
\def\adelef{\A_{\cF,f}}
\def\cmpt{\varsigma}
\def\cmptv{\varsigma_v}
\def\skewhf{\vartheta}
\def\OK{\cO_\cK}
\def\cmJ{\boldsymbol{J}}
\def\OFv{\cO_F}
\def\OKv{\cO_E}
\def\OF{O}
\def\LR{\cR}
\def\addchar{\psi}
\def\cK{E}
\def\cF{F}
\def\Op{O_p}
\def\torsbgp{\cU_p}
\def\loca{\mathrm{l.a.}}
\def\baseR{R}
\def\Prd{\boldsymbol{P}}
\def\wtsp{\frakX^{\mathrm{crit}}_{p}}
\def\IwGamma{\Gamma^-}
\def\AF{\A}
\def\AFf{\AF_f}
\def\CMP{\vartheta}
\def\OKbasis{\bftheta}
\newcommand\localW[1]{W_{#1}}
\newcommand\localK[1]{\bfa_{#1}}
\newcommand\localP[2]{\Prd(#1,#2)}
\def\rmN{\mathrm N}
\def\Sgbar{\ol{\Sg}}
\renewcommand\MX[4]{\begin{pmatrix}#1&#2\\
#3&#4\end{pmatrix}}
\renewcommand\DII[2]{\begin{pmatrix}#1&\\
&#2\end{pmatrix}}
\def\Cinert{\frakN^-}
\def\Csplit{\frakF}
\def\Cram{D_{\cK/\cF}}
\def\wt{k}
\def\plideal{\frakc}
\def\dt{dt}
\def\Bad{\frakr}
\def\CMring{{\cW_p}}
\def\#{\sharp}
\section*{Introduction}
The purpose of this article is to (i) construct a class of anticyclotomic Rankin-Selberg \padic $L$-functions and study the vanishing/non-vanishing of the associated Iwasawa $\mu$-invariant, (ii) prove a result on the non-vanishing modulo $p$ of central Rakin-Selberg $L$-values with anticyclotomic twists.
Let $\cF$ be a totally real algebraic number field of degree $d$ over $\Q$ and $\cK$ be a totally imaginary quadratic extension of $\cF$. Denote by $z\mapsto \ol{z}$ the nontrivial element in $\Gal(\cK/\cF)$. Let $\pi$ be an irreducible cuspidal automorphic representation of $\GL_2(\A_{\cF})$ with unitary central character $\om$. Let $\pi_\cK$ be a lifting of $\pi$ to $\cK$ constructed in \cite[Thm.\,20.6]{Jacquet:GLtwoPartII}. Then $\pi_\cK$ is an irreducible automorphic representation of $\GL_2(\A_\cK)$, which is cuspidal if $\pi$ is not obtained from the automorphic induction from $\cK/\cF$. Let $\lam:\AK^\x/\cK^\x\to\C^\x$ be a unitary Hecke character of $\cK^\x$ such that
\beq\label{E:self-dual}\lam|_{\A_{\cF}^\x}=\om^{-1}.\eeq
The automorphic representation $\pi_\cK\ot\lam$ is therefore conjugate self-dual, \ie $\pi_\cK^\vee\ot\lam^{-1}=\ol{\pi_\cK\ot\lam}$. For each place $v$ of $\cF$, we can associate a local $L$-function $L(s,\pi_{\cK_v}\ot\lam_v)$ and a local epsilon factor $\ep(s,\pi_{\cK_v}\ot\lam_v,\psi_v)$ (which depends on a choice of non-trivial character $\addchar_v:\cF_v\to\C^\x$) to the local constituent $\pi_{\cK_v}\ot\lam_v$ of $\pi_{\cK}\ot\lam$ (\cite[Thm.\,2.18 (iv)]{Jacquet_Langlands:GLtwo}). Let $L(s,\pi_\cK\ot\lam)$ be the global $L$-function obtained by the meromorphic continuation of the Euler product of local $L$-functions at all \emph{finite} places. In this paper, we study the \padic variation of the central value $L(\onehalf,\pi_{\cK}\ot\lam)$ with anticyclotomic twists under certain hypotheses.

To introduce our hypotheses precisely, we need some notation.  Fix a CM-type $\Sg$ of $\cK$. Namely, $\Sg$ is a subset of $\Hom(\cK,\C)$ such that
\[\Sg\disjoint \Sgbar=\Hom(\cK,\C)\,;\,\Sg\cap\Sgbar=\emptyset.\]
Then $\Sg$ induces an identification $\cK\ot_\Q\R\iso\C^\Sg$. We shall identify $\Sg$ with the set of archimedean places of $\cF$ via the restriction.  
For each $\wt=\sum_{\sg\in\Sg}k_\sg\sg\in\Z[\Sg]$, we write $\Gamma_\Sg(\wt)=\prod_{\sg\in\Sg}\Gamma(k_\sg)$, and if $x=(x_\sg)\in (A^\x)^\Sg$ for an algebra $A$, we let $x^{\wt}:=\prod_{\sg\in\Sg}x_\sg^{k_\sg}$. A Hecke character $\chi$ of $\cK^\x$ is of infinity type $(k_1,k_2)$ for $k_1,k_2\in 2^{-1}\Z[\Sg]$ such that $k_1-k_2\in\Z[\Sg]$ if \[\chi_\infty(z)=z^{k_1-k_2}(z\zbar)^{k_2}\text{ for all }z\in(\cK\ot_\Q\R)^\x\iso(\C^\x)^\Sg.\] For each ideal $\fraka$ of $\cF$ (resp. ideal $\frakA$ of $\cK$),  we have a unique factorization $\fraka=\fraka^+\fraka^-$ (resp. $\frakA=\frakA^+\frakA^-$), where $\fraka^+$ (resp. $\frakA^+$) is only divisible by primes split in $\cK$ and $\fraka^-$ (resp. $\frakA^-$) is divisible by primes inert or ramified in $\cK$. Let $\frakn=\frakn^+\frakn^-$ be the conductor of $\pi$. We define the normalized local root number attached to $\pi_{\cK_v}$ and $\lam_v$ for each place $v$ by
\[\ep^*(\pi_{\cK_v},\lam_v):=\ep(\onehalf,\pi_{\cK_v}\ot\lam_v,\addchar_v)\cdot \om_v(-1).\]
We remark that the value $\ep(\onehalf,\pi_{\cK_v}\ot\lam_v,\addchar_v)$ does not depend on the choice of $\addchar_v$.

We assume that $\pi$ has infinity type $\wt=\sum_\sg k_\sg\sg\in\Z_{>0}[\Sg]$ and $\lam$ has infinity type $(\frac{\wt}{2},-\frac{\wt}{2})$. In other words, $\pi_\sg$ is a discrete series or limit of discrete series of weight $k_\sg$ with unitary central character at every archimedean place $\sg$. In particular, this implies $\stt{k_\sg}_{\sg\in\Sg}$ have the same parity and the local root number $\ep^*(\pi_{\cK_v},\lam_v)=+1$ at every archimedean place. We further assume the following local root number hypothesis for $(\pi,\lam)$:
\begin{hypA}\label{H:HypA}The local root number $\ep^*(\pi_{\cK_v},\lam_v)=+1$ for each $v\mid\frakn^-$.
\end{hypA}
Let $p$ be an odd rational prime. Fix an embedding $\iota_\infty:\Qbar\hookto\C$ and isomorphism $\iota:\C\isoto\Cp$ once and for all. Each $\xi\in\Z[\Sg\disjoint\Sgbar]$ will be regarded as an algebraic character $\xi:(E\ot\Qp)^\x\to\Cp^\x$ via $\iota$. Throughout this article, we make the following assumption \begin{align}
\label{ord}\tag{ord}&\text{$\Sg$ is $p$-ordinary}.
\end{align} Let $\Sg_p$ be the set of \padic places of $\cK$ induced by $\Sg$. Then the ordinary assumption \eqref{ord} means that $\Sg_p$ and its complex conjugation $\Sgbar_p$ gives a full partition of the set of \padic places of $\cK$. Let $\cK^-_{p^\infty}$ be the maximal anticyclotomic $\Zp^{[\cF:\Q]}$-extension of $\cK$. Let $\Gamma^-=\Gal(\cK^-_{p^\infty}/\cK)$ and let $\Lam=\Zbarp\powerseries{\Gamma^-}$ be the Iwasawa algebra of $[\cF:\Q]$-variables. If $L$ is a number field, we write $G_L=\Gal(\Qbar/L)$ for the absolute Galois group. Denote by $\rec_\cK:\AK^\x\to G_\cK^{ab}$ the geometrically normalized reciprocity law. To each locally algebraic \padic character $\wh\phi:\Gamma^-\to\Cp^\x$ of weight $(m,-m)$, we can associate a Hecke character $\phi:\AK^\x/\cK^\x\to\C^\x$ of infinity type $(m,-m)$ defined dy
\[\phi(a):=\iota^{-1}(\wh\phi(\rec_\cK(a))a_p^{-m}\ol{a}_p^{m})a_\infty^{m}\ol{a}_\infty^{-m},\]
where $a_p\in (\cK\ot_\Q\Qp)^\x$ and $a_\infty\in (\cK\ot_\Q\R)^\x$ are the $p$-component and $\infty$-component of $a$. We say $\wh\phi$ is the \emph{\padic avatar} of $\phi$. Denote by $\wtsp$ the set of \emph{critical specializations}, consisting of locally algebraic \padic characters on $\Gamma^-$ of weight $(m,-m)$ with $m\in\Z_{\geq 0}[\Sg]$ (See \subsecref{SS:padicL}).

Our first result is the construction of the anticyclotomic primitive \padic $L$-function attached to $\pi_{\cK}\ot\lam$. We need more notation. Let $\cD_\cF$ be the different of $\cF$ and $\cD_{\cK/\cF}$ be the relative different of $\cK/\cF$. Let $\frakN$ be the prime-to-$p$ conductor of $\pi_{\cK}\ot\lam$. We have a unique decomposition $\frakN=\frakN^+\frakN^-$ and fix a decomposition $\frakN^+=\Csplit\ol{\Csplit}$ with $(\Csplit,\ol{\Csplit})=1$ such that $\frakN^-$ is only divisible by prime inert or ramified in $\cK/\cF$ and $\Csplit$ is only divisible by primes split in $\cK/\cF$ (Such a decomposition is possible because $\pi_\cK\ot\lam$ is conjugate self-dual). We choose $\delta\in\cK$ such that
\begin{itemize}
\item $\ol{\delta}=-\delta$,
\item  $\Im\sg(\delta)>0$ for all $\sg\in\Sg$,
\item The polarization ideal $\frakc(\cO_\cK):=\cD_\cF^{-1}(2\delta \cD_{\cK/\cF})$ is prime to $p\frakN\ol{\frakN}$.
\end{itemize}
Let $(\Omega_\infty,\Omega_p)$ be the complex and \padic periods attached to $(\cK,\Sg)$ defined in \cite[(4.4a), (4.4b)]{HidaTilouine:KatzPadicL_ASENS}. For each Hecke character $\chi$ of $\cK^\x$, we define the \padic multiplier $E_{\Sg_p}(\pi,\chi)$ by
\beq\label{E:coates_modified}E_{\Sg_p}(\pi,\chi):=
\prod_{\frakP\in\Sg_p,\,\frakp=\frakP\ol{\frakP}}\ep(\onehalf,\pi_\frakp\ot\chi_{\ol{\frakP}},\addchar_\frakp)L(\onehalf,\pi_\frakp\ot\chi_{\ol{\frakP}})^{-2}\om_\frakp^{-1}\chi_{\ol{\frakP}}^{-2}(-2\delta).\eeq
The shape of $E_{\Sg_p}(\pi,\chi)$ has been suggested by Coates \cite{Coates:motivic-padic-L-function}.
\begin{thmA}\label{T:A.W}In addition to \eqref{ord} and \hypref{H:HypA}, we further assume that
\beqcd{sf}\frakn^-\text{ is square-free}.\eeqcd
Then there exists an element $\cP_\Sg(\pi,\lam)\in \Lam$ such that for every $\wh\phi\in\wtsp$ of weight $(m,-m)$, we have
\[\frac{\wh\phi(\cP_\Sg(\pi,\lam)^2)}{\Omega_p^{2(\wt+2m)}}=[\cO_\cK^\x:\cO_\cF^\x]^2\cdot\frac{\Gamma_\Sg(\wt+m)\Gamma_\Sg(m+1)}{(\Im\delta)^{k+2m}(4\pi)^{k+2m+1\cdot\Sg}}\cdot E_{\Sg_p}(\pi,\lam\phi)\cdot \frac{L(\onehalf,\pi_\cK\ot\lam\phi)}{\Omega_\cK^{2(\wt+2m)}}\cdot \phi(\Csplit)C(\pi,\lam),\]
where $\Omega_\cK=(2\pii)^{-1}\Omega_\infty$ and $C(\pi,\lam)\in\Zbar_\setp^\x$ is an explicit \padic unit independent of $\phi$, consisting of a product of local epsilon factors outside $p$.
\end{thmA}
In the above expression, $\Im\delta=(\Im\sg(\delta))_{\sg\in\Sg}$ and $4\pi$ is considered to be the diagonal element $(4\pi)_{\sg\in\Sg}$ in $(\C^\x)^\Sg$.

In virtue of the above specialization formula, $\cP_\Sg(\pi,\lam)$ is the \padic $L$-function that interpolates the square root of central $L$-values. We shall call $\cL_\Sg(\pi,\lam):=\cP_\Sg(\pi,\lam)^2$ the anticyclotomic \padic $L$-function for $\pi_\cK\ot\lam$ with respect to the $p$-ordinary CM type $\Sg$. When $\cF=\Q$, $\pi$ is unramified at $p$ and $\frakn^-$ is only divisible by primes ramified in $\cK$, $\cL_\Sg(\pi,\lam)$ is constructed in \cite{BDP}, and a non-primitive version of this \padic $L$-function is also considered in \cite{Miljan:1} under different hypotheses. We remark that when $\pi$ is obtained from the automorphic induction of a Hecke character of $\cK^\x$, one can show that, up to an explicit unit in $\Lam$, $\cL_{\Sg}(\pi,\lam)$ is a product of two \padic Hecke $L$-functions for CM fields constructed by Katz \cite{Katz:p_adic_L-function_CM_fields} and Hida-Tilouine \cite{HidaTilouine:KatzPadicL_ASENS} by comparing the interpolation formulas on both sides.

Our second theorem is to prove the vanishing of the Iwasawa $\mu$-invariant $\mu^-_{\pi,\lam,\Sg}$ of $\cP_\Sg(\pi,\lam)$ under certain hypothesis. This in particular implies that the \padic $L$-function $\cL_\Sg(\pi,\lam)$ is non-trivial. Recall that the $\mu$-invariant $\mu^-_{\pi,\lam,\Sg}$ is defined by
 \[\mu^-_{\pi,\lam,\Sg}=\inf\stt{r\in\Q_{\geq 0}\mid p^{-r}\cP_\Sg(\pi,\lam)\not\con 0\pmod{\frakm_p\Lam}},\]where $\frakm_p$ is the maximal ideal of $\Zbarp$. Let $\frakc_\lam$ be the conductor of $\lam$. For each $v\mid\frakc_\lam^-$, let $\Delta_{\lam,v}$ be the finite group $\lam(\cO_{\cK_v}^\x)$.
\begin{thmA}\label{T:B.W}With the assumptions in \thmref{T:A.W}, suppose further that
\begin{mylist}\item $p$ is unramified in $\cF$,
\item the residual Galois representation \[\ol{\rho}_p(\pi_\cK):=\rho_p(\pi)|_{G_\cK}\pmod{\frakm_p}\text{ is absolutely irreducible},\]
\item $p\ndivides \prod_{v|\frakc_\lam^-}\#(\Delta_{\lam,v})$.
\end{mylist}Then $\mu^-_{\pi,\lam,\Sg}=0$.
\end{thmA}
The global assumption on the irreducibility of residual Galois representation assures that the new form associated to $\pi$ is not congruent to theta series from $\cK$ or Eisenstein series, and the local assumption (3) is equivalent to saying that the local residual character $\lam_v\pmod{\frakm_p}$ is ramified for all $v|\frakc_\lam^-$. In this situation, our \thmref{T:B.W} roughly suggests that the $\mu$-invariant is essentially contributed by the congruences among primitive but residually non-primitive \padic $L$-functions. In particular, suppose that $(\frakn,\cD_{\cK/\cF})=1$ and $\rho_p(\pi)$ is residually irreducible. Then we always have $\mu^-_{\pi,\lam}=0$ whenever $\lam$ has split conductor over $\cF$ (\ie $\frakc_\lam^-=(1)$).
In addition, \thmref{T:B.W} shares the same flavor with Vatsal's $\mu$-invariant formula. In \cite{Vatsal:nonvanishing}, Vatsal obatins the precise $\mu$-invariant formula of a different class of anticyclotomic \padic Rankin-Selberg $L$-functions associated to elliptic new forms of weight two twisted by \emph{finite order} anticyclotomic characters of $p$-power conductors (so the local root number at the archimedean place is $-1$).  In virtue of his formula, the $\mu$-invariant can be positive when we have either Eisenstein congruence or the congruence between forms with opposite signs in the functional equations. Note that the latter congruence does not happen under the assumption (3).

The \padic $L$-function $\cL_\Sg(\pi,\lam)$ is expected to encode the arithmetic information on certain Selmer groups through the main conjecture \`{a} la R. Greenberg \cite{Greenberg:Iwasawa_theory_and_motives}. To introduce the $\Lam$-adic Selmer groups connected with $\cL_\Sg(\pi,\lam)$, we recall that thanks to the works of Deligne, Carayol, Blasius-Rogawski and Taylor et.al (\cite{Blasius_Rogawski:Hilber_modular_forms}, \cite{Taylor:Galois_rep_HMF}, \cite{Jarvis:wtone}), there exists a finite extension $L_\pi$ of $\Qp$ and a continuous \padic Galois representation $\rho_p(\pi):G_\cF\to\GL_2(\cO_{L_\pi})$ such that $\rho_p(\pi)$ is unramified outside $p\frakn$, and for each finite place $v\ndivide p\frakn$, \[\iota^{-1}(L(s,\rho_p(\pi)|_{W_{\cF_v}}))=L(s+\frac{1-k_{mx}}{2},\pi_v^\vee)\quad(k_{mx}=\max_{\sg\in\Sg}k_\sg),\]
where $W_{\cF_v}$ is the Weil group of $\cF_v$.  Let $\e_\Lam:G_\cK\to\Gamma\hookto \Lam^\x,\,g\mapsto g|_{\cK_{p^\infty}^-}$ be the universal $\Lam$-adic Galois character. We consider the rank two $\Lam$-adic Galois representation:
\[\rho_\Lam:=\rho_p(\pi)|_{G_\cK}\ot\lam\e_\Lam:G_\cK\to\GL_2(\Lam).\]
and define the local condition for each $w\mid p$ by
\[F^+_w\rho_\Lam=\begin{cases}
\rho_\Lam&\text{ if } w\in\Sg_p,\\
\stt{0}&\text{ if } w\in\Sgbar_p.
\end{cases}\]
The triple $(\rho_\Lam,\stt{F^+_w\rho_\Lam}_{w|p},\wtsp)$ satisfies the Panchishkin condition in \cite[\S 4, p.217]{Greenberg:Iwasawa_theory_and_motives}.
Let $\Lam^*$ be the \pont dual of $\Lam$. The associated Greenberg's Selmer group is defined by
\[\Sel_\Sg(\pi,\lam):=\ker\stt{H^1(E,\rho_\Lam\ot_\Lam\Lam^*)\to\prod_{w\not\in\Sg_p}H^1(I_w,\rho_\Lam\ot_\Lam\Lam^*)},\]
where $w$ runs over places of $\cK$ outside $\Sg_p$ and $I_w$ is the inertia group of $w$ in $G_\cK$. It is known that the \pont dual $\Sel_\Sg(\pi,\lam)^*$ is a finitely generated $\Lam$-module. Denote by $\Ch_\Lam\Sel_\Sg(\pi,\lam)^*$ the characteristic ideal of $\Sel_\Sg(\pi,\lam)^*$. We formulate the anticyclotomic main conjecture for $\pi_E\ot\lam$ (under the hypotheses in \thmref{T:A.W}).
\begin{conj*}We have the following equality between ideals in $\Lam$
\[\Ch_\Lam\Sel_\Sg(\pi,\lam)^*=(\cL_\Sg(\pi,\lam)).\]
\end{conj*}

Let $\ell\not =p$ be a rational prime. We next consider the problem of the non-vanishing modulo $p$ of $L$-values twisted by characters of $\ell$-power conductor. This problem has been studied extensively in the literature in various settings (\cf \cite{Washington:nonvanishing}, \cite{Vatsal:nonvanishing}, \cite{Hida:nonvanishingmodp}, \cite{Finis:nonvanishingell}, \cite{Hsieh:Hecke_CM}) and has many arithmetic applications in Iwasawa theory. To state our result along this direction, we introduce some notation. Let $\frakl$ be a prime of $\cF$ above $\ell$ and let $\cK_{\frakl^\infty}^-$ be the anticyclotomic pro-$\ell$ extension in the ray class field over $\cK$ of conductor $\frakl^\infty$. Let $\Gamma_\frakl^-:=\Gal(\cK_{\frakl^\infty}^-/\cK)$ and let $\frakX_\frakl^0$ be the set consisting of finite order characters $\phi:\Gamma_\frakl^-\to\mu_{\ell^\infty}$. Let $\chi$ be a Hecke character of infinity type $(\frac{k}{2}+m,-\frac{k}{2}-m)$ and of conductor $\frakc_\chi$. We obtain the following theorem.
\begin{thmA}\label{T:C.W}
With the same assumptions in \thmref{T:A.W}, we assume that
\begin{mylist}
\item $(p\frakl,\frakn\frakc_\chi \cD_{\cK/\cF})=1$,
\item $(\frakn,\cD_{\cK/\cF})=1$ and $\rho_p(\pi)$ is residually irreducible,
\item $p\ndivides \prod_{v|\frakc_\chi^-}\#(\Delta_{\chi,v})$.
\end{mylist}
Then for almost all characters $\phi\in\frakX_\frakl^0$, we have
\[[\cO_\cK^\x:\cO_\cF^\x]^2\cdot\frac{\Gamma_\Sg(\wt+m)\Gamma_\Sg(m+1)}{(\Im\delta)^{k+2m}(4\pi)^{\wt+2m+1\cdot\Sg}}\cdot \frac{L(\onehalf,\pi_\cK\ot\chi\phi)}{\Omega_\cK^{2(\wt+2m)}}\not\con 0\pmod{\frakm_p}.\]
Here almost all means "except for finitely many $\phi\in\frakX_\frakl^-$" if $\dim_{\Q_\ell} F_\frakl=1$ and "for $\phi$ in a Zariski dense subset of $\frakX_\frakl^0$" if $\dim_{\Q_\ell}F_\frakl>1$ $($\cite[p.737]{Hida:nonvanishingmodp}$)$.
\end{thmA}
In our forthcoming work \cite{Hsieh:Namikawa}, we will apply \thmref{T:C.W} to study the non-vanishing (modulo p) of certain theta lifts constructed in \cite{Murase:FJcoeff_Kudla} and \cite{MuraseNarita:Arakawa}.
\bigskip

The proof of \thmref{T:A.W} is based on an explicit Waldspurger formula, which connects toric period integrals of automorphic forms in $\pi$ and central values of $\pi_\cK\ot\lam$ \cite{Waldsupurger:Central_value}. To obtain the optimal $p$-integrality of cental $L$-values, we will consider holomorphic \emph{toric} cusp forms and calculate explicitly their period integrals. To be precise, let $\chi$ be a Hecke character of $\cK^\x$ such that $\chi|_{\AF^\x}=\om^{-1}$ and let $\cT\subset \AK^\x$ be the subgroup consisting of ideles $z=(z_v)\in\prod_v\cK_v^\x$ with $z_v/\ol{z_v}\in\cO_{\cK_v}^\x$ for all primes $v$ split in $\cK$. Fixing an embedding $\iota:\cK^\x\hookto \GL_2(\cF)$, we say an automorphic form $\vp:\GL_2(\cF)\bksl\GL_2(\A_{\cF})\to\C$ in $\pi$ is a toric form of character $\chi$ if
\[\vp(g\iota(t))=\chi^{-1}(t)\vp(g)\text{ for all }t\in\cT.\]
In other words, $\vp$ belongs to the space $\Hom_\cT(\C,\pi\ot\chi)$.

The construction of $\cP_\Sg(\pi,\lam)$ is outlined as follows.
 \begin{mylist}
\item Construct a toric Hilbert modular form $\vp_{\lam\phi}$ of character $\lam\phi$ for each $\wh\phi\in\wtsp$ as above by a careful choice of toric local Whittaker functions in local Whittaker models of $\pi$ (See \defref{D:toric}, \lmref{L:4.W}).
\item Make an explicit calculation of the Fourier expansion of $\vp_{\lam\phi}$.
\item Via the \padic interpolation of the Fourier expansion, construct a \padic family of toric forms $\stt{\vp_{\lam\phi}}_{\wh\phi\in\wtsp}$. The \padic $L$-function $\cP_\Sg(\pi,\lam)$ is obtained by a weighted sum of the evaluation of this family at a finite set of CM points.
\end{mylist}
The evaluation of $\vp_\chi$ with $\chi=\lam\phi$ at CM points in the step (3) is actually the toric period integral $P_{\chi}(\vp_{\chi})$ given by
\[P_\chi(\vp_\chi)=\int_{\AF^\x\cK^\x\bksl\AK^\x}\vp_\chi(\iota(t))\chi(t)dt.\]
To prove the formula in \thmref{T:A.W}, we have to express the square $P_\chi(\vp_\chi)^2$ in terms of the central $L$-value $L(\onehalf,\pi_\cK\ot\chi)$. This is usually referred to as an explicit Waldspurger formula. Such a formula has been exploited widely in the literature based on either explicit theta lifts (\cite{Murase:WaldI}, \cite{Murase:WaldII}, \cite{XH:central_value}, \cite{Hida:central_value} and \cite{BDP}) or the technique of relative trace formula (\cite{Kimble_W:centrals}). In this paper we adopt a different approach, making use of a formula of Waldspurger which is indeed proved but not stated explicitly in \cite{Waldsupurger:Central_value}. This formula decomposes the square $P_\chi(\vp)$ of the global period toric integral into a product of local period integrals involving local Whittaker functions of $\vp$. Explicit computation of these local integrals shows that $P_\chi(\vp)^2$ is essentially equal to the central value of the $L$-function $L(s,\pi_\cK\ot\lam)$. We emphasize that this explicit formula does not depend on the choices of explicit \BS functions in the classical approach of theta lifting but on choices of local Whittaker functions which reflect the arithmetic of modular forms directly via the Fourier expansion. We make a few remarks on our assumptions. The restriction \eqref{sf} is due to the computational difficulty on the local period integrals and the local Fourier coefficients, and it is expected to be unnecessary. We hope to come back to the removal of this in the future. However, \hypref{H:HypA} is fundamental, the failure of which makes the period integral $P_\chi(\vp_\chi)$ vanish by a well-known theorem of Saito-Tunnell (\cite{Saito:SaitoTunnell}, \cite{Tunnell:SaitoTunnell}).

The proofs of \thmref{T:B.W} and \thmref{T:C.W} are based on the ideas of Hida in \cite{Hida:nonvanishingmodp} and \cite{Hida:mu_invariant}. Thanks to Hida's theorems on the linear independence of modular forms applied by transcendental automorphisms of the local moduli of CM points in Hilbert modular varieties modulo $p$ \loccit, the vanishing/non-vanishing modulo $p$ properties of the algebraic part of $L(\onehalf,\pi_\cK\ot\chi)$ with anticyclotomic twists can be deduced from the vanishing/non-vanishing of the Fourier expansions of the toric cusp form $\vp_\lam$ at cusps $(\fraka,\frakb)$ such that $\fraka\frakb^{-1}$ is the polarization of abelian varieties with CM by $\cO_\cK$. A new ingredient of this paper is the explicit computation of Fourier coefficients of toric new forms $\vp_\lam$ and a study on their non-vanishing modulo $p$ property in \subsecref{SS:localFourier}. Exploiting the connection between the Fourier coefficients of Hilbert modular forms and the trace of Frobenius of the associated Galois representations, we deduce from the vanishing modulo $p$ of Fourier coefficients of $\vp_\lam$ at these cusps that the trace of residual Galois representation $\ol{\rho}_p(\pi)$ is vanishing on the coset $G_\cF-G_\cK$. A simple lemma (\lmref{L:3.W}) on modular representations shows that $\ol{\rho}_p(\pi)|_{G_\cK}$ is reducible.

This paper is organized as follows. After fixing notation and definitions in \secref{S:notation.W}, we derive a key formula of Waldspurger on the decomposition of global toric period integrals into local toric period integrals (\propref{P:Waldformula.W}) in \secref{S:Wald}. The bulk of this article is \secref{S:toric}, where we give the choices of local toric Whittaker functions $W_{\chi,v}$ and calculate explicitly these local period integrals attached to $W_{\chi,v}$. The explicit Waldspurger formula is proved in \thmref{T:main1.W}, and a non-vanishing modulo $p$ of these toric Whittaker functions is proved in \propref{P:4.W}. After reviewing briefly theory of complex and geometric Hilbert modular forms in \secref{S:Hilbert.W}, we prove \thmref{T:A.W} in \subsecref{S:anticy_padic_L}. The key ingredient is  \propref{P:3.W}, the construction of a \padic measure $\EucF_{\lam,\frakc}$ on $\Gamma^-$ with values in the space of \padic modular forms, and the \padic $L$-function $\cP_\Sg(\pi,\lam)$ is thus obtained by evaluating $\EucF_{\lam,\frakc}$ at suitable CM points. The precise evaluation formula of $\cP_\Sg(\pi,\lam)^2$ is established in \thmref{T:padicL.W}. In \secref{S:mu_invariant}, we study the $\mu$-invariant of $\cP_\Sg(\pi,\lam)$ and prove \thmref{T:B.W} in \thmref{T:mu-invariant}. Finally, the non-vanishing of central $L$-values modulo $p$ is considered in \secref{S:NV} and \thmref{T:C.W} is proved in \thmref{T:main3.W}.


\section{Notation and definitions}\label{S:notation.W}
\subsection{Measures on local fields}We fix some general notation and conventions on local fields through this article. Let $\addchar_\Q:\A_\Q/\Q\to \C^\x$ be the additive character such that $\addchar_\Q(x_\infty)=\exp(2\pii x_\infty)$ with $x_\infty\in\R$. Let $q$ be a place of $\Q$ and let $F$ be a finite extension of $\Q_q$. Let $\addchar_q$ be the local component of $\addchar$ at $q$ and let $\addchar_F:=\addchar_q\circ\rmT_{F/\Q_q}$, where $\rmT_{F/\Q_q}$ is the trace from $\cF$ to $\Q_q$. Let $dx$ be the Haar measure on $F$ self-dual with respect to the pairing $(x,x')\mapsto \addchar_F(xx')$. Let $\Abs_F$ be the absolute value of $\cF$ normalized by $d(ax)=\abs{a}_Fdx$ for $a\in F^\x$. We often simply write $\Abs=\Abs_F$ if it is clear from the context without possible confusion. We recall the definition of the local zeta function $\zeta_F(s)$. If $\cF$ is non-archimedean, let $\uf_F$ be a uniformizer of $F$ and let \[\zeta_F(s)=\frac{1}{1-\abs{\uf_F}_F^s}.\] If $F$ is archimedean, then \[\zeta_\R(s)=\pi^{-s/2}\Gamma(s/2);\,\zeta_\C(s)=2(2\pi)^{-s}\Gamma(s).\]
The Haar measures $\dx x$ on $F^\x$ is normalized by
\[\dx x=\zeta_{F}(1)\abs{x}_F^{-1}dx.\]
In particular, if $F=\R$, then $dx$ is the Lebesgue measure and $\dx x=\abs{x}_\R^{-1}dx$, and if $F=\C$, then $dx$ is twice the Lebesgue measure on $\C$ and $\dx x=2\pi^{-1}r^{-1}drd\theta\,\,(x=re^{i\theta})$.

 Suppose that $F$ is non-archimedean. Let $\cO_F$ be the ring of integers of $F$ and let $\cD_F$ be the absolute different of $F$. Then $\cD_F^{-1}$ is the \pont dual of $\cO_F$ with respect to $\addchar_F$, and $\vol(\cO_F,dx)=\abs{\cD_F}_F^\onehalf$. If $\mu:F^\x\to\C^\x$ is a character of $F^\x$, define the local conductor $a(\mu)$ by
\[a(\mu)=\inf\stt{n\in \Z_{\geq 0}\mid
\mu(x)=1\text{ for all }x\in (1+\uf_F^n \cO_F)\cap \cO_F^\x}.\]

\subsection{}
If $L$ is a number field, the ring of integers of
$L$ is denoted by $\cO_L$, $\A_L$ is the adele of $L$ and $\A_{L,f}$
is the finite part of $\A_L$. For $a\in\A_L^\x$, we put
\[\il_L(a):=a(\cO_L\ot\Zhat)\cap L.\]
Denote by $G_L$ the absolute Galois group and by $\rec_L:\A_L^\x\to G^{ab}_{L}$ the geometrically normalized reciprocity law. 
 We define $\addchar_L:\A_L^\x/L\to\C^\x$ by
$\addchar_L(x)=\addchar_\Q\circ\Tr_{L/\Q}(x)$.

Let $v_p$ be the \padic valuation on $\Qbarp$ normalized so that $v_p(p)=1$. We regard $L$ as a subfield in $\C$ (resp. $\C_p$) via
$\iota_\infty:\Qbar\hookto\C$ (resp. $\iota_p=\iota^{-1}\circ\iota_\infty:\Qbar\hookto\Cp$) and $\Hom(L,\Qbar)=\Hom(L,\C_p)$.

Let $\Zbar$ be the ring of algebraic integers of $\Qbar$ and let $\Zbarp$ be the \padic completion of $\Zbar$ in $\Qbarp$. Let $\Zbar$ be the ring of algebraic integers of $\Qbar$ and let $\Zbarp$ be the \padic completion of $\Zbar$ in $\Qbarp$ with the maximal ideal $\frakm_p$. Let $\frakm=\iota_p^{-1}(\frakm_p)$.

\subsection{Local $L$-functions}
Let $F$ be a non-archimedean local filed. Let $\mu,\nu:F^\x\to\C^\x$ be two characters of $F^\x$. Denote by $I(\mu,\nu)$ the space consisting of smooth and $\GL_2(\cO_F)$-finite functions $f:\GL_2(F)\to\C$ such that
\[f(\MX{a}{b}{0}{d}g)=\mu(a)\nu(d)\abs{\frac{a}{d}}^\onehalf f(g).\]
Then $I(\mu,\nu)$ is an admissible representation of $\GL_2(F)$. Denote by $\pi(\mu,\nu)$ the unique infinite dimensional subquotient of $I(\mu,\nu)$.  We call $\pi(\mu,\nu)$ a principal series if $\mu\nu^{-1}\not=\Abs^\pm$ and a special representation if $\mu\nu^{-1}=\Abs^\pm$.

Let $E$ be a quadratic extension of $F$ and let $\chi:E^\x\to\C^\x$ be a character. We recall the definition of local $L$-functions $L(s,\pi_{\cK}\ot\chi)$ (\cite[\S 20]{Jacquet:GLtwoPartII}) when $\pi=\pi(\mu,\nu)$ is a subrepresentation of $I(\mu,\nu)$. If $E=F\oplus F$, then we write $\chi=(\chi_1,\chi_2):F^\x\oplus F^\x\to\C^\x$ and put
\begin{align*}L(s,\pi_{\cK}\ot\chi)=&\begin{cases}L(s,\pi\ot\chi_1)L(s,\pi\ot\chi_2)&\text{ if $\mu\nu^{-1}\not =\Abs^\pm$},\\
L(s,\mu\chi_1)L(s,\mu\chi_2)&\text{ if $\mu\nu^{-1}=\Abs$}.\end{cases}
\intertext{If $E$ is a field, then} L(s,\pi_{E}\ot\chi)=&\begin{cases}L(s,\mu'\chi )L(s,\nu'\chi)&\text{ if $\mu\nu^{-1}\not =\Abs^\pm$},\\
L(s,\mu'\chi )&\text{ if $\mu\nu^{-1}=\Abs$}.\end{cases}\end{align*}
Here $\mu'=\mu\circ\rmN_{E/F},\,\nu'=\nu\circ\rmN_{E/F}$ are characters of $E^\x$.
\subsection{Whittaker and Kirillov models}
Let $F$ be a local field. Let $\pi$ be an irreducible admissible representation of $\GL_2(F)$ and let $\addchar:F\to\C^\x$ be a non-trivial additive character. We let $\cW(\pi,\addchar)$ be the Whittaker model of $\pi$. Recall that $\cW(\pi,\addchar)$ is a subspace of smooth functions $W:\GL_2(F)\to\C$ such that
\begin{mylist}
\item $W(\MX{1}{x}{0}{1}g)=\addchar(x)W(g)$ for all $x\in F$.
\item If $v$ is archimedean, $W(\DII{a}{1})=O(\abs{a}^N)$ for some positive number $N$.
\end{mylist}
(\cf\cite[Thm.\,6.3]{Jacquet_Langlands:GLtwo}). Let $\calK(\pi,\addchar)$ be the Kirillov model of $\pi$. If $F$ is non-archimedean, then $\calK(\pi,\addchar)$ is a subspace of smooth $\C$-valued functions on $F^\x$, containing all \BS functions on $F^\x$. A function in $\calK(\pi,\addchar)$ shall be called a \emph{local Fourier coefficient} of $\pi$. In addition, it is well known that we have the following $\GL_2(F)$-equivariant isomorphism
\beq\label{E:KisoW}\begin{aligned}
\cW(\pi,\addchar)&\isoto \calK(\pi,\addchar)\\
W&\mapsto \xi_W(a):=W(\DII{a}{1}).\end{aligned}\eeq

\section{Waldspurger formula}\label{S:Wald}
Let $\cF$ be a number field and $\cK$ be a quadratic field extension of $\cF$. Let $\AF=\A_\cF$. Let $G=\GL_2{}_{/\cF}$. Let $\pi$ be an irreducible cuspidal automorphic representation of $G(\AF)$ with unitary central character $\om$. Denote by $\cA(\pi)$ the realization of $\pi$ in the space $\cA_0(G)$ of cusp forms on $G(\AF)$. Let $\chi$ be a unitary Hecke character of $\cK^\x$ such that $\chi|_{\AF^\x}=\om^{-1}$. Let $\pi_\cK$ be the quadratic base change of $\pi$ to the quadratic extension $\cK/\cF$. The existence of $\pi_\cK$ is established in \cite{Jacquet:GLtwoPartII}. The goal of this section is to deduce from results in \cite{Waldsupurger:Central_value} a formula (\propref{P:Waldformula.W}) which expresses the central value $L(\onehalf,\pi_\cK\ot\chi)$ in terms of a product of local toric period integrals of Whittaker functions.

Let $\addchar:=\addchar_\cF:\AF/\cF\to\C^\x$ be the standard non-trivial additive character. For a place $v$ of $\cF$, we let $G_v=G(\cF_v)$ and let $\chi_v:\cK_v^\x\to\C^\x$ (resp. $\addchar_v:\cF_v\to\C^\x$) denote the local constituent of $\chi$ (resp. $\addchar$).
\subsection{}\label{SS:notation1}For $x\in\cK$, let $\rmT(x):=x+\ol{x}$ and $\rmN(x)=x\ol{x}$.
Let $\stt{1,\CMP}$ be a basis of $\cK$ over $\cF$.  We let $\iota:\cK\to M_2(\cF)$ be the embedding attached to $\CMP$ given by
\beq\label{E:imb.W}\iota(a\CMP+b)=\MX{a\rmT(\CMP)+b}{-a\rmN(\CMP)}{a}{b}\quad(a,b\in\cF).\eeq
Put
\[\cmJ:=\MX{-1}{\rmT(\CMP)}{0}{1}.\]Then $M_2(\cF)=\iota(\cK)\oplus\iota(\cK)\cmJ$. It is clear that $\cmJ^2=1$ and $\iota(t)\cmJ=\cmJ\iota(\ol{t})$ for all $t\in \cK$.
\subsection{The local bilinear form and toric integral} For each place $v$ of $\cF$, denote by $\pi_v$ (resp. $\addchar_v$) the local constituent of $\pi$ (resp. $\addchar$) at $v$.
Define a $\C$-bilinear form $\bfb_v:\cW(\pi_v,\addchar_v)\x\cW(\pi_v,\addchar_v)\to\C$ by
\begin{align*}\bfb_v(W_{1},W_{2}):=&\sum_{n=-\infty}^\infty\int_{\uf^n\cO_F^\x}W_{1}(\DII{a}{1})W_{2}(\DII{-a}{1})\om^{-1}(a)\dx a\quad (W_1,W_2\in \cW(\pi_v,\addchar_v))\\
=&\int_{\cF_v^\x}W_{1}(\DII{a}{1})W_{2}(\DII{-a}{1})\om^{-1}(a)\dx a.
\end{align*}
It is known that this series converges absolutely as $\pi_v$ is a local constituent of a unitary cuspidal automorphic representation. Moreover, the pairing $\bfb_v$ enjoys the property:
\beq\label{E:bilinear.W}\bfb_v(\pi(g)W_1,\pi(g)W_2)=\om(\det g)\bfb_v(W_1,W_2).\eeq
The pairing $\bfb_v$ thus gives rise to an isomorphism between the contragredient representation $\pi^\vee$ and $\pi\ot\om^{-1}$.

The local toric period integral for $W_1,W_2\in\cW(\pi_v,\addchar_v)$ is given by
\[P(W_{1},W_{2},\chi_v):=\int_{\cK_v^\x/\cF_v^\x}\bfb_v(\pi(\iota(t))W_{1},\pi(\cmJ)W_{2})\chi_v(t)\dt\cdot\frac{L(1,\tau_{\cK_v/\cF_v})}{\zeta_{\cF_v}(1)}.\]
The above integral converges as $\chi_v$ is unitary (\cite[LEMME 7]{Waldsupurger:Central_value}).
\subsection{A formula of Waldspurger}
Let $\Lam(s,\pi_{\cK}\ot\chi)$ be the completed $L$-function of $\pi_{\cK}\ot\chi$ given by
\[\Lam(s,\pi_{\cK}\ot\chi):=\prod_v L(s,\pi_{\cK_v}\ot\chi_v)=L(s,\pi_{\cK}\ot\chi)\cdot\prod_{v|\infty}L(s,\pi_{\cK_v}\ot\chi_v).\]
It is well known that $\Lam(s,\pi_{\cK}\ot\chi)$ converges absolutely for $\Re s\gg 0$ and has meromorphic continuation to all $s\in\C$. Moreover, it satisfies the functional equation
\[\Lam(s,\pi_{\cK}\ot\chi)=\ep(s,\pi_\cK\ot\chi)\Lam(s,\pi_{\cK}\ot\chi).\]
Here we have used the self-duality condition $\chi|_{\AF^\x}=\om^{-1}$. The global toric period integral for $\vp\in\cA(\pi)$ is defined by
\[P_\chi(\vp):=\int_{\cK^\x\AF^\x\bksl \AK^\x}\vp(\iota(t))\chi(t)\dt.\]
The following proposition connects the global toric periods and central $L$-values of $\pi_\cK\ot\chi$.
\begin{prop}[Waldspurger]\label{P:Waldformula.W}Let $\vp_1,\vp_2\in \cA(\pi)$ and let $W_{\vp_1},W_{\vp_2}$ be the associated global Whittaker functions. We suppose that $W_{\vp_i}=\prod_vW_{i,v}$, where  $W_{i,v}\in\cW(\pi_v,\addchar_v)$ such that $W_{i,v}(1)=1$ for almost $v$ $(i=1,2)$. Then there exists a finite set $S_0$ of places of $\cF$ including all archimedean places such that for every finite set $S\supset S_0$, we have
\[P_\chi(\vp_1)P_\chi(\vp_2)=\Lam(\onehalf,\pi_\cK\ot\chi)\cdot\prod_{v\in S}\frac{1}{L(\onehalf,\pi_{\cK_v}\ot\chi_v)}\cdot P(W_{1,v},W_{2,v},\chi_v).
\]
\end{prop}
\begin{proof}The proof is the combination of various formulae established in \cite{Waldsupurger:Central_value}. We first recall some important local integrals. Let $D=G\x G$. For each place $v$ of $\cF$, let $\cS_v=\cS(M_2(\cF_v))\ot\cS(\cF_v^\x)$ and let $D_v=G_v\x G_v$.
Let $r=r'\x r'':G_v\x D_v\to \End\cS_v$ be the Weil representation of $G_v\x D_v$ defined in \cite[\S I.3 p.178]{Waldsupurger:Central_value}

Let $\vp\in\cA(\pi)$ be an automorphic form in the automorphic realization of $\pi$. Recall that the global Whittaker function of $\vp$ is defined by
\[W_\vp(g)=\int_{\cF\bksl\A_F}\vp(\MX{1}{x}{0}{1}g)\addchar(-x)dx.\]
Write $\cW_v=\cW(\pi_v,\addchar_v)$. We further assume that $W_\vp$ has the factorization $W_\vp=\prod _vW_{\vp,v}\in \ot_v'\cW_v$ such that $W_{\vp,v}(1)=1$ for almost $v$. For each $v$, let $U:\cS_v\to \cW_v\ot\cW_v,\,f_v\to U_{f_v}$ be the $G_v\x G_v$-equivaraint surjective morphism associated to $W_v$ introduced in \cite[COROLLAIRE, p.187]{Waldsupurger:Central_value}. Define the following local integrals:
\begin{align*}C(f_v):=&\int_{\cF_v^\x}U_{f_v}(\DII{a}{1},\DII{-a}{1}))\om^{-1}(a)\dx a,\\
B(f_v,1):=&\int_{Z_v\bksl G_vN_v}\int_{\cF_v^\x}W_{\vp,v}(\sg)r'(\sg)f_v(x,x^{-2})dxd\sg,\\
P(f_v,\chi_v,\onehalf):=&\int_{\cF^\x_v\bksl \cK^\x_v}B(r''(\iota(t),1)f_v,1)\chi_v(t)dt.
\end{align*}
The convergence and analytic properties of these local integrals are studied in \cite[LEMME 2, LEMME 3, LEMME 5]{Waldsupurger:Central_value}.
Moreover, we have
\[B(f_v,1)=C(f_v)\cdot\frac{1}{\zeta_{\cF_v}(1)}.\]
For each $v$, we take a special test function $f_v\in\cS_v$  such that
\beq\label{E:1.W}U_{f_v}=W_{1,v}\ot \pi(\cmJ)W_{2,v}.\eeq
Note that $f_v$ can be chosen to be the spherical test function $f_v^0:=\bbI_{M_2(\cO_{\cF_v})}\ot\bbI_{\cO_{\cF_v}^\x}$ for all but finitely many $v$.
With this particular choice of $f_v$, we have
\beq\label{E:2.W}\begin{aligned}
P(f_v,\chi_v,\onehalf)&=\int_{\cF^\x_v\bksl \cK^\x_v}C(r''(\iota(t),1)f_v)dt\cdot\frac{1}{\zeta_{\cF_v}(1)}\\
&=\int_{\cF^\x_v\bksl \cK^\x_v}\bfb_v(\pi(\iota(t))W_1,\pi(\cmJ)W_2)\chi_v(t)dt\cdot\frac{1}{\zeta_{\cF_v}(1)}\\
&=P(W_{1,v},W_{2,v},\chi_v)\cdot \frac{1}{L(1,\tau_{\cK_v/\cF_v})}.
\end{aligned}\eeq

Let $\cS=\ot \cS_v$ be the restricted product with respect to spherical test functions $\stt{f_v^0}_v$.
Define the theta kernel for $f:=\ot f_v\in \cS$ by
\[\theta_f(\sg,g):=\sum_{(x,u)\in M_2(\cF)\x\cF^\x}r(\sg,g)f(x,u),\,\sg\in G(\AF)\quad(g\in D(\AF)=G(\AF)\x G(\AF)),\]
and define the automorphic form $\theta(f,\vp,g)$ on $G(\AF)\x G(\AF)$ by
\[\theta(f,\vp,g)=\int_{G(\cF)\bksl G(\AF)}\vp(\sg)\theta_f(\sg,g)d\sg.\]
Note that according to \eqref{E:1.W}, we have \[\theta(f,\vp,g_1,g_2)=\vp_1(g_1)\vp_2(g_2\cmJ).\]
We define the toric period integral $P(f,\chi)$ by
\[P(f,\chi):=\int_{[\cK^\x\AF^\x\bksl \AK^\x]^2}\theta(f,\vp,\iota(t_1),\iota(t_2))\chi(t_1)\chi(\ol{t}_2)dt_1dt_2.\]
By the relation $\cmJ \iota(t_2)\cmJ=\iota(\ol{t}_2)$ and the automorphy of $\vp_2$, we find that
\[P(f,\chi)=P_\chi(\vp_1)P_\chi(\vp_2).\]
Let $S_0$ be a finite set of places of $\cF$ such that $W_{\vp,v}$, $W_{i,v}$ and $f_v$ are spherical for all $v\not\in S_0$.
From \cite[Prop.\,4, p.196 and LEMME 7, p.219]{Waldsupurger:Central_value}, we deduce the following formula for every finite set $S\supset S_0$:
\begin{align*}P_\chi(\vp_1)P_\chi(\vp_2)=\Lam(\onehalf,\pi_\cK\ot\chi)\cdot \prod_{v\in S}  P(f_v,\chi_v,\onehalf)\cdot \frac{L(1,\tau_{\cK_v/\cF_v})}{L(\onehalf,\pi_{\cK_v}\ot\chi)}.\end{align*}
We thus establish the desired formula in virtue of \eqref{E:2.W}.
\end{proof}

\section{Toric period integrals}\label{S:toric}
\subsection{Notation}
Throughout we suppose that $\cF$ is a totally real number field and $\cK$ is a totally imaginary quadratic extension of $\cF$. We retain the notation in the introduction and \subsecref{SS:notation1}. Let $\Sg$ be a fixed CM type of $\cK$. Let $\pi$ be an irreducible automorphic cuspidal representation of $\GL_2(\AF)$. Let $\frakn$ be the conductor of $\pi$. Suppose that $\pi$ has infinity type $k=\sum_{\sg\in\Sg}k_\sg\sg\in\Z_{\geq 1}[\Sg]$. Let $m=\sum_\sg m_\sg\sg\in\Z_{\geq 0}[\Sg]$ and let $\chi$ be a Hecke character of infinity type $(k/2+m,-k/2-m)$ such that $\chi|_{\AF^\x}=\om^{-1}$. Let $\bdh$ be the set of finite places of $\cF$. Recall that the set of infinite places of $\cF$ is identified with the CM-type $\Sg$..

In this section, we will choose a special local Whittaker function at each place $v$ of $\cF$ in \subsecref{SS:Whittaker.W} and calculate their associated local toric period integrals in \subsecref{SS:toricintegralI.W} and \subsecref{SS:toricII.W}. Finally, we prove in \subsecref{SS:localFourier} a non-vanishing modulo $p$ result of these local Whittaker functions. This result plays an important role in the later application to the calculation of the $\mu$-invariant.

Let $\frakC_\chi$ (resp. $\frakc_\om$) be the conductor of $\chi$ (resp. $\om$).  Let $\frakc_\chi=\frakC_\chi\cap\cF$. We further decompose $\frakn^-=\frakn^-_s\frakn^-_r$, where $\frakn^-_s$ is prime to $\frakc_\om$ and $\frakn^-_r$ is only divisible by prime factors of $\frakc_\om$. Put \beq\label{E:alv.W}\begin{aligned}c_v(\chi)=&\inf\stt{n\in\Z_{\geq 0}\mid \chi=1\text{ on }(1+\uf^n\cO_E)^\x},\\
m_v(\chi,\pi)=&c_v(\chi)-v(\frakn^-).\end{aligned}\eeq
It is clear that $c_v(\chi)=v(\frakc_\chi)$. We put
\begin{align*}
A(\chi)=&\stt{v\in\bdh\mid\text{$\cK_v$ is a field, $\pi_v$ is special and $c_v(\chi)=0$}}.
\end{align*}

Let $p>2$ be a rational prime satisfying \eqref{ord}. The assumption \eqref{ord} in particular implies that every prime factor of $p$ in $\cF$ splits in $\cK$. Let $\Sg_p$ be the \padic places induced by $\Sg$ via $\iota_p$. Thus $\Sg_p$ and its complex conjugation $\Sgbar_p$ give a partition of the places of $\cK$ above $p$. Let $\frakN$ be the prime-to-$p$ conductor of $\pi_{\cK}\ot\chi$. We fix a decomposition $\frakN^+=\Csplit\ol{\Csplit}$ such that $(\Csplit,\ol{\Csplit})=1$.

\subsection{Galois representation attached to $\pi$}\label{SS:Galois.W}
Let $\rho_p(\pi):G_F\to\GL_2(\cO_{{L_\pi}})$ be the \padic Galois representation associated to $\pi$ as in the introduction. Let $v\ndivides p$ and let $W_{\cF_v}$ be the local Weil group at $v$. Suppose that $\pi_v=\pi(\mu_v,\nu_v)$ is a subquotient of the induced representations. 
By the local-global compatibility (\cite{Carayol:GaloisHMF}, \cite{Taylor:Galois_rep_HMF} and \cite{Jarvis:wtone}), we have
\beq\label{E:Galois.W}\rho_p(\pi)|_{W_{\cF_v}}\iso\MX{\mu_v^{-1}\Abs^\frac{1-k_{mx}}{2}}{*}{0}{\nu_v^{-1}\Abs^\frac{1-k_{mx}}{2}}\quad(k_{mx}=\max_{\sg}k_\sg).\eeq
In particular, this implies that $\mu_v(\uf_{\cF_v})$ and $\nu_v(\uf_{\cF_v})$ are \padic units in $\cO^\x_{L_\pi}$.

\subsection{Open-compact subgroups}
For each finite place $v$, we put
\begin{align*}K^0_v=&\stt{g=\MX{a}{b}{c}{d}\in G_v\mid a,d\in \cO_{\cF_v},\,b\in\cD_{F_v}^{-1},\,c\in\cD_{\cF_v},\,\det g\in\cO_{\cF_v}^\x},
\intertext{and for an integral ideal $\fraka$ of $\cF$, we put}
K^0_v(\fraka)=&\stt{g=\MX{a}{b}{c}{d}\in K^0_v\mid c\in\fraka\cD_{\cF_v},\,a-1\in\fraka},\\
U_v(\fraka)=&\stt{g\in \GL_2(\cO_{\cF_v})\mid g\con 1\pmod{\fraka}}.
\end{align*}
Let $\opcpt^0=\prod_{v\in\bdh}\opcpt^0_v$ and $U(\fraka)=\prod_{v\in\bdh} U_v(\fraka)$ be open-compact subgroups of $\GL_2(\AFf)$.
\subsection{The choices of $\CMP$ and $\cmpt_v$}\label{SS:choiceofcmpt}
 Fix an integral ideal $\Bad\subset \frakc_\chi\frakn \cD_{\cK}^2$ of $\cF$. For each finite place, let $d_{\cF_v}$ be a generator of the absolute different $\cD_{F_v}$. We choose $\CMP\in\cK$ such that
\begin{itemize}
\item[(d1)] $\Im \sg(\CMP)>0$ for all $\sg\in\Sg$,
\item[(d2)]$\stt{1,d_{\cF_v}^{-1}\CMP}$ is an $\cO_{\cF_v}$-basis of $\cO_{\cK_v}$ for all $v\mid p\Bad$,
\item[(d3)]$d_{\cF_v}^{-1}\CMP$ is a uniformizer of $\cK_v$ for every $v$ ramified in $\cK$.\end{itemize}
Then $\CMP$ is a generator of $\cK$ over $\cF$ which determines an embedding $\cK\hookto M_2(\cF)$ in \eqref{E:imb.W}. Let
\[\delta=2^{-1}(\CMP-\ol{\CMP})\in \cK^\x.\]

For each $v$ split in $\cK$, we shall fix a place $w$ of $\cK$ above $v$ throughout, and decompose $\cK_v:=\cK\ot_\cF\cF_v=\cF_ve_w\oplus \cF_v e_{\wbar}$, where $e_w$ and $e_{\wbar}$ are the idempotents attached to $w$ and $\wbar$. If $v|p\frakN^+$, we further require that $w|\Csplit\Sgbar_p$, \ie $w|\Csplit$ or $w\in\Sgbar_p$. We identify $\delta\in \cK_w=\cF_v$ and write $\CMP_v=\CMP_we_w+\CMP_{\wbar}e_{\wbar}$ for split $v$.

For each finite place $v$, we fix a uniformizer $\uf_v=\uf_{\cF_v}$ of $\cF_v$. By (d2), we fix a choice of $d_{\cF_v}$ as follows.
\[d_{\cF_v}=\begin{cases}2\delta &\text{ if } v\mid p\Bad\text{ is split },\\
\uf_v^{v(\cD_{\cF})}&\text{ otherwise }.\end{cases}\] We also fix an $\cO_{\cF_v}$-basis $\stt{1,\OKbasis_v}$ of $\cO_{\cK_v}$ such that $\OKbasis_v=\CMP$ except for finitely many $v$ and \[\OKbasis_v=d_{\cF_v}^{-1}\CMP\text{ for }v|p\Bad.\] Write $\OKbasis_v=a_v\CMP+b_v$ with $a_v,b_v\in\cF_v$.

For each place $v$, we define $\cmptv\in \GL_2(\cF_v)$ as follows:
\beq\label{E:cmptv.W}\begin{aligned}
\cmptv=&\MX{\Im\sg(\CMP)}{\Re \sg(\CMP)}{0}{1}\text{ for }v=\sg\in\Sg,\\
\cmptv=&(\CMP_w-\CMP_{\wbar})^{-1}\MX{d_{\cF_v}\CMP_w}{\CMP_{\wbar}}{d_{\cF_v}}{1}\text{ for split }v=w\wbar,\\
\cmptv=&\MX{d_{\cF_v}}{-b_v}{0}{a_v}\text{ for non-split finite $v$.}
\end{aligned}\eeq 
For $t\in \cK_v$, we put
\[\iota_{\cmptv}(t):=\cmptv^{-1}\iota(t)\cmptv.\]
It is straightforward to verify that if $v=\sg\in\Sg$ is archimedean and $t=x+iy\in\C^\x$, then
\begin{align}
\label{E:cm1.W}\iota_{\cmpt_\sg}(t)=&\MX{x}{-y}{y}{x},\,\intertext{ and if $v=w\wbar$ is split and $t=t_1e_w+t_2e_{\wbar}$, then }
\label{E:cm2.W}\iota_{\cmptv}(t)=&\DII{t_1}{t_2}.
\end{align}
Moreover, we note that for all finite places $v$\[\iota_{\cmptv}(\cO_{\cK_v}^\x)=\iota_{\cmptv}(\cK_v^\x)\cap K^0_v.\]

\subsection{Running assumptions}
In this section, we will assume \hypref{H:HypA} for $(\pi,\chi)$ and
\beqcd{sf} \frakn^-\text{ is square-free}.\eeqcd
The assumption \eqref{sf} implies that $\pi_v$ is an unramified special representation if $v|\frakn^-_s$ and $\pi_v$ is a ramified principal series if $v|\frakn^-_r$. In particular, for every place $v$ inert or ramified in $\cK$, $\pi_v$ is a sub-quotient of induced representations and the local $L$-function $L(s,\pi_v)\not =1$. We shall write $\pi_v=\pi(\mu_v,\nu_v)$ such that $L(s,\pi_v)=L(s,\mu_v)$ for $v|\frakn^-$. Moreover, by the local root number formulas \cite[Prop.\,3.5,Thm.\,2.18]{Jacquet_Langlands:GLtwo}, under the assumption \eqref{sf} \hypref{H:HypA} on the sign of local root numbers is equivalent to the following condition:
\beqcd{R1}
\begin{aligned}\text{ For each }v\in A(\chi),\,v\text{ is ramified in $\cK$ and }\mu_v'\chi_v(\uf_{\cK_v})=-\abs{\uf}^\onehalf\quad(\mu'_v=\mu_v\circ\rmN_{\cK_v/\cF_v}).\end{aligned}\eeqcd

\def\Fv{\cF}
\def\Ev{\cK}
\bigskip
In what follows, we fix a place $v$ of $\cF$. Let $F=\cF_v$ and $E=\cK_v$. Let $\cO=\cO_F$ and $\uf=\uf_v$ if $v$ is finite. We shall suppress the subscript $v$ and write $\pi=\pi_v$, $\chi=\chi_v$, $\cmpt=\cmptv$ and $\addchar=\addchar_v$.

\subsection{The choice of local toric Whittaker functions}\label{SS:Whittaker.W}
If $v$ is finite, we let $W^0_v$ denote the new Whittaker function in $\cW(\pi,\addchar)$. In other words, $W^0_v$ is invariant by $K^0_v(\frakn)$ and $W^0_v(1)=1$. The existence of $W^0_v$ is a consequence of the theory of local new vectors \cite{Casselman:Atkin-Lehner}. Now we introduce special local Whittaker functions.

\subsubsection{The archimedean case} Suppose that $v=\sg\in\Sg$ is an archimedean place and $\Fv=\R$. Then $\pi_\sg=\pi(\Abs^\frac{k_\sg-1}{2},\Abs^\frac{1-k_\sg}{2}\sgn^{k_\sg})$ is the discrete series of minimal $\SO(2,\R)$-type $k_\sg$. Let $W_{k_\sg}\in\cW(\pi_v,\psi_v)$ be the Whittaker function given by
\beq\label{E:archiWhittaker.W}W_{k_\sg}(z\DII{a}{1}\kappa_\theta)=a^\frac{k_\sg}{2}e^{-2\pi a}\bbI_{\R_+}(a)\cdot e^{ik_\sg\theta}\sgn(z)^{k_\sg}\quad (z\in\R^\x,\,\kappa_\theta=\MX{\cos\theta}{\sin\theta}{-\sin\theta}{\cos\theta}).\eeq
Let $V_+$ and $V_-$ be the weight raising and lowering differential operators in \cite[p.165]{Jacquet_Langlands:GLtwo} given by
\[V_\pm=\MX{1}{0}{0}{-1}\ot 1\pm\MX{0}{1}{1}{0}\ot i\in\Lie(\GL_2(\R))\ot_\R\C.\]
Define the normalized weight raising differential operator $\wtd V_+$ by
\beq\label{E:Shi_Maass}\wtd V_+=\frac{1}{(-8\pi)}\cdot V_+.\eeq
Then we have
\beq\label{E:15.W}\wtd V_+^{m_\sg} W_{k_\sg}(g\kappa_\theta)=\wtd V_+^{m_\sg} W_{k_\sg}(g)e^{i(k_\sg+2m_\sg)\theta}.\eeq

\subsubsection{The split case} Suppose that $v=w\wbar$ is split with $w|\Sgbar_p\Csplit$ if $v|p\frakN^+$. We introduce some smooth functions $\localK{\chi,v}$ on $F^\x$ in the Kirillov model $\calK(\pi,\addchar)$. Write $\chi=(\chi_w,\chi_{\wbar}):F^\x\oplus F^\x\to\C^\x$. If the local $L$-function $L(s,\pi\ot\chi_w)=1$, we simply put \[\localK{\chi,v}(a)=\bbI_{\cO^\x}(a)\chi_w(a^{-1}).\]
Suppose that $L(s,\pi\ot\chi_w)\not =1$. Then $\pi=\pi(\mu,\nu)$ is a principal series or $\pi=\pi(\mu,\nu)$ is special with $\mu\nu^{-1}=\Abs$ and $\mu\chi_w$ is unramified. If $\pi\ot\chi_w$ is unramified, we set
\[\localK{\chi,v}(a)=\bbI_{\cO}(a)\cdot \chi_w^{-1}\Abs^\onehalf(a)\sum_{i+j=v(a),\,i,j\geq 0}\mu\chi_w(\uf^i)\nu\chi_w(\uf^j).\]
If $\mu_i\chi_w$ is unramified and $\mu_j\chi_w$ is ramified for $\stt{\mu_1,\mu_2}=\stt{\mu,\nu}$, we set
\[\localK{\chi,v}(a)=\mu_i\Abs^\onehalf(a)\bbI_{\cO}(a).\]
If $\pi$ is special, we set
\[\localK{\chi,v}(a)=\mu\Abs^\onehalf(a)\bbI_{\cO}(a).\]

These functions $\localK{\chi,v}$ indeed belong to the Kirillov model $\calK(\pi,\addchar)$ in virtue of the description of the Kirillov models \cite[Lemma 14.3]{Jacquet:GLtwoPartII}.
For each $\xi\in \calK(\pi,\addchar)$, by the isomorphism \eqref{E:KisoW} we denote by $W_\xi\in \cW(\pi,\addchar)$ the unique Whittaker function such that $W_\xi(\DII{a}{1})=\xi(a)$.
We put
\[\localW{\chi,v}:=W_{\localK{\chi,v}}.\]
It follows from the choice of $\localW{\chi,v}$ that
\[\localW{\chi\phi,v}=\localW{\chi,v}\text{ if $\phi: E^\x\to\C^\x$ is unramified}.\]
Recall that the zeta integral $\Psi(s,W,\chi_w)$ for $W\in \cW(\pi,\addchar)$ is defined by
\[\Psi(s,W,\chi_w):=\int_{F^\x}W(\DII{a}{1})\chi_w(a)\abs{a}^{s-\onehalf}\dx a.\]
Then the zeta integral for $\localW{\chi,v}$ satisfies the following equation:
\beq\label{E:DeflocalWsplit.W}\Psi(s,\localW{\chi,v},\chi_w)=L(s,\pi\ot\chi_w)\abs{\cD_F}^\onehalf\quad(\vol(\cO_F^\x,\dx a)=\abs{\cD_F}^\onehalf).\eeq

Suppose that $v=w\wbar$ with $w\in\Sgbar_p$. We define some $p$-modified Whittaker functions as follows. For each $u\in\cO_F^\x$, we put
\[\bfa_{u,v}(a):=\bbI_{u(1+\uf\cO)}(a)\chi_w(a^{-1})\text{ and }\localW{\chi,u,v}=W_{\localK{u,v}}.\]
Let $\localK{\chi,v}^\flat(a):=\bbI_{\cO^\x}(a)\chi_w(a^{-1})$ and let $\localW{\chi,v}^\flat$ be the $p$-modified Whittaker function given by
\beq\label{E:DeflocalWp.W}
\localW{\chi,v}^\flat:=W_{\localK{\chi,v}^\flat}=\sum_{u\in\cU_v}\localW{\chi,u,v},\eeq
where $\cU_v$ is the torsion subgroup of $\cO^\x$.
It is easy to verify that
\beq\label{E:10.W}\begin{aligned}\Psi(s,\localW{\chi,v}^\flat,\chi_w)=1\,;\,\pi(\MX{a}{b}{0}{d})\localW{\chi,v}^\flat=&\chi_w^{-1}(a)\chi_{\wbar}^{-1}(d)\localW{\chi,v}^\flat\text{ for }a,d\in\cO^\x,\,b\in\cD_F^{-1}.
\end{aligned}\eeq

\subsubsection{The inert and ramified case} Suppose that $v$ is an inert or ramified finite place. Then $E$ is a non-archimedean local field. Define the operators $\LR_v$ and $\cP_{\chi,\cmpt}$ on $W\in\cW(\pi,\psi)$ by
\begin{align*}\LR_v W(g):=&W(g\DII{1}{\uf}),\\
\cP_{\chi,\cmpt} W(g):=& \bfv_E^{-1}\cdot \int_{\Kv^\x/\Fv^\x}\pi(\iota_\cmpt(t))W(g)\chi(t)\dt.\\
=&\bfv_E^{-1}\int_{\Kv^\x/\Fv^\x}W(g\cmpt^{-1}\iota(t)\cmpt)\chi(t)\dt.\end{align*}
Note that
\[\bfv_E=\vol(\Kv^\x/\Fv^\x,\dt)=e_v\cdot \abs{\cD_E}_E^{\onehalf}\abs{\cD_F}^{-\onehalf},\quad e_v=\begin{cases}1&\text{ if $v$ is inert},\\
2&\text{ if $v$ is ramified}.\end{cases}
\]
We define the Whittaker function $\localW{\chi,v}$ by
\beq\label{E:DeflocalWinert.W}\localW{\chi,v}:=\cP_{\chi,\cmpt} \LR_v^{m_v(\chi,\pi)} W^0_v.
\eeq
\subsubsection{}
Define the subgroup $\cT_v$ of $E^\x$ by
\[\cT_v=\begin{cases}\cO^\x_{E}F^\x&\text{ if $v$ is split,}\\
E^\x&\text{ if $v$ is non-split.}\end{cases}\]
Then $\cT_v=\stt{x\in E\mid x/\ol{x}\in\cO_E^\x}$ if $v$ is finite.
\begin{defn}[Toric Whittaker functions]\label{D:toric} We say that $W\in \cW(\pi,\addchar)$ is a toric Whittaker function of character $\chi$ if
\[\pi(\iota_\cmpt(t))W=\chi^{-1}(t)\cdot W\text{ for all }t\in\cT_v.\]
\end{defn}
\begin{lm}\label{L:toric.W}The Whittaker functions $\localW{\chi,v}$ chosen as above are toric. To be precise, we have \begin{mylist}
\item $\wtd V_+^{m_\sg} W_{k_\sg}$ is a toric Whittaker function of the character $\chi_{\sg}:\C^\x\to\C^\x,\,z\mapsto z^{k_\sg+m_\sg}\ol{z}^{-m_\sg}\abs{z\ol{z}}^{-k_\sg/2}$.
\item If $v$ is finite, then $\localW{\chi,v}$ are toric Whittaker functions of character $\chi_v$.
\item If $v|p$, then  $\localW{\chi,v}^\flat$ is toric, and for $u\in\cO_F^\x$
\[\pi(\iota_\cmpt(t))\localW{\chi,u,v}=\chi^{-1}(t)\localW{\chi,u.t^{1-c},v},\]
where $u.t^{1-c}:=ut_{\wbar} t_{w}^{-1}$, $t=t_we_w+t_{\wbar} e_{\wbar}\in\cO_E^\x$ with $w\in\Sgbar_p$,
\end{mylist}
\end{lm}
\begin{proof} It follows immediately from the definitions of these Whittaker functions together with \eqref{E:15.W}, \eqref{E:cm1.W} and \eqref{E:cm2.W}.\end{proof}
\subsection{Local toric period integrals (I)}\label{SS:toricintegralI.W}
\subsubsection{}
Define the local toric period integral for $W\in \cW(\pi,\addchar)$ by
\begin{align*}\localP{W}{\chi}:=&P(W,W,\chi)\\
=&\int_{E^\x/F^\x}\bfb_v(\pi(\iota(t))W,\pi(\cmJ)W)\chi(t)\dt\cdot\frac{L(1,\tau_{E/F})}{\zeta_F(1)}.\end{align*}
The main task of this section is to evaluate $\localP{\pi(\cmpt)\localW{\chi,v}}{\chi}$. Put
\[\bfd(a)=\DII{a}{1}\quad(a\in F^\x).\]
We first treat the archimedean and split cases.
\subsubsection{The archimedean case}\label{SS:local_archi}
Suppose $v=\sg\in\Sg\isoto\Hom(\cF,\R)$ is an archimedean place.
\begin{prop}\label{P:Architoric.W}We have
\[\localP{\pi(\cmpt) \wtd V_+^{m_\sg}W_{k_\sg}}{\chi}=2^3\cdot \frac{\Gamma(m_\sg+1)\Gamma(k_\sg+m_\sg)}{(4\pi)^{k_\sg+1+2m_\sg}}.\]
\end{prop}
\begin{proof}Introduce the Hermitian inner product on $\cW(\pi,\addchar)$ defined by
\[\pair{W_1}{W_2}:=\bfb_v(W_1,c(W_2)),\text{ where }c(W_2)(g):=\ol{W(\DII{-1}{1}g)}\om(\det g).\]
Write $k=k_\sg$ and $m=m_\sg$. It is clear that
\[\pair{W_k}{W_k}=(4\pi)^{-k}\Gamma(k).\]
Since $c(V_+^mW_k)$ and $\pi(\DII{-1}{1})V_+^mW_k$ are both nonzero Whittaker functions of weight $-k-2m$, there exists some constant $\gamma$ such that
\[\pi(\DII{-1}{1})V_+^mW_k=\gamma\cdot c(V_+^mW_k)\iff V_+^mW_k(\bfd(a))=\gamma\cdot \ol{V_+^mW_k(\bfd(a))}\text{ for all }a\in\R_+.\]
Let $h_m(x):=V_+^mW_k(\bfd(x))$. Then $h_0(x)=W_k(\bfd(x))$ is a real-valued function in view of the definition \eqref{E:archiWhittaker.W}. A simple calculation shows that
\[h_{m+1}=2x\frac{dh_m}{dx}+(2\pi x-k-2m)h_m,\]
so by induction $h_m(x)$ takes value in $\R$ (\cf\cite[p.189]{Jacquet_Langlands:GLtwo}). This implies that $\gamma=1$. We thus have
\[\bfb_v(\pi(\cmpt)V_+^mW_k,\pi(\cmJ\cmpt)V_+^mW_k)=\pair{V_+^mW_k}{V_+^mW_k}\quad(\cmpt^{-1}\cmJ\cmpt=\DII{-1}{1}).\]

To evaluate $\pair{V_+^mW_k}{V_+^mW_k}$, note that by \cite[p.166]{Jacquet_Langlands:GLtwo} we have
\beq\label{E:6.N}\begin{aligned}V_-^mV_{+}^{m}W_{\wt}&=(-4)^m\frac{\Gamma(k+m)\Gamma(m+1)}{\Gamma(k)}\cdot W_k,
\end{aligned}\eeq
and hence we find that
\begin{align*}\pair{V_+^mW_k}{V_+^mW_k}=&(-1)^m\pair{W_k}{V_-^mV_+^mW_k}\\
=&4^m\frac{\Gamma(k+m)}{\Gamma(k)}\Gamma(m+1)\pair{W_k}{W_k}\\
=&4^m(4\pi)^{-k}\cdot \Gamma(k+m)\Gamma(m+1).\end{align*}
Recall that $\dt=2\pi^{-1}d\theta$ with$ t=e^{i\theta}$, so $\vol(\C^\x/\R^\x,\dt)=2\pi^{-1}\cdot \pi=2$. Combining these with \lmref{L:toric.W} (1), we find that
\begin{align*}
\localP{\pi(\cmpt)\wtd V_+^mW_k}{\chi}&=2\cdot (-8\pi)^{-2m}\cdot\bfb_v(\pi(\cmpt)V_+^mW_k,\pi(\cmJ\cmpt)V_+^mW_k)\cdot\frac{\zeta_\R(2)}{\zeta_\R(1)}\\
&=2^3(4\pi)^{-2m-1}4^{-m}\cdot\pair{V_+^mW_k}{V_+^mW_k}\\
&=2^3\cdot (4\pi)^{-k-2m-1}\Gamma(k+m)\Gamma(m+1).\qedhere
\end{align*}
\end{proof}
\subsubsection{The split case}\label{SS:toricI.W}
 Suppose that $v=w\wbar$ is a finite place split in $E$. Recall that we have assumed $w|\Sgbar_p\Csplit$ if $v|p\frakN^+ $.
\begin{lm}\label{L:key.W} We have
\[\localP{\pi(\cmpt)W}{\chi}=\Psi(\onehalf,W,\chi_w)^2\cdot \frac{L(\onehalf,\pi\ot\chi_{\wbar})}{L(\onehalf,\pi\ot\chi_w)}\cdot\ep(\onehalf,\pi\ot\chi_w,\addchar)\cdot\om^{-1}\chi_w^{-2}(-d_F)\om(\det\cmpt).\]
\end{lm}
\begin{proof}
Let $\wh W(g):=W(g\MX{0}{1}{-1}{0})\om^{-1}(\det g)$. By \cite[Thm.\,2.18 (iv)]{Jacquet_Langlands:GLtwo}, we have
the local functional equation:
\[\frac{\Psi(1-s,\wh W,\chi_w^{-1})}{L(1-s,\pi^\vee\ot\chi_w^{-1})}=\ep(s,\pi\ot\chi_w,\addchar)\cdot\frac{\Psi(s,W,\chi_w)}{L(s,\pi\ot\chi_w)}.\]
We note that
\[\cmptv^{-1}\cmJ\cmptv=\MX{0}{d_F^{-1}}{d_F}{0}.\]
A straightforward computation shows that \begin{align*}
\localP{\pi(\cmpt)W}{\chi}&=\om(\det\cmpt)\int_{F^\x}\int_{F^\x}W(\bfd(at_1))W(\bfd(-a)\MX{0}{d_F^{-1}}{d_F}{0})\chi_w(t_1)\om^{-1}(a)\dx a \dt_1\\
&=\om(-\det\cmpt\cdot d_F)\int_{F^\x}\int_{F^\x}W(\bfd(t_1))W(\bfd(a d_F^{-2})\MX{0}{1}{-1}{0})\chi_w(t_1)\om^{-1}\chi_w^{-1}(a)\dx a \dt_1\\
&=\om(-\det\cmpt)\om^{-1}\chi_w^{-2}(d_F)\cdot \Psi(\onehalf,W,\chi_w)\Psi(\onehalf,\wh W,\chi_w^{-1})\\
&=\om(\det\cmpt)\om^{-1}\chi_w^{-2}(-d_F)\Psi(\onehalf,W,\chi_w)^2\cdot \ep(\onehalf,\pi\ot\chi_w,\addchar)\cdot \frac{L(\onehalf,\pi^\vee\ot\chi_w^{-1})}{L(\onehalf,\pi\ot\chi_w)}.
\end{align*}
The lemma thus follows.
\end{proof}
\begin{prop}\label{P:toric_split.W}We have
\begin{align*}\frac{1}{L(\onehalf,\pi_E\ot\chi)}\cdot\localP{\pi(\cmpt)\localW{\chi,v}}{\chi}=&\abs{\cD_F}\cdot\begin{cases}\ep(\onehalf,\pi\ot\chi_w,\addchar)\om^{-1}\chi_w^{-2}(-2\delta)
&\text{ if } v\mid\frakN^+,\\
\om(\det\cmpt)&\text{ if } v\ndivides \frakN^+.
\end{cases}
\intertext{If $v=w\wbar$ with $w\in\Sgbar_p$, then }
\frac{1}{L(\onehalf,\pi_E\ot\chi)}\cdot\localP{\pi(\cmpt)\localW{\chi,v}^\flat}{\chi}=&\frac{\ep(\onehalf,\pi\ot\chi_w,\addchar)}{L(\onehalf,\pi\ot\chi_w)^2}\cdot \om^{-1}\chi_w^{-2}(-2\delta)\abs{\cD_F}.\end{align*}
\end{prop}
\begin{proof}
The proposition follows immediately from \lmref{L:key.W}. \eqref{E:DeflocalWsplit.W} and \eqref{E:10.W} combined with our choices of $d_F$ for $v|p\frakN^+$ and the fact that
\[\ep(\onehalf,\pi\ot\chi_w,\addchar)\cdot\om^{-1}\chi_w^{-2}(-d_F)=1 \text{ if }v\ndivides \frakN^+.\qedhere\]
\end{proof}
\subsection{Local toric period integrals (II)}\label{SS:toricII.W}
In this subsection, we treat the case $v$ is inert or ramified. A large part of the computation in this subsection is inspired by \cite{Murase:WaldII}. Let
\begin{align*}\bfw=&\MX{0}{-d_F^{-1}}{d_F}{0}\text{ and put}\\
K^0(\uf):=&\stt{g=\MX{a}{b}{c}{d}\in K_v^0\mid a-1\in\uf\cO,\,c\in\uf\cD_F}.\end{align*}
Let $\OKbasis=\OKbasis_v\in \cO_E$ be the element chosen in \subsecref{SS:choiceofcmpt} and write $W^0$ for the new local Whittaker function $W^0_v$ at $v$. Recall that $\stt{1,\OKbasis}$ is an $\cO$-basis of $\cO_E$ and $\OKbasis$ is a uniformizer if $E/F$ is ramified. 
\subsubsection{}
 We prepare some elementary lemmas.
\begin{lm}\label{L:coset.W}Suppose that $v|\Bad$. Let $m$ be a non-negative integer and let
\begin{align*}B^1(\cO)=&\stt{\MX{1}{x}{0}{d}\mid x\in\cD_F^{-1},\,d\in\cO^\x},\\
N(\cD_F^{-1})=&\stt{\MX{1}{x}{0}{1}\mid x\in\cD_F^{-1}}.
\end{align*}
If $y\in\uf^{m+1}\cO$, then we have \[
\bfd(\uf^{m})\iota_{\cmpt}(1+y\OKbasis)\bfd(\uf^{-m})\in K^0(\uf).\]
If $y\in\uf^r\cO^\x$ and $0\leq r\leq m$, then \begin{align*}
N(\cD_F^{-1})\bfd(\uf^{m})\iota_{\cmpt}(1+y\OKbasis)\bfd(\uf^{-m}) B^1(\cO)=&
N(\cD_F^{-1})\DII{\uf^{m-r}}{y\uf^{-m}}\bfw B^1(\cO).
\intertext{If $y\in \uf\cO$, then  }
N(\cD_F^{-1})\bfd(\uf^{m})\iota_\cmpt(y+\OKbasis)\bfd(\uf^{-m}) B^1(\cO)=&
N(\cD_F^{-1})\DII{\uf^{m+e_v-1}}{\uf^{-m}}\bfw B^1(\cO).
\end{align*}
\end{lm}
\begin{proof} Recall that if $v|\Bad$, then $\OKbasis=d_{F}^{-1}\CMP$, $\cmpt=\DII{d_F}{d_F^{-1}}$, and hence 
\[\iota_{\cmpt}(x+y\OKbasis)=\MX{x+y\rmT(\OKbasis)}{yd_F^{-1}\rmN(\OKbasis)}{yd_F}{x}\quad(x,y\in F).\]
Then the proof is a straightforward calculation, so we omit the details.
\end{proof}

\begin{lm}\label{L:halfsum.W}Suppose that $\chi|_{F^\x}$ is trivial on $1+\uf\cO$. For each non-negative integer $r$, we set
\[X_r:=\int_{\uf^r\cO}\chi(1+y\OKbasis)d'y,\]
where $d'y$ is the Haar measure on $\cO$ such that $\vol(\cO,d'y)=L(1,\tau_{E/F})\abs{\cD_E}_E^{\onehalf}\abs{\cD_F}^{-\onehalf}$. Then $X_r=0$ if $c_v(\chi)>1$ and $0<r<c_v(\chi)$ and $X_{r}=\abs{\uf^r}\cdot L(1,\tau_{E/F})\abs{\cD_E}_E^{\onehalf}\abs{\cD_F}^{-\onehalf}$ if $r\geq c_v(\chi)$.
\end{lm}
\begin{proof}Let $Q_r:=1+\uf^r\cO_E/1+\uf^r\cO$. If $0<r<c_v(\chi)$, then $\chi$ is a non-trivial character on the group $Q_r$. Note that we have a bijection $\uf^r\cO\isoto Q_r,\,y\mapsto 1+y\OKbasis$ and the pull-back of the quotient measure $dt$ on $Q_r$ is $d'y$. Therefore, we have
\[\int_{\uf^r\cO}d'y=\int_{Q_r}\chi(t)dt=\begin{cases} 0&\text{ if }0<r<c_v(\chi)\\
\vol(\uf^r\cO,d'y)&\text{ if }r\geq c_v(\chi).\end{cases}\]
This finishes the proof.
\end{proof}
Define the matrix coefficient $\bfm^0:\GL_2(F)\to\C$ by
\begin{align*}\bfm^0(g):=&\bfb_v(\pi(g)W^0,\pi(\DII{-1}{1})W^0)\\
=&\int_{F^\x}W^0(\bfd(a)g)W^0(\bfd(a))\om^{-1}(a)\dx a.
\end{align*}
Since $W^0$ is invariant by $K^0(\uf)$, $\bfm^0(g)$ only depends on the double coset $K^0(\uf)gK^0(\uf)$ by \eqref{E:bilinear.W}.  Put \[m=m_v(\chi,\pi)=c_v(\chi)-v(\frakn^-)\geq -1.\] We set
\beq\label{E:def.W}\begin{aligned}\Prd^*(\pi(\cmpt)\LR_v^{m}W^0,\chi)&:=\localP{\pi(\cmpt)\LR_v^{m}W^0}{\chi}\cdot \frac{\zeta_F(1)}{L(1,\tau_{E/F})}\om(\uf^{-m}\det\cmpt^{-1})\\
&=\int_{E^\x/F^\x}\bfb_v(\LR_v^{-m}\pi(\iota_\cmpt(t))\LR_v^{m}W^0,\pi(\DII{-1}{1})W^0)\chi(t)\dt\\
&=\int_{E^\x/F^\x}\bfm^0(\bfd(\uf^{m}) \iota_\cmpt(t) \bfd(\uf^{-m}) )\chi(t)\dt.
\end{aligned}
\eeq
Here we have used the fact that $m+v(\rmT(\OKbasis))\geq 0$ in the second equality. It follows immediately from the definition of the projector $\cP_{\chi,\cmpt}$ that
\beq\label{E:7.W}\begin{aligned}\localP{\pi(\cmpt)\localW{\chi,v}}{\chi}=& \localP{\pi(\cmpt)\cP_{\chi,\cmpt}\LR^m_vW^0}{\chi}\\
=&\Prd^*(\pi(\cmpt)\LR_v^{m} W^0,\chi)\cdot\frac{\om(\uf^{m}\det\cmpt)L(1,\tau_{E/F})}{\zeta_F(1)}.\end{aligned}\eeq

Using the decomposition \[E^\x=F^\x(1+\cO\OKbasis)\disjoint F^\x(\uf\cO+\OKbasis)\]
and \lmref{L:coset.W}, we find that
\beq\label{E:long.W}
\begin{aligned}
\Prd^*(\pi(\cmpt)\LR_v^{m}W^0,\chi)=&\int_{\cO}\chi(1+y\OKbasis)\bfm^0(\bfd(\uf^m)\iota_\cmpt(1+y\OKbasis)\bfd(\uf^{-m}) )d'y\\
&+\int_{\uf\cO}\chi(x+\OKbasis)\bfm^0(\bfd(\uf^m)\iota_\cmpt (y+\OKbasis)\bfd(\uf^{-m}) )\abs{y+\OKbasis}_E^{-1}d'y\\
=&X_{m+1}\cdot\bfm^0(1)+\sum_{r=0}^{m}\int_{\uf^r\cO^\x}\chi(1+y\OKbasis)\om(\uf^{-m}y)d'y\cdot \bfm^0(\bfd(\uf^{2(m-r)})\bfw)\\
&+Y_0\cdot \om(\uf^{-m})\bfm^0(\bfd(\uf^{2m+e_v-1})\bfw),
\end{aligned}
\eeq
where
\[Y_0:=\int_{\uf\cO}\chi(y+\OKbasis)d'y\cdot\abs{\uf}^{1-e_v}.\]
In what follows, we use \lmref{L:halfsum.W} and \eqref{E:long.W} to calculate $\localP{\pi(\cmpt)\localW{\chi,v}}{\chi}$.
\subsubsection{The case $v\ndivide \frakn^-_r$}
Suppose that $v\ndivide \frakn^-_r$, \ie the central character $\om$ is unramified. Then \eqref{sf} implies that $\pi$ is either an unramified principal series or an unramified special representation.
\begin{prop}\label{P:toric_inert_un.W}Suppose that $\pi$ is an unramified principal series. Then
\begin{align*}\frac{1}{L(\onehalf,\pi_E\ot\chi)}\cdot\localP{\pi(\cmpt)\localW{\chi,v}}{\chi}=\om(\uf^{m})\abs{\uf^{c_v(\chi)}}\abs{\cD_E}_E^\onehalf\cdot \begin{cases}\om(\det\cmpt)&\text{ if } c_v(\chi)=0,\\
L(1,\tau_{E/F})^2&\text{ if } c_v(\chi)>0.\end{cases}\end{align*}
\end{prop}
\begin{proof} Since $\pi$ is unramified, $\om$ is unramified and $m=c_v(\chi)$. Write $\pi=\pi(\mu,\nu)$ and let $\al=\mu(\uf)$ and $\beta=\nu(\uf)$.
The matrix coefficient $\bfm^0$ is a spherical function on $\GL_2(F)$ in the sense of \cite[Definition 4.1, p.150]{Cartier:corvallis}, and $\bfm^0(g)$ only depends on the double coset $K^0_vgK^0_v$. By Macdonald formula,
\begin{align}\label{E:MCF1.W}\bfm^0(1)=&\frac{\zeta_F(1)L(1,\Ad\pi)}{\zeta_F(2)}\cdot\abs{\cD_F}^\onehalf
=\frac{(1+\abs{\uf})\zeta_F(1)}{(1-\al\beta^{-1}\abs{\uf})(1-\al^{-1}\beta\abs{\uf})}\cdot\abs{\cD_{\cF}}^\onehalf;\\
\label{E:2W}
\bfm^0(\bfd(\uf))=&\frac{\abs{\uf}^\onehalf}{1+\abs{\uf}}\cdot (\al+\beta)\cdot \bfm^0(1);\\
\label{E:3W}\bfm^0(\bfd(\uf^2))=&\frac{\abs{\uf}}{1+\abs{\uf}}\cdot (\al^2+\beta^2+(1-\abs{\uf})\al\beta)\cdot \bfm^0(1).\end{align}
If $v$ is inert and $m=0$, then
\begin{align*}\om(\det\cmpt^{-1})\localP{\pi(\cmpt)W^0}{\chi}=&\bfm^0(1)\cdot \frac{L(1,\tau_{E/F})}{\zeta_F(1)}\cdot\abs{\cD_E}_E^{\onehalf}\abs{\cD_F}^{-\onehalf}\\
=&\frac{1}{(1-\al\beta^{-1}\abs{\uf})(1-\al^{-1}\beta\abs{\uf})}\cdot\abs{\cD_E}_E^{\onehalf}\\
=&L(\onehalf,\pi_E\ot\chi)\cdot \abs{\cD_E}_E^\onehalf.\end{align*}

Suppose that either $v$ is ramified or $m>0$ (so $v|\Bad$ and $\det\cmpt=1$). Then we deduce from \eqref{E:long.W} that
\beq\label{E:16.W}\begin{aligned}\Prd^*(\pi(\cmpt)\LR_v^{m}W^0,\chi)=&X_m\cdot\bfm^0(1)+\sum_{r=0}^{m-1}(X_r-X_{r+1})\om(\uf^{r-m})\cdot\bfm^0(\bfd(\uf^{2(m-r)}))\\
&+Y_0\cdot \om(\uf^{-m})\bfm^0(\bfd(\uf^{2m+e_v-1})).
\end{aligned}\eeq
If $v$ is ramified and $m=0$, then $X_0=\abs{\cD_E}_E^{\onehalf}\abs{\cD_F}^{-\onehalf}$ and $Y_0=\chi(\uf_E)\abs{\cD_E}_E^{\onehalf}\abs{\cD_F}^{-\onehalf}$. By \eqref{E:16.W}, we find that
\begin{align*}
\localP{\pi(\cmpt)W^0}{\chi}
=&(\bfm^0(1)+\chi(\uf_E)\bfm^0(\bfd(\uf)))\frac{L(1,\tau_{E/F})}{\zeta_F(1)}\abs{\cD_E}_E^{\onehalf}\abs{\cD_F}^{-\onehalf}\\
=&(1+\frac{\al+\beta}{1+\abs{\uf}}\cdot\abs{\uf}^\onehalf\chi(\uf_E))\cdot\bfm^0(1)\cdot\frac{\abs{\cD_E}_E^{\onehalf}\abs{\cD_F}^{-\onehalf}}{\zeta_F(1)}\quad(\text{by \eqref{E:2W}})\\
=&\frac{(1+\chi(\uf_E)\al\abs{\uf}^\onehalf)(1+\chi(\uf_E)\beta\abs{\uf}^\onehalf)}{1+\abs{\uf}}\cdot\bfm^0(1)\cdot\frac{\abs{\cD_E}_E^{\onehalf}\abs{\cD_F}^{-\onehalf}}{\zeta_F(1)}\\
=&\abs{\cD_E}_E^\onehalf\cdot L(\onehalf,\pi_E\ot\chi).
\end{align*}
Suppose that $m>0$. Note that since $\chi|_{\cO^\x}=1$, $Y_0=-X_0$ if $v$ is inert and $Y_0=X_0=0$ if $v$ is ramified. Combining with \lmref{L:halfsum.W}, \eqref{E:3W} and \eqref{E:16.W}, we find that
\begin{align*}
\localP{\pi(\cmpt)\LR_v^mW^0}{\chi}=&X_m\cdot (\bfm^0(1)-\om(\uf^{-1})\bfm^0(\bfd(\uf^2))\cdot\om(\uf^{-m}) \frac{L(1,\tau_{E/F})}{\zeta_F(1)}\\
=&\om(\uf^{m})\abs{\uf^m}\cdot\frac{(1-\al\beta^{-1}\abs{\uf})(1-\al^{-1}\beta\abs{\uf})}{1+\abs{\uf}}\cdot \bfm^0(1)\\
&\times \frac{L(1,\tau_{E/F})^2}{\zeta_F(1)}\abs{\cD_E}_E^{\onehalf}\abs{\cD_F}^{-\onehalf}\\
=&\om(\uf^{m})\abs{\uf^m}\abs{\cD_E}_E^\onehalf\cdot L(1,\tau_{E/F})^2.
\end{align*}
The proposition follows immediately.
\end{proof}

\begin{prop}\label{P:toric_inert_sp.W}Suppose that $\pi$ is an unramified special representation. Then
\begin{align*}\frac{1}{L(\onehalf,\pi_E\ot\chi)}\cdot \localP{\pi(\cmpt)\localW{\chi,v}}{\chi}=&\om(\uf^{m})\abs{\uf^{c_v(\chi)}}\abs{\cD_E}_E^\onehalf\\
&\x \begin{cases}L(1,\tau_{E/F})^2&\text{ if $c_v(\chi)>0$},\\
2&\text{ if $v$ is ramified and $c_v(\chi)=0$.}\end{cases}\end{align*}
\end{prop}
\begin{proof} Suppose that $v|\frakn^-_s$. Then $m=m_v(\chi,\pi)=c_v(\chi)-1$. Recall that $\pi=\pi(\mu,\nu)$ is a special representation with a unramified character $\mu$ and $\mu\nu^{-1}=\Abs$. We have
\begin{align*} W^0(\bfd(a))=&\mu(a)\abs{a}^\onehalf\bbI_{\cO}(a),\\
W^0(\bfd(a)\bfw)=&-\mu(a)\abs{a}^\onehalf\abs{\uf}\bbI_{\uf^{-1}\cO}(a).
\end{align*}
A direct computation shows that
\[\bfm^0(1)=\frac{\abs{\cD_{\cF}}^\onehalf}{1-\abs{\uf}^2}\,;\,\bfm^0(\bfw)=(-\abs{\uf})\cdot \bfm^0(1)\,;\,\bfm^0(\DII{1}{\uf}\bfw)=(-\mu(\uf)\abs{\uf}^{-\onehalf})\cdot\bfm^0(1).\]
If $c_v(\chi)>0$, then it follows from \eqref{E:long.W} and \lmref{L:halfsum.W} that
\begin{align*}\localP{\pi(\cmpt)\LR_v^mW^0}{\chi}=&X_{m+1}\cdot(\bfm^0(1)-\bfm^0(\bfw))\cdot\om(\uf^{m})\frac{L(1,\tau_{E/F})}{\zeta_F(1)}\\
=&\om(\uf^{m})\abs{\uf^{m+1}}\cdot\abs{\cD_E}_E^\onehalf L(1,\tau_{E/F})^2.
\end{align*}
If $c_v(\chi)=0$ ($m=-1$), then $v$ is ramified, $X_0=\abs{\cD_E}_E^{\onehalf}\abs{\cD_F}^{-\onehalf},$, $Y_0=\chi(\uf_E)\abs{\cD_E}_E^{\onehalf}\abs{\cD_F}^{-\onehalf},$ and
\begin{align*}
\localP{\pi(\cmpt)\LR_v^mW^0}{\chi}&=\left(X_0\cdot \bfm^0(1)+Y_0\cdot \bfm^0(\DII{1}{\uf}\bfw)\right)\om(\uf^{-1})\cdot\frac{L(1,\tau_{E/F})}{\zeta_F(1)}\\
&=(1-\mu(\uf)\chi(\uf_E)\abs{\uf}^{-\onehalf})\abs{\cD_E}_E^{\onehalf}\abs{\cD_F}^{-\onehalf}\cdot \bfm^0(1)(1-\abs{\uf})\om(\uf^{-1})\\
&=\frac{2\abs{\cD_E}_E^\onehalf\om(\uf^{-1})}{1+\abs{\uf}}\quad(\text{by (R1)})\\
&=2\abs{\cD_E}_E^\onehalf\om(\uf^{-1})\cdot L(\onehalf,\pi_E\ot\chi).\qedhere
\end{align*}
\end{proof}

\subsubsection{The case $v|\frakn^-_r$}
We consider the case $\pi$ is a ramified principal series. Recall that \eqref{sf} suggests that $\pi=\pi(\mu,\nu)$, where $\mu$ is unramified and $\nu$ is ramified, and the conductor $a(\nu)=a(\om)=1$. Since $\chi|_{F^\x}=\om^{-1}$, we must have $m=c_v(\chi)-1\geq 0$. Let $\bfdelta_v:=\OKbasis-\ol{\OKbasis}$. Let $D_{E/F}$ be the discriminant of $E/F$. We begin with a lemma.
\begin{lm}\label{L:1.W}Suppose that $\chi|_{\cO^\x}\not =1$ and $\chi|_{1+\uf\cO}=1$. Then
\[\int_{\uf^{-m}\cO}\chi(y+\OKbasis)d'y=\chi(\bfdelta_v)\abs{\bfdelta_v}_E^{\onehalf}\cdot\frac{\ep(0,\chi^{-1},\psi_E)}{\ep(-1,\om,\psi)}\cdot L(1,\tau_{E/F})\abs{D_{E/F}}^\onehalf.\]
\end{lm}
\begin{proof}
 By \cite[Prop.\,8.2]{Harris:Theta_dichotomy}, we have
\begin{align*}
\int_F\chi(y+2^{-1}\bfdelta_v)dy:=&\int_F\chi(y+2^{-1}\bfdelta_v)\abs{y+2^{-1}\bfdelta_v}_E^{-s-\onehalf}dy|_{s=-\onehalf}\\
=&\chi(\bfdelta_v)\abs{\bfdelta_v}_E^{\onehalf}\cdot\frac{\ep(0,\chi^{-1},\psi_E)}{\ep(-1,\om,\psi)}.
\end{align*}
By the assumption, for all $r\geq m+1$ we have
\[\int_{\uf^{-r}\cO^\x}\chi(y+\OKbasis)dy=\chi(\uf^{-r})\cdot \int_{\cO^\x}\chi(y)dy=0.\]
Thus \[\int_{\uf^{-m}\cO}\chi(y+\OKbasis)dy=\dirlim_r\int_{\uf^{-r}\cO}\chi(y+\OKbasis)dy=\int_{F}\chi(y+\OKbasis)dy=\int_F\chi(y+2^{-1}\bfdelta_v)dy.\]
The lemma follows from the fact that
\[d'y=L(1,\tau_{E/F})\abs{D_{E/F}}^\onehalf\cdot dy.\qedhere\]
\end{proof}
\begin{prop}\label{P:toric_ramified.W} Then we have
\[\frac{1}{L(\onehalf,\pi_E\ot\chi)}\cdot \localP{\pi(\cmpt)\localW{\chi,v}}{\chi}=\abs{\uf^{c_v(\chi)}}\abs{\cD_E}_E^\onehalf\chi(\bfdelta_v d_F^{-1})\abs{\bfdelta_v}_E^\onehalf\ep(0,\chi,\psi_E)\cdot L(1,\tau_{E/F})^2\cdot n_v^2,\]
where $n_v$ is given by
\beq\label{E:normalized.W}n_v:=\frac{\mu(\uf)\abs{\uf}^{m/2}\abs{\cD_F}^{\frac{1}{4}}}{\ep(0,\om,\psi)}\in\Zbar_\setp^\x.\eeq
\end{prop}
\begin{proof} Note that $L(s,\pi_E\ot\chi)=1$ and the conductor $a(\nu)=1$. A straightforward computation shows that
\begin{align*}W^0(\bfd(a))=&\nu\Abs^\onehalf(a)\bbI_{\cO}(a),\\
W^0(\bfd(a)\bfw)=&\mu\Abs^\onehalf(a)\bbI_{\uf^{-1}\cO}(a)\cdot\frac{\om(d_F)\mu(\uf^2)}{\ep(0,\om,\addchar)}.\end{align*}
\[\bfm^0(1)=0,\,\bfm^0(\bfw)=\frac{\om(d_F)\mu(\uf^2)}{\ep(0,\om,\addchar)\cdot (1-\abs{\uf})}\cdot\abs{\cD_{\cF}}^\onehalf.\]

It is not difficult to show that if $m=c_v(\chi)-1>0$, then
\begin{align*}\int_{\uf^r\cO^\x}\chi(y^{-1}+\OKbasis)d'y&=0 \text{ for }0<r<m\text{ and }\\
\int_{\cO}\chi(y+\OKbasis)d'y&=0,\end{align*}
and that if $v$ is ramified, then
\[Y_0=\int_{\uf\cO}\chi(y+\OKbasis)d'y=0.\]
From the above equations, we find that
\begin{align*}
\Prd^*(\pi(\cmpt)\LR_v^{m}W^0,\chi)
=&X_{m+1}\cdot\bfm^0(1)+\sum_{r=0}^{m}\int_{\uf^r\cO^\x}\chi(y^{-1}+\OKbasis)d'y\cdot\om(\uf^{-m})\bfm^0(\bfd(\uf^{2m-2r})\bfw)\\
&+Y_0\cdot \om(\uf^{-m}) \bfm^0(\bfd(\uf^{2m+e_v-1})\bfw)\\
=&\int_{\uf^{-m}\cO}\chi(y+\OKbasis)d'y\cdot \om(\uf^{-m})\abs{\uf^{2m}}\bfm^0(\bfw).
\end{align*}
By \lmref{L:1.W}, we obtain
\begin{align*}
\localP{\pi(\cmpt)\LR_v^{m}W^0}{\chi}=&\Prd^*(\pi(\cmpt)\LR_v^{m}W^0,\chi)\cdot\om(\uf^m)\frac{L(1,\tau_{E/F})}{\zeta_F(1)}\\
=&\abs{\uf^{2m}}\cdot \chi(\bfdelta_v)\abs{\bfdelta_v}_E^\onehalf\abs{\cD_E}_E^\onehalf\abs{\cD_F}^{-\onehalf}\cdot\frac{\ep(0,\chi^{-1},\psi_E)}{\ep(-1,\om,\psi)}\cdot\frac{\mu(\uf^2)\om(d_F)}{\ep(0,\om,\addchar)(1-\abs{\uf})}\cdot \frac{L(1,\tau_{E/F})^2}{\zeta_F(1)}\\
=&\frac{L(1,\tau_{E/F})^2\mu(\uf^2)\abs{\cD_F}^\onehalf\abs{\uf^{2m+1}}}{\ep(0,\om,\psi)^2 }\cdot \abs{\cD_{E}}_E^\onehalf\chi(\bfdelta_vd_F^{-1})\abs{\bfdelta_v}_E^\onehalf\ep(0,\chi^{-1},\psi_E).
\end{align*}
The last equality follows from
\[\ep(-1,\om,\addchar)=\abs{\uf\cD_F}^{-1}\ep(0,\om,\addchar).\]
The proposition follows.
\end{proof}
\def\chpi{\varPsi_{\pi,\chi,v}}
\subsection{Non-vanishing of the local Fourier coefficients}\label{SS:localFourier}
The function $\localK{\chi,v}:F^\x\to\C$ defined by
\[\localK{\chi,v}(a)=\localW{\chi,v}(\bfd(a))\] is called the \emph{local Fourier coefficient} associated to $\localW{\chi,v}$.
\begin{defn}[Normalized local Fourier coefficients]\label{D:localFourier.W} Let \[B(\chi)=\stt{v\in\bdh\mid\text{$v$ is non-split with $c_v(\chi)>0$}}.\]
For $v\in B(\chi)$, let $n_v$ be defined as in \eqref{E:normalized.W} if $v|\frakn^-_r$ and $n_v=1$ if $v\ndivides \frakn^-_r$. Define the normalized local Fourier coefficient $\localK{\chi,v}^*$ by
\[\localK{\chi,v}^*:= \localK{\chi,v}\cdot\begin{cases}n_v^{-1}L(1,\tau_{\cK_v/\cF_v})^{-1}&\text{ if } v\in B(\chi),\\
1&\text{ if }v\in A(\chi),\\
e_v&\text{ otherwise.}\end{cases}\]
Recall that $e_v=1$ if $v$ is unramified and $e_v=2$ if $v$ is ramified.
\end{defn}
Let $v\ndivides p$ be a finite place. We are going to show the normalized local Fourier coefficients $\localK{\chi,v}^*$ indeed take value in a finite extension of $\Zp$ when regarded as a $\Cp$-valued function via $\iota_p:\C\iso\Cp$ and is not identically zero modulo $\frakm$. This is clear if $v$ is split in view of the definition of $\localK{\chi,v}^*=\localK{\chi,v}$ in \subsecref{SS:Whittaker.W}. We give the formulae of $\localK{\chi,v}$.
\begin{lm}Suppose that $c_v(\chi)=0$. Then
\[\localK{\chi,v}^*(a)=\begin{cases}W^0_v(\bfd(a))&\text{ if $v\ndivides \frakn$ is unramified,}\\
W^0_v(\bfd(a))+W^0_v(\bfd(a\uf_v))\chi(\uf_{E_v})&\text{ if $v\ndivides \frakn$ is ramified.}\end{cases}\]
If $v\mid \frakn$, then $v$ is ramified and
\[\localK{\chi,v}^*(a)= \mu\Abs^\onehalf(a)\bbI_{\uf^{-1}\cO}(a).\]
\end{lm}
\begin{proof} It is well-known that if $\pi=\pi(\mu,\nu)$ is a unramified principal series, then
\[W^0_v(\bfd(a))=\bbI_{\cO}(a)\abs{a}^\onehalf\cdot \sum_{i+j=v(a),\,i,j\geq 0}\mu(\uf^i)\nu(\uf^j)\]
(\cf \cite[Thm.\,4.6.5]{Bump:AutoRep}). It follows from the definition of $\localW{\chi,v}$ that $\localW{\chi,v}=W^0_v$ if $v\ndivides \frakn$ is unramified and
\[\localW{\chi,v}(g)=\onehalf \cdot W^0_v(g)+\onehalf\cdot W^0_v(g\bfd(\uf))\chi(\uf_E)\text{ if $v\ndivides \frakn$ is ramified }.\]
If $v|\frakn$, then $v\in A(\chi)$. By \eqref{R1} $v$ is ramified, and we find that
\[\localW{\chi,v}(g)=\onehalf\cdot W^0_v(g\bdd(\uf))+\onehalf\cdot W^0_v(g\bfw)\om(\uf)\chi(\uf_E).\]
The assertion follows from the formulas of $W^0_v$ in \propref{P:toric_inert_sp.W}.
\end{proof}
To treat the case $v$ is non-split with $c_v(\chi)>0$, \ie $v\in B(\chi)$, we need to introduce certain partial Gauss sums. For a non-split place $v$, write $\pi=\pi(\mu,\nu)$ with unramified $\mu$ and $\mu\nu^{-1}(\uf)\not =\abs{\uf}^{-1}$ if $\pi$ is unramified or special. Define a character $\chpi:E^\x\to\C^\x$ by
\beq\label{E:char.W}\chpi(t):=\mu(\rmN(t))\cdot\chi\Abs_\cK^\onehalf(t).\eeq
Recall that the partial Gauss sum $\wtd A_\beta(\chpi)$ in \cite[(4.17)]{Hsieh:Hecke_CM} is defined by
\[\wtd A_\beta(\chpi):=\lim_{n\to\infty}\int_{\uf^{-n}\cO}\chpi^{-1}(x+\OKbasis)\addchar(-d_F^{-1}\beta x)dx \quad(\beta\in F^\x).\]
\begin{lm}\label{L:FCformula.W}Let $v\in B(\chi)$ be a non-split place with $c_v(\chi)>0$. Then we have
\begin{align*}\frac{e_v}{L(1,\tau_{E/F})}\cdot \localK{\chi,v}(a)=&\wtd A_a(\chpi)\nu\Abs^\onehalf(a)\cdot \nu\Abs^\onehalf(\uf^m)\abs{D_{F}}^{-\onehalf}\\
&\times \begin{cases}1&\text{ if } v\ndivide \frakn^-,\\
-1&\text{ if } v\mid \frakn^-_s,\\
\abs{\uf^m}\chi(\bfdelta_v)\abs{\bfdelta_v}_E^\onehalf\frac{\ep(0,\chi^{-1},\addchar_E)}{\ep(-1,\om,\addchar)}\abs{\cD_F}^{-\onehalf}&\text{ if } v\mid \frakn^-_r.
\end{cases}\end{align*}
\end{lm}
\begin{proof}
We recall the Whittaker linear functional $\Lam:I(\mu,\nu)\to\C$ (\cite[(6.9), p.498]{Bump:AutoRep}) is defined by
\[\Lam(f)=\int_{F^\x}f(\MX{0}{-1}{1}{x})\addchar(-x)dx:=\lim_{n\to\infty}\int_{\uf^{-n}\cO}f(\MX{0}{-1}{1}{x})\addchar(-x)dx.\]
Let $\cmpt=\cmptv=\DII{d_F}{d_F^{-1}}$ and $m=m_v(\chi,\pi)$. Define $\cP_{\chi,\cmpt} \LR_v^m\in\End_\C I(\mu,\nu)$ by
\[\cP_{\chi,\cmpt} \LR_v^m f(g)=\vol(E^\x/F^\x,dt)^{-1}\int_{E^\x/F^\x}f(g\DII{1}{\uf^m}\iota_\cmpt(t))dt.\]
We are going to choose a section $f^0$ in $I(\mu,\nu)^{K^0(\uf)}$ such that
\[W^0_v(g)=\Lam(\pi(g)f^0).\]
Let $f^0_\chi=\cP_{\chi,\cmpt}\LR_v^m f^0$. Then it is not difficult to see that
\[\localK{\chi,v}(a)=\localW{\chi,v}(\bfd(a))=\Lam(\pi(\bfd(a))f^0_\chi).\]
Put \[f^0_\chi(\cmpt)^*:=\nu^{-1}\Abs^\onehalf(\uf^m)\bfv_E\cdot  f^0_\chi(\cmpt)\quad(\bfv_E=e_v\abs{\cD_E}_E^{\onehalf}\abs{\cD_F}^{-\onehalf}).\]
Therefore we have
\begin{align*}\localK{\chi,v}(a)=&\nu\Abs^\onehalf(a)\int_{F}f^0_\chi(\MX{0}{-1}{1}{x})\addchar(-ax)dx\\
=&f^0_\chi(\DII{d_F}{d_F^{-1}})\abs{\cD_F}^{-1}\cdot\nu\Abs^\onehalf(a)\int_{F}\chpi^{-1}(x+\OKbasis)\addchar(-d_F^{-1}ax)dx\\
=&f^0_\chi(\cmpt)^*\cdot \nu\Abs^\onehalf(a)\wtd A_a(\chpi)\cdot e_v^{-1}\abs{\cD_E}_E^{-\onehalf}\abs{\cD_F}^{-\onehalf}\nu\Abs^{-\onehalf}(\uf^m).
\end{align*}
It remains to determine $f^0$ and $f^0_\chi(\cmpt)$. Since $W^0$ is uniquely characterized by the property that $W^0\in\cW(\pi,\psi)$ is right $K^0(\uf)$-invariant with $W^0(1)=1$, it suffices to construct $f^0\in \pi(\mu,\nu)^{K^0(\uf)}$ with $\Lam(f^0)=1$. To calculate $f^0_\chi(\cmpt)$, following \eqref{E:long.W} we have
\begin{align*}f^0_\chi(\cmpt)^*
=&\int_{\cO}\chi(1+y\OKbasis)f^0(\cmpt\cdot \LR_v^{-m}\iota_\cmpt(1+y\OKbasis)\LR_v^m )d'y
+\int_{\uf\cO}\chi(y+\OKbasis)f^0(\cmpt\cdot \LR_v^{-m}\iota_\cmpt(y+\OKbasis)\LR_v^m)\abs{\cD_E}_E^{-1}d'y\\
=&X_{m+1}\cdot f^0(\cmpt)+\sum_{r=0}^m \int_{\uf^r\cO^\x}\chi(1+y\OKbasis)\om(\uf^{-m}y)d'y\cdot f^0(\cmpt\bfd(\uf^{2(m-r)}))\\
&+Y_0\cdot \om(\uf^{-m})f^0(\cmpt\bfd(\uf^{2m+e_v-1})).
\end{align*}

Suppose that $\pi$ is a unramified principal series ($v\ndivides \frakn^-$) or special representation ($v\mid\frakn^-_s$). Then \begin{align*}f^0_\chi(\cmpt)^*
=&X_m\cdot f^0(\cmpt)+\sum_{r=0}^{m-1}(X_r-X_{r+1})\cdot\om(\uf^{r-m})\cdot f^0(\cmpt\bfd(\uf^{2(m-r)}))\\
&+Y_0\cdot \om(\uf^{-m})f^0(\cmpt\bfd(\uf^{2m+e_v-1})).\end{align*}
Let $f^{sph}$ be the unique $K^0_v$-invariant function in $I(\mu,\nu)$ with $f^{sph}(\cmpt)=L(1,\mu\nu^{-1})\abs{\cD_{\cF}}^\onehalf$.
If $\pi$ is an unramified principal series, then we can take $f^0=f^{sph}$ (\cite[Prop.\,4.6.8]{Bump:AutoRep}), and following the computation of the case $c_v(\chi)>0$ in \propref{P:toric_inert_un.W} we find that
\begin{align*}f^0_\chi(\cmpt)^*
=&X_m\cdot (f^0(\cmpt)-\om(\uf^{-1})f^0(\bfd(\uf^2)\cmpt))=X_m \cdot(1-\mu\nu^{-1}\Abs(\uf))\cdot f^0(\cmpt)\\
=&\abs{\uf^m}\abs{\cD_E}_E^\onehalf L(1,\tau_{E/F}).
\end{align*}
If $\pi$ is special, then
\[f^0=f^{sph}-\mu^{-1}\Abs^\onehalf(\uf)^{-1}\pi(\DII{1}{\uf})f^{sph},\]
and following the computation of the case $c_v(\chi)>0$ in \propref{P:toric_inert_sp.W} we find that
\begin{align*}f^0_\chi(\cmpt)^*=&X_{m+1}\cdot (f^0(\cmpt)-f^0(\cmpt\cdot \bfw))\\
=&X_{m+1}\cdot (-\abs{\uf}^{-1}+\abs{\uf})f^{sph}(\cmpt)\\
=&(-1)\cdot\abs{\uf^{m}}\abs{\cD_E}_E^\onehalf\cdot L(1,\tau_{E/F}).\end{align*}

Finally, suppose that $\pi$ is a ramified principal series with ramified $\nu$ ($v\mid\frakn^-_r$). Let $B(F)$ be the group of upper triangular matrices in $\GL_2(F)$. Let $f^0\in I(\mu,\nu)$ be the function supported in $B(F)\bfw N(\cD_F^{-1})$ such that
\[f^0(\cmpt\bfw n)=\abs{\cD_F}^\onehalf\text{ for every }n\in N(\cD_F^{-1}).\]
Then one checks easily that $f^0$ does the job. Following the computation in \propref{P:toric_ramified.W}, we find that
\begin{align*}
f^0_\chi(\cmpt)^*&=\om(\uf^{-m})\abs{\uf^{2m}}\cdot\chi(\bfdelta_v)\abs{\bfdelta_v}_E^\onehalf\frac{\ep(0,\chi^{-1},\psi_E)}{\ep(-1,\om,\psi)}L(1,\tau_{E/F})\abs{\cD_{E}}_E^\onehalf\abs{\cD_{F}}^{-1}\cdot f^0(\cmpt\cdot\bfw)\\
&=\abs{\uf^m}\cdot L(1,\tau_{E/F})\cdot \om^{-1}\Abs(\uf^m)\cdot \chi(\bfdelta_v)\frac{\ep(0,\chi^{-1},\psi_E)}{\ep(-1,\om,\psi)}\abs{\cD_E}_E^\onehalf\abs{\cD_F}^{-\onehalf}.
\end{align*}
This completes the proof in all cases.
\end{proof}

To investigate the $p$-integrality of $\localK{\chi,v}^*$, we define the local invariant $\mu_p(\chpi)$ by
\beq\label{E:9.W}\mu_p(\chpi):=\inf_{x\in\cK_v^\x}v_p(\chpi(x)-1).\eeq
By \cite[(4.17)]{Hsieh:Hecke_CM}, $\wtd A_\beta(\chpi)$ is indeed an algebraic integer. 
Moreover, it is proved in \cite[Lemma 6.4]{Hsieh:Hecke_CM} that
\[\mu_p(\chpi)>0\iff \wtd A_\beta(\chpi)\con 0\pmod{\frakm}\text{ for all $\beta\in F^\x$.}\]
Therefore, it follows from \lmref{L:FCformula.W} that if $v\in B(\chi)$, then $\localK{\chi,v}^*$ takes values in $\Zbarp$ and \[\localK{\chi,v}^*\con 0\pmod{\frakm}\iff \mu_p(\chpi)>0.\] We summarize our discussion in the following proposition.
\begin{prop}\label{P:4.W} Let $\EucO$ be the finite extension of $\cO_{L_\pi}$ generated by $\stt{\localK{\chi,v}^*(1)}_{v\in B(\chi)}$ and the values of $\wh\chi$. Then we have
\begin{mylist}\item
the normalized local Fourier coefficient $\localK{\chi,v}^*$ takes values in $\EucO$ for every finite place $v\ndivide p$,
 \item if either $v\not\in B(\chi)$ is unramified or $v\in A(\chi)$, then $\localK{\chi,v}^*(1)=1$,
 \item if $v\ndivides \frakn$ is ramified with $c_v(\chi)=0$, then $\localK{\chi,v}^*(\uf^{-1})=1$,
 \item if $v\in B(\chi)$, then $\mu_p(\chpi)=0$ if and only if there exists $\eta_v\in F^\x$ such that
\[\localK{\chi,v}^*(\eta_v)\not\con 0\pmod{\frakm}.\]
\end{mylist}
\end{prop}

\subsection{The global toric period integral}We return to the global situation. Let $\localW{\chi,f}^\setp$ be the prime-to-$p$ Whittaker function given by
\[\localW{\chi,f}^\setp=\prod_{v\in\bdh,\,v\ndivides p}\localW{\chi,v}\in\bigot_{v\in\bdh,v\ndivides p}\cW(\pi_v,\addchar_v).\]

\begin{defn}\label{D:toricWhittaker.W}Let $\localW{\chi,\infty}:=\prod_{\sg\in\Sg}W_{k_\sg}$. Define the $p$-modified toric Whittaker function $\localW{\chi}$ by
\beq\label{E:globalWhitaker.W}\begin{aligned}
\localW{\chi}=&W_{\chi,\infty}\cdot\localW{\chi,f}^\setp\cdot\prod_{v|p}\localW{\chi,v}^\flat\in\cW(\pi,\addchar).\\
\end{aligned} \eeq
Let $u=(u_v)\in (\cO_{\cF}\ot_\Z\Zp)^\x=\prod_v\cO_{\cF_v}^\x$. The \emph{$u$-component} $\localW{\chi,u}$ of $\localW{\chi}$ is defined by
\beq\label{E:Wu.W}\localW{\chi,u}=\localW{\chi,\infty}\cdot\localW{\chi,f}^\setp\cdot\prod_{v|p}\localW{\chi,u_v,v}.\eeq
Recall that the automorphic form $\vp_{W}\in\cA(\pi)$ associated to $W\in\cW(\pi,\addchar)$ is defined by
\beq\label{E:WhittakerAuto.W}\vp_{W}(g):=\sum_{\beta\in\cF}W(\DII{\beta}{1}g).\eeq
Let $\vp_\chi$ (resp. $\vp_{\chi,u}$) be the automorphic form associated to $\localW{\chi}$ (resp. $\localW{\chi,u}$). Let $\cU_p=\prod_{v|p}\cU_v$ be the torsion subgroup of $(\cO_{\cF}\ot_\Z\Zp)^\x$. It follows immediately from the definition \eqref{E:DeflocalWp.W} that
\beq\label{E:ucompnent.W}\vp_\chi=\sum_{u\in\cU_p}\vp_{\chi,u}.\eeq
\end{defn}

Choose a sufficiently small prime-to-$p$ integral ideal $\frakn_1$ such that $\localW{\chi,v}$ is invariant by $U_v(\frakn_1)$ for all $v\ndivides p$. Let $\opcpt=\prod_v \opcpt_v\subset \GL_2(\AFf)$ be an open-compact subgroup such that
\beq\label{E:opncpt.W}
\opcpt_v=K^0_v\text{ if }v\mid p\,;\,
\opcpt_v\subset U_v(\frakn_1)\text{ if }v\ndivides p.\eeq
For each positive integer $n$, put
\[\lsgN:=\stt{g\in \opcpt\mid g_v\con\MX{1}{*}{0}{1}\pmod{p^n}\text{ for all }v|p}.\]
It is easy to verify that $\localW{\chi}$ and $\localW{\chi,u}$ (and hence $\vp_\chi$ and $\vp_{\chi,u}$) are invariant by $\lsgN$ for sufficiently large $n$. The following lemma immediately follows from \lmref{L:toric.W},
\begin{lm}\label{L:4.W} Let $\cT=\prod_v'\cT_v\subset \AK^\x$. Then $\vp_\chi$ is a toric automorphic form in the sense that for all $t\in\cT$, we have
\begin{align*}\pi(\iota_\cmpt(t))\wtd V_+^m\vp_\chi
=&\chi^{-1}(t)\wtd V_+^m\vp_\chi.
\intertext{In addition, for all $t\in\cT_f=\prod_{v\in\bdh}'\cT_v$, we have}
\pi(\iota_\cmpt(t))\vp_{\chi,u}=&\chi^{-1}(t)\vp_{\chi,u\cdot t^{1-c}},\end{align*}
where $u\cdot t^{1-c}:=ut_{\Sg_p}t_{\Sgbar_p}^{-1}\in(\cO_\cF\ot_\Z\Zp)^\x$.
\end{lm}
Decompose $\frakc^-_\chi=\frakc^-_{\chi,1}\frakc^-_{\chi,2}$ such that $(\frakc^-_{\chi,1},\frakn^-_r)=1$ and $\frakc^-_{\chi,2}$ has the same support with $\frakn^-_r$. Define a constant $C'(\pi,\chi)$ by
\beq\label{E:constantC.W}\begin{aligned}
C'(\pi,\chi):=&2^{\#(A(\chi))+3[\cF:\Q]}\cdot\rmN_{\cF/\Q}(\frakc^-_\chi)^{-1}\om(\frakc^-_{\chi,1})\om(\frakn_s^-)^{-1}\prod_{v\ndivides p\Bad}\om(\det \cmptv)\\ &\x\prod_{w|\Csplit,\,v=w\wbar}\ep(\onehalf,\pi_v\ot\chi_w,\addchar_v)\om^{-1}\chi_w^{-2}(-2\delta)\cdot\prod_{v|\frakn^-_r}\chi_v(-\bfdelta_vd_{\cF_v}^{-1})\abs{\bfdelta_v}_{\cK_v}^\onehalf\ep(0,\chi_v^{-1},\addchar_{\cK_v})
\end{aligned}\eeq
Note that $C'(\pi,\chi)$ is actually a \padic unit as $p>2$ and $(p,\Csplit\frakn^-)=1$. We introduce the normalization factor $N(\pi,\chi)$ given by
\beq\label{E:constantN.W}N(\pi,\chi):=\prod_{v\in B(\chi)}L(1,\tau_{\cK_v/\cF_v})n_v.
\eeq
 We have the following central value formula of the toric integral $P_\chi(\pi(\cmpt)\wtd V_+^m\vp_\chi)$.
\begin{thm}\label{T:main1.W}We have
\begin{align*}
P_\chi(\pi(\cmpt)\wtd V_+^m\vp_\chi)^2=&\abs{D_{\cK}}_\R^{-\onehalf}\frac{\Gamma_\Sg(k+m)\Gamma_\Sg(m+1)}{(4\pi)^{k+2m+1\cdot \Sg}}\cdot E_{\Sg_p}(\pi,\chi)\cdot L(\onehalf,\pi_\cK\ot\chi)\cdot C'(\pi,\chi)N(\pi,\chi)^2,
\end{align*}
where $E_{\Sg_p}(\pi,\chi)$ is the Coates' \padic multiplier given by
\[E_{\Sg_p}(\pi,\chi)=
\prod_{w\in\Sg_p,v=w\wbar}\ep(\onehalf,\pi_v\ot\chi_{\wbar},\addchar_v)L(\onehalf,\pi_v\ot\chi_{\wbar})^{-2}\om^{-1}\chi_w^{-2}(-2\delta).\]
\end{thm}
\begin{proof}Note that $\wtd V_+^m\vp_\chi$ is the automorphic form associated to the Whittaker function
\[\wtd V_+^m\localW{\chi}=\wtd V_+^m \localW{\chi,\infty}\cdot \localW{\chi,f}^\setp\cdot \prod_{v|p}\localW{\chi,v}^\flat.\]
Hence, by \propref{P:Waldformula.W} we find that
\begin{align*}&P_\chi(\pi(\cmpt)\wtd V_+^m\vp_\chi)^2\\
=&\prod_{\sg\in\Sg}\localP{\pi(\cmpt_\sg)\wtd V^{m_\sg}_+W_{k_\sg}}{\chi_\sg}\prod_{v|p}\frac{1}{L(\onehalf,\pi_{\cK_v}\ot\chi_v)}\cdot\localP{\pi(\cmptv)\localW{\chi,v}^\flat}{\chi_v}\\
&\times \prod_{v\in\bdh,v\ndivides p}\frac{1}{L(\onehalf,\pi_{\cK_v}\ot\chi_v)}\cdot\localP{\pi(\cmptv)\localW{\chi,v}}{\chi_v}\cdot L(\onehalf,\pi_\cK\ot\chi).
\end{align*}
Combining the local calculations of toric integrals of our Whittaker functions (\propref{P:Architoric.W}, \propref{P:toric_split.W}, \propref{P:toric_inert_un.W}, \propref{P:toric_inert_sp.W} and \propref{P:toric_ramified.W}) yields the central value formula.
\end{proof}
\begin{Remark}Let $\vp_\chi^\sharp$ be the automorphic form associated to the toric Whittaker function $\localW{\chi}^\sharp:=\localW{\chi,\infty}\cdot \prod_{v\in\bdh}\localW{\chi,v}$. Then we obtain the following central value formula:
\[P_\chi(\pi(\cmpt)\wtd V_+^m\vp_\chi^\sharp)^2=\abs{D_\cK}_\R^{-\onehalf}\frac{\Gamma_\Sg(k+m)\Gamma_\Sg(m+1)}{(4\pi)^{k+2m+1\cdot\Sg}}\cdot L(\onehalf,\pi_\cK\ot\chi)\cdot C'(\pi,\chi)N(\pi,\chi)^2.\]
\end{Remark}

%% file: WaldPHF.tex
\def\Ig{I}
\def\setpN{{(p\None)}}
\def\OF{\cO_{\cF}}
\def\SCH{SCH}
\def\ENS{SETS}
\def\OPU{U}
\def\baseR{\cW}
\def\None{N}
\section{Review of Hilbert modular forms}\label{S:Hilbert.W}
In this section, we review some standard facts about Hilbert modular Shimura varieties and Hilbert modular forms. The main purpose of this section is to recall the notation in
\subsection{}Let $V=\cF e_1\oplus\cF e_2$ be a two dimensional
$\cF$-vector space and $\pairing:V\x V\to \cF$ be the $\cF$-bilinear
alternating pairing defined by $\pair{e_1}{e_2}=1$.  Let $\sL=\OF e_1\oplus \cD_\cF^{-1} e_2$ be the standard $\OF$-lattice in $V$, which is self-dual with respect to $\addchar(\pairing)$. For $g=\MX{a}{b}{c}{d}\in M_2(\cF)$, we define an involution $g\mapsto g':=\MX{d}{-b}{-c}{a}$. We
identify vectors in $V$ with row vectors according to the basis $e_1,e_2$, so $G(F)=\GL_2(F)$ has a natural right
action on $V$. If $g\in G(F)$, then $g'=g^{-1}\det g$. Define a left action of $G$ on $V$ by $g*x:=x\cdot g',\,x\in V$.

Hereafter, we let $\opcpt$ be an open-compact subgroup of $G(\AFf)$ satisfying \eqref{E:opncpt.W} and the following conditions:
\beqcd{neat}\opcpt\text{ is neat and }\det (\opcpt)\cap
\cO_{F,+}^\x\subset (K\cap \cO_F^\x)^2.\eeqcd
We also fix a prime-to-$p$ positive integer $N$ such that $U(N)\subset \opcpt$.

\subsection{Kottwitz models}
We recall Kottwitz models of Hilbert modular Shimura varieties following the exposition in \cite{Hida:p-adic-automorphic-forms}.
\begin{defn}[$S$-quadruples]\label{D:6.H} Let $\Box$ be a finite set of
rational primes not dividing $\None$ and let $U$ be an open-compact subgroup of $\opcpt^0$ such that $U(\None)\subset U$.
Let $\baseR_U=\Z_\bbox[\zeta_{\None}]$ with $\zeta_N=\exp(\frac{2\pii}{\None})$. Define the fibered category $\cA^\bbox_{\OPU}$ over the category $SCH_{/\baseR_U}$ of schemes over $\baseR_U$ as follows. Let $S$ be a locally noethoerian connected $\baseR_U$-scheme and let $\ol{s}$ be a geometric point of $S$. Objects are abelian varieties with real
multiplication (AVRM) over $S$ of level $\OPU$, \ie a
$S$-\emph{quadruple} $(A,\ollam,\iota,\ol{\eta}^\bbox)_S$ consisting of the
following data:
\begin{enumerate}
\item $A$ is an abelian scheme of dimension $d$ over $S$.
\item $\iota :\OF\hookrightarrow \End_S A\ot_\Z\ZZbox$.
\item $\lam$ is a prime-to-$\Box$ polarization of $A$ over $S$ and
$\ollam$ is the $\cO_{\cF,\bbox,+}$-orbit of $\lam$. Namely
\[\ollam=\cO_{\cF,\bbox,+}\lam:=\stt{\lam'\in\Hom(A,A^t)\ot_\Z\ZZbox\mid \lam'=\lam\circ a,\,a\in \cO_{\cF,\bbox,+}}.\]
\item $\ol{\eta}^\bbox=\eta^\bbox\OPU^\bbox$ is a $\pi_1(S,\ol{s})$-invariant $\OPU^\bbox$-orbit of isomorphisms of $\OF$-modules $\eta^\bbox: \sL\ot_\Z\A_f^\bbox\isoto V^\bbox(A_{\ol{s}}):=H_1(A_{\ol{s}},\Zhat^\bbox)\ot_\Z\A_f^\bbox$. Here we define $\eta^\bbox g$ for $g\in G(\AFf^\bbox)$ by $\eta^\bbox g(x)=\eta^\bbox(g*x)$.
\end{enumerate}
Furthermore, $(A,\ollam,\iota,\ol{\eta}^\bbox )_S$ satisfies the
following conditions:
\begin{itemize}
\item Let ${}^t$ denote the Rosati involution induced by $\lam$ on
$\End_SA\ot\ZZbox$. Then $\iota(b)^t=\iota(b),\, \forall\, b\in
\OF.$
\item Let $e^\lam$ be the Weil pairing induced by $\lam$. Lifting the isomorphism $\Z/\None\Z\iso \Z/\None\Z(1)$ induced by $\zeta_{\None}$ to an isomorphism $\zeta:\Zhat\iso\Zhat(1)$, we can regard $e^\lam$ as an $\cF$-alternating form
$e^\lam:V^\bbox(A)\times V^\bbox(A)\to \cD_F^{-1}\ot_\Z\A_f^\bbox$. Let $e^\eta$ denote the
$\cF$-alternating form on $V^\bbox(A)$ induced by
$e^\eta(x,x')=\pair{x\eta}{x'\eta}$. Then
\[e^\lam=u\cdot e^\eta\text{ for some }u\in\AFf^\bbox.\]
\item As
$\OF\ot_\Z\cO_S$-modules, we have an isomorphism $\Lie A\iso \OF\ot_\Z\cO_S$ locally under Zariski topology of $S$.
\end{itemize}
For two $S$-quadruples $\ulA=(A,\ollam,\iota,\ol{\eta}^\bbox )_S$ and $\ul{A'}=(A',\ol{\lam'},\iota',\ol{\eta'}{}^\bbox)_S$,
we define morphisms by
\[\Hom_{\cA^\bbox_{\opcpt}}(\ulA,\ul{A'})=\stt{\phi\in \Hom_{\OF}(A,A')\mid
\phi^*\ol{\lam'}=\ollam,\,\phi\circ \ol{\eta'}{}^\bbox=\ol{\eta}^\bbox }.\] We
say $\ulA\sim \ul{A'}$ (resp. $\ulA\iso \ul{A'}$) if there exists a
prime-to-$\Box$ isogeny (resp. isomorphism) in
$\Hom_{\cA_\opcpt^\bbox}(\ulA,\ul{A'})$.
\end{defn}
We consider the cases when $\Box=\emptyset$ and $\stt{p}$. When
$\Box=\emptyset$ is the empty set and $\OPU$ is an open-compact
subgroup in $G(\AFf^\bbox)=G(\AFf)$, we define the functor
$\cE_{\OPU}:\SCH_{/\baseR_U}\to\ENS$ by
\[\cE_\OPU(S)=\stt{(A,\ollam,\iota,\ol{\eta} )_S\in\cA_\opcpt(S)}/\sim.\] By the theory of Shimura-Deligne, $\cE_{\OPU}$ is represented by
$\sh_\OPU$ which is a quasi-projective scheme over $\baseR_{\OPU}$. We define
the functor $\frakE_\OPU:\SCH_{/\baseR_{\OPU}}\to\ENS$ by
\[\frakE_\OPU(S)=\stt{(A,\ollam,\iota,\ol{\eta})\in\cA^\bbox_\OPU(S)\mid \eta^\bbox(\sL\ot_\Z\Zhat)=H_1(A_{\ol{s}},\Zhat)}/\iso.\]
By the discussion in \cite[p.136]{Hida:p-adic-automorphic-forms}, we have $\frakE_\opcpt\isoto\cE_\opcpt$ under the hypothesis \eqref{neat}.

When $\Box=\stt{p}$ and $\OPU=\opcpt$, we let $\baseR=\baseR_\opcpt=\Z_\setp[\zeta_N]$ and define functor
$\cE^\setp_{\opcpt}:\SCH_{/\baseR}\to\ENS$ by
\[\cE^\setp_{\opcpt}(S)=\stt{(A,\ollam,\iota,\ol{\eta}^\setp)_S\in\cA_{\opcpt^\setp}^\setp(S)}/\sim.\]
In \cite{Kottwitz:Points-On-Shimura-Varieties}, Kottwitz shows
$\cE^\setp_{\opcpt}$ is representable by a quasi-projective
scheme $\sh^\setp_{\opcpt}$ over $\baseR$ if $\opcpt$ is neat. Similarly we
define the functor $\frakE_K^\setp:\SCH_{/\baseR}\to\ENS$ by
\[\frakE_K^\setp(S)=\stt{(A,\ollam,\iota,\ol{\eta}^\setp)\in\cA^\setp_\opcpt(S)\mid \eta^\setp(\sL\ot_\Z\Zhat^\setp)=H_1(A_{\ol{s}},\Zhat^\setp)}/\iso.\]
It is shown in \cite[\S 4.2.1]{Hida:p-adic-automorphic-forms} that
$\frakE^\setp_K\isoto\cE^\setp_K$.

Let $\plideal$ be a prime-to-$p\None$ ideal of $\OF$ and let $\bfc\in (\AFf^\setpN)^\x$ such that $\plideal=\il_{\cF}(\bfc)$. We say $(A,\lam,\iota,\ol{\eta}^\setp)$ is $\plideal$-polarized if $\lam\in\ollam$ such that $e^\lam=ue^{\eta},\,u\in \bfc\det(\opcpt)$.
The isomorphism class $[(A,\lam,\iota,\ol{\eta}^\setp)]$ is independent of a choice of $\lam$ in $\ol{\lam}$ under the assumption \eqref{neat} (\cf\cite[p.136]{Hida:p-adic-automorphic-forms}).
We consider the functor
\[\frakE^\setp_{\plideal,\opcpt}(S)=\stt{\text{$\plideal$-polarized $S$-quadruple }[(A,\lam,\iota,\ol{\eta}^\setp)_S]\in\frakE^\setp_{\opcpt}(S)}.\]
Then $\frakE^\setp_{\plideal,\opcpt}$ is represented by a geometrically irreducible scheme $\sh^\setp_\opcpt(\plideal)_{/\baseR}$, and we have
\beq\label{E:decomposition}\sh^\setp_\opcpt{}_{/\baseR}=\disjoint_{[\plideal]\in\Cl^+_\cF(\opcpt)}\sh^\setp_\opcpt(\plideal)_{/\baseR},\eeq
where $\Cl^+_\cF(\opcpt)$ is the narrow ray class group of $\cF$ with level $\det(\opcpt)$.
\def\Igusa{\Ig_{\opcpt,n}}
\subsection{Igusa schemes}\label{17.H}
Let $n$ be a positive integer. Define the functor $\cI^\setp_{K,n}:\SCH_{/\baseR}\to\ENS$ by
\[S\mapsto \cI^\setp_{K,n}(S)=\stt{(A,\ollam,\iota,\lpp,\lp)_S}/\sim,
\]
where $(A,\ollam,\iota,\lpp)_S$ is a $S$-quadruple, $\lp$ is a level $p^n$-structure, \ie an $\OF$-group scheme morphism:
\[\lp:\cD_F^{-1}\ot_\Z\bbmu_{p^n}\hookto A[p^n],\]
and $\sim$ means modulo prime-to-$p$ isogeny. It is known that $\cI^\setp_{K,n}$ is relatively representable over $\cE^\setp_{K}$ (\cf\cite[Lemma (2.1.6.4)]{HLS}) and thus is represented by a scheme $\Igusa$.

Now we consider $S$-quintuples $(A,\lam,\iota,\lpp,\lp)_S$ such that $[(A,\lam,\iota,\lpp)]\in\frakE^\setp_{\plideal,K}(S)$.
Define the functor $\cI^\setp_{K,n}(\plideal):\SCH_{/\baseR}\to\ENS$ by
\begin{align*}S\mapsto\cI^\setp_{K,n}(\plideal)(S)&=\stt{(A,\lam,\iota,\lpp,\lp)_S\text{ as above}}/\iso.\\
\end{align*}
Then $\cI^\setp_{K,n}(\plideal)$ is represented by a scheme $\Igusa(\plideal)$ over $\sh^\setp_\opcpt(\plideal)$, and $\Igusa(\plideal)$ can be identified with a geometrically irreducible subscheme of $\Igusa$ (\cite[Thm.\,(4.5)]{DR_padic_L}).
For $n\geq n'>0$, the natural morphism
$\pi_{n,n'}:\Igusa(\plideal)\to\Ig_{\opcpt,n'}(\plideal)$ induced by the
inclusion $\cD_F^{-1}\ot\bbmu_{p^{n'}}\hookto \cD_F^{-1}\ot\bbmu_{p^n}$ is finite
\etale.
The forgetful
morphism $\pi:\Igusa(\plideal)\to \sh^\setp_{\opcpt}(\plideal)$ defined by
$\pi:(\ulA,\lp)\mapsto \ulA$ is \etale for all $n>0$. Hence
$\Igusa(\plideal)$ is smooth over $\Spec\baseR$. We write $\Ig_K(\plideal)$ for $\prolim_n\Igusa(\plideal)$.

\subsection{Complex uniformization}\label{S:cpx} We describe the complex points $\sh_\OPU(\C)$ for $U\subset G(\AFf)$. Put
\[X^+=\stt{\tau=(\tau_\sg)_{\sg\in\Sg}\in\C^{\Sg}\mid \Im \tau_\sg >0\text{ for all } \sg\in\Sg}.\]
The action of $g=(g_\sg)_{\sg\in\Sg}\in G(\cF\ot_\Q\R)$ with $g_\sg=\MX{a_\sg}{b_\sg}{c_\sg}{d_\sg}$ and $\det g_\sg>0$ on $X^+$ is given by $\tau=(\tau_\sg)\mapsto g\tau=\left(\frac{a_\sg \tau_\sg+b_\sg}{c_\sg \tau_\sg+d_\sg}\right)$. Let $\cF_+$ be the set of totally positive elements in $\cF$ and let $G(\cF)^+=\stt{g\in G(\cF)\mid \det g\in\cF_+}$. Define the complex Hilbert modular Shimura variety by
\[M(X^+,\OPU):=G(\cF)^+\bksl X^+\x G(\AFf)/\OPU.\]
It is well known that $M(X^+,\opcpt)\isoto\sh_\OPU(\C)$ by the theory
of abelian varieties over $\C$ (\cf\cite[\S\,4.2]{Hida:p-adic-automorphic-forms}).

For $\tau=(\tau_\sg)_{\sg\in\Sg}\in X^+$, we let $p_\tau$ be the isomorphism $V\ot_\Q\R\isoto \C^{\Sg}$ defined by
$p_\tau(ae_1+be_2)=a\tau+b$ with $a,b\in \cF\ot_\Q\R=\R^{\Sg}$. We can associate a AVRM to $(\tau,g)\in X^+\x G(\AFf)$ as follows.
\begin{itemize}
\item The complex abelian variety $\EucA_g(\tau)=\C^{\Sg}/p_\tau(g*\sL)$.
\item  The $\cF_+$-orbit of polarization
$\ol{\pairing}_\can$ on $\EucA_g(\tau)$ is given by the Riemann form $\pairing_\can:=\pairing\circ p_\tau^{-1}$.
\item The $\iota_\C:O\hookto\End \EucA_g(\tau)\ot_\Z\Q$ is induced from the pull back of the natural $\cF$-action on $V$ via $p_\tau$.
\item The level structure $\eta_g:
\sL\ot_\Z\A_f \isoto (g*\sL)\ot_\Z\A_f=H_1(\EucA_g(\tau),\A_f)$ is defined by  $\eta_g(v)= g*v$.\end{itemize}
Let $\ul{\EucA_g(\tau)}$ denote the $\C$-quadruple $(\EucA_g(\tau),\ol{\pairing}_\can,\iota_\C,\opcpt\eta_g)$. Then the map $[(\tau,g)]\mapsto [\ul{\EucA_g(\tau)}]$ gives rise to
an isomorphism $M(X^+,\OPU)\isoto \sh_\OPU(\C)$.


For a positive integer $n$, the exponential map gives the isomorphism $\exp(2\pii -):p^{-n}\Z/\Z\iso\bbmu_{p^n}$ and thus induces a level $p^n$-structure $\lp(g_p)$:
\begin{align*}\lp(g_p)\colon&\cD_F^{-1}\ot_\Z\bbmu_{p^n}\isoto \cD_F^{-1}e_2\ot_\Z p^{-n}\Z/\Z\hookto \sL\ot_\Z p^{-n}\Z/\Z\stackrel{g*}\isoto \EucA_g(\tau)[p^n].
\end{align*}
Put
\[\lsgN:=\stt{g\in\opcpt\mid g_p\con\MX{1}{*}{0}{1}\pmod{p^n}}.\]
We have a non-canonical isomorphism:
\begin{align*}
M(X^+,\lsgN)&\isoto \Igusa(\C)\\
 [(\tau,g)]&\mapsto  [(\EucA_g(\tau),\ol{\pairing}_\can,\iota_\C,\ol{\eta}^\setp_g, \lp(g_p))].\end{align*}

Let $\ulz=\stt{z_\sg}_{\sg\in\Sg}$ be the standard complex coordinates of $\C^{\Sg}$ and $d\ulz=\stt{dz_\sg}_{\sg\in\Sg}$. Then $\OF$-action on $d\ulz$ is given by
$\iota_\C(\al)^* dz_\sg=\sg(\al)dz_\sg,\,\sg\in\Sg\iso\Hom(\cF,\C)$. Let $z=z_{id}$ be the coordinate corresponding to $\iota_\infty:\cF\hookto\Qbar\hookto\C$.
Then
\beq\label{E:7.N}(\OF\ot_\Z\C) dz=H^0(\EucA_g(\tau),\Omega_{\EucA_g(\tau)/\C}).\eeq

\subsection{Hilbert modular forms}
Let $k=\sum_{\sg}k_{\sg}\sg\in\Z_{\geq 1}[\Sg]$ such that
\[k_{\sg_1}\con k_{\sg_2}\con\cdots\con k_{\sg_d}\pmod{2}\text{ for all }\sg_1,\cdots,\sg_d\in\Sg.\]
For
$\tau=(\tau_\sg)_{\sg\in\Sg}\in X^+$ and $g=(\MX{a_\sg}{b_\sg}{c_\sg}{d_\sg})_{\sg\in\Sg}\in
G(\cF\ot_\Q\R)$, we put
\[\ul{J}(g,\tau)^k=\prod_{\sg\in\Sg}(c_\sg\tau_\sg+d_\sg)^{k_\sg}.\]

\begin{defn}
Let $k_{mx}=\max_{\sg\in\Sg}{k_{\sg}}$. Denote by $\bfM_k(\lsgN,\C)$ the space of holomorphic Hilbert modular forms of weight $k$ and level $\lsgN$. Each $\bff\in\bfM_k(\lsgN,\C)$ is a $\C$-valued function $\bff:X^+\x G(\AFf)\to \C$ such that the function $\bff(-,g_f):X^+\to\C$ is holomorphic for each $g_f\in G(\AFf)$, and for $u\in \lsgN$
and $\al\in G(\cF)^+$,
\[\bff(\al(\tau,g_f)u)=(\det \al)^{-\frac{(k_{mx}-2)\Sg+k}{2}}\ul{J}(\al,\tau)^k\cdot \bff(\tau,g_f).\]
Here $\det\al$ is considered to be the element $(\sg(\det\al))_{\sg\in\Sg}$ in $(\C^\x)^\Sg$.
\end{defn}

For every $\bff\in \bfM_k(\lsgN,\C)$,we have the Fourier expansion
\[\bff(\tau,g_f)=\sum_{\beta\in\cF_+\cup\stt{0}}W_\beta(\bff,g_f)e^{2\pii\Tr_{\cF/\Q}(\beta \tau)}.\]
For a semi-group $L$ in $\cF$, let $L_+=\cF_+\cap L$ and $L_{\geq 0}=L_+\cup \stt{0}$. If $B$ is a ring, we denote by $B\powerseries{L}$ the set of all formal series
\[\sum_{\beta\in L}a_\beta q^\beta,\,a_\beta\in B.\]
Let $a,b\in(\AFf^{(p\None)})^\x$ and let $\fraka=\il_\cF(a)$ and
$\frakb=\il_\cF(b)$. The $q$-expansion of $\bff$ at the cusp $(\fraka,\frakb)$ is given by
\beq\label{E:FC0}\bff|_{(\fraka,\frakb)}(q)=\sum_{\beta\in (\None^{-1}\fraka\frakb)_{\geq 0}}W_\beta(\bff,\MX{a^{-1}}{0}{0}{b})q^\beta\in \C\powerseries{(\None^{-1}\fraka\frakb)_{\geq 0}}. \eeq
If $B$ is a $\baseR$-algebra in $\C$, we put
\begin{align*}
\bfM_k(\plideal,\lsgN,B)&=\stt{\bff\in \bfM_k(\lsgN,\C)\mid \bff|_{(\fraka,\frakb)}(q)\in B\powerseries{(\None^{-1}\fraka\frakb)_{\geq 0}}\text{ for all }(\fraka,\frakb)
\text{ such that $\fraka\frakb^{-1}=\frakc$}}.\\
\end{align*}
\subsubsection{Tate objects}
Let $\sS$ be a set of $d$ linearly $\Q$-independent elements in $\Hom(\cF,\Q)$ such that $l(\cF_+)>0$ for $l\in\sS$. If $L$ is a lattice in $\cF$ and $n$ a positive integer, let
$L_{\sS,n}=\stt{x\in L\mid l(x)>-n\text{ for all }l\in\sS}$ and put $B((L;\sS))=\lim\limits_{n\to\infty} B\powerseries{L_{\sS,n}}$.
To a pair $(\fraka,\frakb)$ of two prime-to-$pN$ fractional ideals, we can attach the Tate AVRM $Tate_{\fraka,\frakb}(q)=\fraka^*\ot_\Z\Gm/q^{\frakb}$ over $\Z((\fraka\frakb;\sS))$ with $O$-action $\iota_\can$, where $\fraka^*:=\fraka^{-1}\cD_F^{-1}$. As described in \cite{Katz:p_adic_L-function_CM_fields}, $Tate_{\fraka,\frakb}(q)$ has a canonical $\fraka\frakb^{-1}$-polarization $\lam_\can$ and also carries $\Om_\can$ a canonical $\OF\ot\Z((\fraka\frakb;\sS))$-generator of $\Omega_{Tate_{\fraka,\frakb}}$ induced by the isomorphism $\Lie(Tate_{\fraka,\frakb}(q)_{/\Z((\fraka\frakb;\sS))})=\fraka^*\ot_\Z\Lie(\Gm)\iso\fraka^*\ot\Z((\fraka\frakb;\sS))$.
Since $\fraka$ is prime to $p$, the natural inclusion $\fraka^*\ot_\Z\bbmu_{p^n}\hookto\fraka^*\ot_\Z\Gm$ induces a canonical level $p^n$-structure $\eta_{p,\can}\colon \cD_\cF^{-1}\ot_\Z\bbmu_{p^n}=\fraka^*\ot_\Z\bbmu_{p^n}\hookto Tate_{\fraka,\frakb}(q)$. Let $\sL_{\fraka,\frakb}=\sL\cdot\MX{\frakb}{}{}{\fraka^{-1}}=\frakb e_1\oplus\fraka^*e_2$. Then we have a level $\None$-structure  $\eta_\can^\setp:\None^{-1}\sL_{\fraka,\frakb}/\sL_{\fraka,\frakb}\isoto Tate_{\fraka,\frakb}(q)[\None]$
over $\Z[\zeta_{\None}]((\None^{-1}\fraka\frakb;\sS))$ induced by the fixed primitive $\None$-th root of unity $\zeta_{\None}$.
We write $\ul{Tate}_{\fraka,\frakb}$ for the Tate $\Z((\fraka\frakb;\sS))$-quadruple $(Tate_{\fraka,\frakb}(q),\lam_\can,\iota_\can,\ol{\eta}^\setp_\can,\eta_{p,\can})$ at $(\fraka,\frakb)$.

\subsubsection{Geometric modular forms}\label{S:GME}We collect here definitions and basic facts of geometric modular forms. The whole theory can be found in \cite{Katz:p_adic_L-function_CM_fields} and \cite{Hida:p-adic-automorphic-forms}.
Let $T$ be the algebraic torus over $\baseR$ such that $T(R)=(\OF\ot_\Z R)^\x$ for every $\baseR$-algebra $R$. Let $\wt\in\Hom(T,\Gm{}_{/\baseR})$. Let $B$ be a $\baseR$-algebra.
Consider $[(\ulA,\lp)]=[(A,\lam,\iota,\lpp,\lp)]\in\Igusa(\plideal)(C)$ (resp. $[(\ulA,\lp)]=[(A,\ollam,\iota,\lpp,\lp)]\in\Igusa(C)$) for a $B$-algebra $C$ with a differential form $\Om$ generating
$H^0(A,\Omega_{A/C})$ over $\OF\ot_\Z C$. A geometric modular form $f$ over $B$ of weight $\wt$ on $\Igusa(\plideal)$ (resp. $\Igusa$) is a functorial rule of
assigning a value $f(\ulA,\lp,\Om)\in C$ satisfying the following axioms.
\begin{mylist}
\item[(G1)] $f(\ulA,\lp,\Om)=f(\ulA',\lp',\Om')\in C$ if $(\ulA,\lp,\Om)\iso (\ulA',\lp',\Om')$ over $C$,
\item[(G2)]For a $B$-algebra homomorphism $\varphi:C\to C'$, we have
\[f((\ulA,\lp,\Om)\ot_C C')=\varphi(f(\ulA,\lp,\Om)),\]
\item[(G3)]$f((\ulA,\lp,a\Om)=\wt(a^{-1})f(\ulA,\lp,\Om)$ for all $a\in T(C)=(\OF\ot_\Z C)^\x$,
\item[(G4)]$f(\ul{Tate}_{\fraka,\frakb},\Om_\can)\in B\powerseries{(\None^{-1}\fraka\frakb)_{\geq 0}}\text{ at all cusps }(\fraka,\frakb)$ in $\Igusa(\plideal)$ (resp. $\Igusa$).
\end{mylist}
For each $k\in\Z[\Sg]$, we regard $k\in\Hom(T,\Gm{}_{/\cW})$ as the character $x\mapsto x^k,\,x\in (\OF\ot_\Z\baseR)^\x$.
We denote by $\cM_k(\plideal,\lsgN,B)$ (resp. $\cM_k(\lsgN,B)$) the space of geometric modular forms over $B$ of weight $k$ on $\Igusa(\plideal)$ (resp. $\Igusa$).
For $f\in\cM_k(\lsgN,B)$, we write $f|_\plideal\in\cM_k(\plideal,\lsgN,B)$ for the restriction $f|_{\Igusa(\plideal)}$.

For each $f\in \cM_k(\lsgN,\C)$, we regard $f$ as a holomorphic Hilbert modular form of weight $k$ and level $\lsgN$ by
\[f(\tau,g_f)=f(\EucA_g(\tau),\ol{\lam_\can},\iota_\C,\ol{\eta}_g,2\pii dz),\]
where $dz$ is the differential form in \eqref{E:7.N}. By GAGA this gives rise to an isomorphism $\cM_{k}(\lsgN,\C)\isoto\bfM_k(\lsgN,\C)$ and $\cM_{k}(\plideal,\lsgN,\C)\isoto\bfM_k(\plideal,\lsgN,\C)$.
Moreover, as discussed in \cite[\S 1.7]{Katz:p_adic_L-function_CM_fields}, we have the following important identity which bridges holomorphic modular forms and geometric modular forms
\[\bff|_{(\fraka,\frakb)}(q)=\bff(\ul{Tate}_{(\fraka,\frakb)},\Om_\can)\in\C\powerseries{(\None^{-1}\fraka\frakb)_{\geq 0}}.\]
By the $q$-expansion principle, if $B$ is $\baseR$-algebra in $\C$ and $\bff \in\bfM_k(\plideal,\lsgN,B)=\cM_{k}(\plideal,\lsgN,\C)$, then $\bff|_\plideal\in\cM_{k}(\plideal,\lsgN,B)$.

\subsubsection{\padic modular forms}Let $B$ be a \padic $\baseR$-algebra in $\Cp$. Let $V(\plideal,\opcpt,B)$ be the space of Katz \padic modular forms over $B$ defined by
\[V(\plideal,\opcpt,B):=\prolim_m\dirlim_n H^0(\Igusa(\plideal){}_{/B/p^mB},\cO_{\Igusa}).\]
In other words, Katz \padic modular forms consist of formal functions on the Igusa tower.

Let $C$ be a $B/p^mB$-algebra. For each $C$-point $[(\ulA,\lp)]=[(A,\lam,\iota,\lpp,\lp]\in\Ig_\opcpt(\plideal)(C)=\prolim_n\Igusa(\plideal){}(C)$, the $p^\infty$-level structure $\lp$ induces an isomorphism $\lp_*:\cD_\cF^{-1}\ot_\Z C\iso \Lie A$ which in turn gives rise to  a generator $\Om(\lp)$ of $H^0(A,\Omega_A)$ as a $\OF\ot_\Z C$-module.
Then we have a natural injection\beq\label{E:padicavatar.N}
\begin{aligned}\cM_k(\plideal,\lsgN,B)&\hookto V(\plideal,\opcpt,B)\\
f&\mapsto \wh{f}(\ulA,\lp):=f(\ulA,\lp,\Om(\lp))
\end{aligned}\eeq
which preserves the $q$-expansions in the sense that $\wh f|_{(\fraka,\frakb)}(q):=\wh f(\ul{Tate}_{\fraka,\frakb})=f|_{(\fraka,\frakb)}(q)$. We call $\wh{f}$ the \padic avatar of $f$.

\subsection{CM points}\label{S:CMpoint}
Recall that we have fixed $\CMP\in\cK$ in \subsecref{SS:choiceofcmpt} satisfying (d1-3) and the associated embedding $\iota:\cK\hookto M_2(\cF)$ in \eqref{E:imb.W}.
Let $q_{\CMP}:\cK\isoto V=\cF e_1\oplus \cF e_2$ be the isomorphism given by $q_{\CMP}(a\CMP+b)=ae_1+be_2$. Then \[q_{\CMP}(x\al)=q_{\CMP}(x)\iota(\al)\text{ for all }x,\al\in\cK,\] and $p_{\CMP}:=q_{\CMP}^{-1}:V\ot_\Q\R\isoto\cK\ot_\Q\R\iso\C^\Sg$ is the period map associated to the point $\CMP_\Sg:=(\sg(\CMP))_{\sg\in\Sg}\in X^+$.

Let $\cmpt=\prod_v\cmptv\in G(\AF)$ where $\cmptv\in G_v$ for each place $v$ is defined in \eqref{E:cmptv.W}. Let $\cmpt_f\in G(\AFf)$ be the finite part of $\cmpt$. According to our choices of $\cmptv$, we have
\[\cmpt_f*(\sL\ot_\Z\Zhat)=(\sL\ot_\Z\Zhat)\cdot\cmpt_f'=q_{\CMP}(\cO_\cK\ot_\Z\Zhat).\]
Define
$x:\AK^\x\to X^+\x G(\AFf)$ by
\[a=(a_\infty,a_f)\mapsto x(a):=(\CMP_\Sg,\iota(a_f)\cmpt_f).\]
Let $a\in (\AKf^{(p\None)})^\x$ and let
\[(\ul{A}(a),\lp(a))_{/\C}=(\EucA_{\iota(a)\cmpt_f}(\CMP_\Sg),\pairing_\can,\iota_\can,\lpp(a),\lp(a))\] be the $\C$-quintuple associated to $x(a)$ as in \secref{S:cpx}.
The alternating pairing $\pairing:\cK\x\cK:\to\cF$
defined by $\pair{x}{y}=(\ol{x}y-x\ol{y})/(\CMP-\ol{\CMP})$ induces an isomorphism
$\OK\wedge_{\OF} \OK=\frakc(\OK)^{-1}\cD_\cF^{-1}$ for the fractional
ideal $\frakc(\OK)=\cD_\cF^{-1}((\CMP-\ol{\CMP}) \cD_{\cK/\cF}^{-1})$. The hypothesis (d2) on $\CMP$ implies that \[\text{$\frakc(\OK)$ is prime to $p\frakc_\chi\frakn D_{\cK/\cF}$}.\]
Note that $\frakc(\OK)$ descends to a fractional ideal of $\OF$ and that $\frakc(\OK)$ is the polarization of $x(1)=(\ul{A}(1),\lp(1))$.
In addition, $x(a)=(\ul{A}(a),\lp(a))_{/\C}$ is an
abelian variety with CM by $\cO_\cK$ with the polarization ideal of $x(a)$ given by \[\frakc(a):=\frakc(\OK)\rmN(\fraka)^{-1}\quad(\fraka=\il_\cK(a)).\]
It thus gives rise to a complex point $[x(a)]$ in $\Ig_\opcpt(\plideal(a))(\C)$. Let $\CMring$ be the \padic completion of the maximal unramified extension of $\Zp$ in $\Cp$. The general theory of CM abelian varieties shows that $[x(a)]$ indeed descends to a point in $\Ig_{\opcpt}(\frakc(a))(\CMring)\hookto\Ig_{\opcpt}(\CMring)$, which is still denoted by $x(a)$. The collection $\stt{[x(a)]}_{a\in (\AKf^{(p\None)})^\x}\subset \Ig_{\opcpt}(\CMring)$ are called \emph{CM points} in Hilbert modular Shimura varieties.


%% file: Waldvmu.tex
\newcommand\mfFC[1]{\bfa_\beta(\bff_{#1}^*,\frakc)}
\def\chpi{\varPsi_{\pi,\lam,v}}

\section{Anticyclotomic \padic Rankin-Selberg $L$-functions}\label{S:anticy_padic_L}
\subsection{Toric forms}
\def\CLKF{Cl_-}
\begin{defn}[The toric form]\label{D:toricform}We define the complex Hilbert modular form $\bff_\chi:X^+\x G(\AFf)\to\C$ associated to $\vp_\chi$ by
\beq\begin{aligned}\bff_\chi(\tau,g_f)=&
\vp_\chi(g)\cdot \ul{J}(g_\infty,\bfi)^k(\det g_\infty)^{-\frac{(k_{mx}-2)\Sg+k}{2}}\abs{\det g}^{k_{mx}/2-1}_{\AF},\\
&\quad (\bfi=(\sqrt{-1})_{\sg\in\Sg},\,g=(g_\infty,g_f),\,g_\infty \bfi=\tau,\,\det g_\infty>0).
\end{aligned}\eeq
Here $\det g_\infty=(\det g_{\sg})_{\sg\in\Sg}\in (\R^\x)^\Sg$ and $\det g_\infty>0$ means $\det g_\sg>0$ for all $\sg\in\Sg$. 

Let $\bff_\chi^{*}$ be the normalization of $\bff_\chi$ given by
\[\bff_\chi^{*}= N(\pi,\chi)^{-1}\abs{\det\cmpt_f}_{\AFf}^{1-k_{mx}/2}\cdot \bff_\chi.\]
Let $\delta_k^m$ be the Shimura-Maass differential operator (\cf\cite[(1.21)]{HidaTilouine:KatzPadicL_ASENS}). Then the normalized differential operator $\wtd V_+^m$ defined in \eqref{E:Shi_Maass} is the representation theoretic avatar of $\delta_k^m$ in the following sense:
\[\delta_\wt^m\bff_\chi(\tau,g_f)=(\wtd V_+^m\vp_\chi)(g_\infty,g_f)\ul{J}(g_\infty,i)^{k+2m}(\det g_\infty)^{1-\frac{k_{mx}\Sg+k}{2}-m}\abs{\det g}^{k_{mx}/2-1}_{\AF}\]
(\cf \cite[\S 4.5]{Hsieh:Hecke_CM}). We call $\delta_\wt^m\bff^*_\chi$ the normalized toric form of character $\chi$.
\end{defn}
Similarly, for each $u\in (\cO_\cF\ot_\Z\Zp)^\x$, we let $\bff_{\chi,u}^*$ be the normalized modular form associated to the $u$-component $\vp_{\chi,u}$ (\cf $\localW{\chi,u}$ in \eqref{E:Wu.W}). It is clear from \eqref{E:ucompnent.W} that
\beq\label{E:13.W}\bff_\chi^{*}=\sum_{u\in\cU_p}\bff^*_{\chi,u}.\eeq
Let $\lsgN$ be the open-compact subgroup defined in \eqref{E:opncpt.W}. Then $\bff_\chi^*$ and $\stt{\bff_{\chi,u}^*}_{u\in\cU_p}$ belong to $\bfM_k(\lsgN,\C)$ for sufficiently large $n$.

 For $a\in(\AKf^\setp)^\x \x(\cO_\cK\ot\Zp)^\x$, we consider the Hecke action $|[a]$ given by
\begin{align*}
|[a]:\bfM_k(\frakc(a),\lsgN,\C)&\to\bfM_k(\frakc,{}_a\lsgN,\C)\quad({}_a\lsgN:=\iota_\cmpt(a)\lsgN\iota_\cmpt(a^{-1})), \\
\bff&\mapsto\bff|[a](\tau,g_f):=\bff(\tau,g_f\iota_\cmpt(a)).\end{align*}
The Hecke action $|[a]$ can be extended to the spaces of $p$-integral modular forms (\cf\cite[\S 2.6]{Hsieh:VMU}). It follows from \lmref{L:4.W} immediately that
\beq\label{E:17.W}\bff^*_{\chi,u}|[a]=\chi^{-1}\Abs^{k_{mx}/2-1}_{\AK}(a)\cdot \bff^*_{\chi,u.a^{1-c}}\text{ for all }a\in\cT_f\quad(u.a^{1-c}:=ua_{\Sg_p}a_{\Sgbar_p}^{-1}).\eeq
\subsection{The toric period integral}
Next we consider the toric period integral of $\bff_\chi^*$. Let $U_\cK=(\cK\ot_\Q\R)^\x\x(\OK\ot_\Z\Zhat)^\x$ be a subgroup of $\AK^\x$ and let $\CLKF=\cK^\x\AF^\x\bksl \AK^\x/U_\cK$. Let $\cR$ be the subgroup of $\AK^\x$ generated by $\cK_v^\x$ for all ramified places $v$ and let $\CLKF^\alg$ be the subgroup of $\CLKF$ generated by the image of $\cR$ By \lmref{L:4.W} and the fact that $\cT=\AF^\x U_\cK\cR $, we have
\beq\label{E:5.W}P_\chi(\pi(\cmpt)\wtd V_+^m\vp_\chi)
=\vol(U_\cK,dt)\#(\CLKF^\alg)\cdot \sum_{[t]\in\CLKF/\CLKF^\alg}\wtd V_+^m\vp_\chi(\iota(t)\cmpt)\chi(t).
\eeq
Let $\cD_1$ be a set of representatives of $\CLKF/\CLKF^\alg$ in $(\AKf^{(p\None)})^\x$. We define the $\chi$-isotypic toric period by
\[P_\chi(\delta_k^m\bff_\chi^{*}):=\sum_{a\in \cD_1}\delta_k^m\bff_\chi^{*}(x(a))\chi\Abs_{\AK}^{1-k_{mx}/2}(a).\]
\begin{prop}\label{P:1.W}Let $D_{\cK/\cF}$ be the discriminant of $\cK/\cF$. We have
\[P_\chi(\delta_k^m\bff_\chi^{*})^2=[\cO_\cK^\x:\cO_\cF^\x]^2\cdot\frac{\Gamma_\Sg(k+m)\Gamma_\Sg(m+1)}{(\Im\CMP)^{k+2m}(4\pi)^{2m+k+1}}\cdot L(\onehalf,\pi_\cK\ot\chi)\cdot E_{\Sg_p}(\pi,\chi)\cdot C(\pi,\chi),\]
where
\[C(\pi,\chi)= C'(\pi,\chi)\cdot 4^{-[F:\Q]}\abs{\rmN_{\cF/\Q}(D_{\cK/\cF})}_\R^{\onehalf}\left(\frac{\#(\CLKF)h_\cF}{\#(\CLKF^\alg)h_\cK}\right)^2\in\Zbar_\setp^\x.\]
\end{prop}
\begin{proof}
By definition, we have
\[\bff_\chi(x(a))=\vp_\chi(\iota(a_f)\cmpt)(\Im \CMP)^{-k/2}\cdot\abs{\rmN(a)\det\cmpt_f}_{\AF}^{k_{mx}/2-1}.\]
By \eqref{E:5.W}, we find that
\begin{align*}
\vol(U_E,\dx t)\#(\CLKF^\alg)\cdot P_\chi(\delta_\wt^m\bff_\chi^{*})
&=\frac{1}{N(\pi,\chi)\cdot (\Im\CMP)^{k/2+m}}\cdot P_\chi(\pi(\cmpt)\wtd V_+^m\vp_\chi).
\end{align*}
From the well-known formula
\[2L(1,\tau_{\cK/\cF})=(2\pi)^{[\cF:\Q]}\cdot\frac{h_{\cK}/h_\cF}{\abs{D_\cK}_\R^\onehalf\abs{D_\cF}_\R^{-\onehalf}\cdot [\cO_\cK^\x:\cO_\cF^\x]},\]
we see that
\begin{align*}\vol(U_\cK,dt)=&\vol(\cK^\x\AF^\x\bksl \AK^\x,dt)\cdot \#(\CLKF)^{-1}\\
=&2\pi^{-[\cF:\Q]}L(1,\tau_{\cK/\cF})\cdot \#(\CLKF)^{-1}=\frac{2^{[\cF:\Q]}\abs{D_\cF}_\R^\onehalf}{\abs{D_{\cK}}_\R^\onehalf [\cO_\cK^\x:\cO_\cF^\x]}\cdot\frac{h_\cK}{h_\cF\#(\CLKF)}.\end{align*}
The proposition follows form \thmref{T:main1.W} immediately. Note that the ratio $\frac{\#(\CLKF)h_\cF}{\#(\CLKF^\alg)h_\cK}$ is a power of $2$, so the constant $C(\pi,\chi)$ is a \padic unit.
\end{proof}

\subsection{The Fourier expansion of $\bff_{\chi,u}^*$}
Let $u=(u_v)\in\cU_p$. We give an expression of the Fourier expansion of $\bff_{\chi,u}^{*}$. Let $\localW{\chi,u,f}$ be the finite part of $\localW{\chi,u}$. By the definition of $\bff_{\chi,u}$, we have
\beq\label{E:12.W}\begin{aligned}
\bff_{\chi,u}(\tau,g_f)&=\sum_{\beta\in\cF}\localW{\chi,u,f}(\DII{\beta}{1}g_f)\localW{\chi,\infty}(\DII{\beta}{1}\MX{y_\infty}{x_\infty}{0}{1})\cdot y_\infty^{-k/2}\\
&=\sum_{\beta\in\cF_+}\localW{\chi,u,f}(\DII{\beta}{1}g_f)\beta^{k/2}e^{2\pii\Tr_{\cF/\Q}(\beta\tau)}.\\
&\quad(\tau=x_\infty+iy_\infty=(x_\sg+iy_\sg)_{\sg\in\Sg}\in X^+)
\end{aligned}\eeq
The second equality follows from the choice of Whittaker functions at the archimedean places \eqref{E:archiWhittaker.W}.

We define the global prime-to-$p$ Fourier coefficient $\localK{\chi}^\setp:(\AFf^\setp)^\x\to\C$ by
\beq\label{E:FC.W}\begin{aligned}\localK{\chi}^\setp(a):=&N(\pi,\chi)^{-1}\cdot \localW{\chi,f}^\setp(\DII{a}{1})\quad (a=(a_v)\in\AFf^\x)\\
=&\prod_{v\in B(\chi)}\frac{1}{n_vL(1,\tau_{\cK_v/\cF_v})}\localW{\chi,v}(\DII{a_v}{1})\prod_{v\not\in B(\chi),v\ndivide p}\localW{\chi,v}(\DII{a_v}{1})\\
=&\prod_{v\in\bdh,v\ndivide p}\localK{\chi,v}^*(a_v).
\end{aligned}
\eeq
Here $\localK{\chi,v}^*$ are the local Fourier coefficients defined in \defref{D:localFourier.W}.
\begin{prop}\label{P:2.W} Let $\frakc$ be a prime-to-$p$ ideal of $\cF$ and let $\bfc\in (\AKf^{\setp})^\x$ such that $\il_\cF(\bfc)=\frakc$. Then the Fourier expansion of $\bff_{\chi,u}^{*}$ at the cusp $(\OF,\frakc)$ is given by
\[\bff_{\chi,u}^{*}|_{(\OF,\frakc)}(q)=\sum_{\beta\in (\None^{-1}\frakc)_+}\bfa_\beta(\bff_{\chi,u}^{*},\frakc)q^\beta,\]
where
\[\bfa_\beta(\bff_{\chi,u}^{*},\frakc)=\beta^{k/2}\localK{\chi}^\setp(\beta\bfc^{-1})\prod_{w\in\Sgbar_p,v|w}\chi_{w}(\beta^{-1})\bbI_{u_v(1+\uf_v\cO_{\cF_v})}(\beta).\]
In particular, $\bff_{\chi,u}^*\in \bfM_k(\lsgN,\EucO)$ by \propref{P:4.W}, and the Fourier expansion of $\bff_\chi^*$ at the cusp $(\OF,\frakc)$ is given by
\[\bff_{\chi}^{*}|_{(\OF,\frakc)}(q)=\sum_{\beta\in(\None^{-1}\frakc)_+}\bfa_\beta(\bff_{\chi}^{*},\frakc)q^\beta,\]
where
\[\bfa_\beta(\bff_{\chi}^{*},\frakc)=\beta^{k/2}\localK{\chi}^\setp(\beta\bfc^{-1})\prod_{w\in\Sgbar_p}\chi_{w}(\beta^{-1})\cdot \bbI_{\cO_{\cF,\setp}^\x}(\beta).\]
\end{prop}
\begin{proof}It follows from the definition of $\localW{\chi,u}$ that
\begin{align*}\localW{\chi,u,f}(\DII{\beta\bfc^{-1}}{1})=&\localW{\chi,f}^\setp(\DII{\beta\bfc^{-1}}{1})\cdot\prod_{v|p}\localW{\chi,u_v,v}(\DII{\beta}{1})\\
=&\localW{\chi,f}^\setp(\DII{\beta\bfc^{-1}}{1})\cdot \prod_{w\in\Sgbar_p,v|w}\chi_{w}^{-1}(a_v)\bbI_{u_v(1+\uf_v\cO_{\cF_v})}(\beta).
\end{align*}
The proposition follows from \eqref{E:12.W} immediately. The Fourier expansion of $\bff_\chi^*$ follows from \eqref{E:13.W}.
\end{proof}
\def\OFp{O_p}
\subsection{\padic $L$-functions}\label{SS:padicL}
We go back to the setting in the introduction. Let $\cK^-_{p^\infty}$ be the maximal anticyclotomic $\Zp^{[\cF:\Q]}$-extension of $\cK$ and let $\IwGamma=\Gal(\cK^-_{p^\infty}/\cK)$. The reciprocity law $\rec_\cK$ at $\Sgbar_p$ induces a morphism \[\rec_{\Sgbar_p}\colon (\cF\ot_\Q\Qp)^\x\iso\prod_{w\in\Sgbar_p}\cK_{w}^\x\stackrel{\rec_\cK}\longto \Gamma^-.\]
Let $\wtsp$ be the set of critical specializations, consisting of \padic characters $\phi:\IwGamma\to\Cp^\x$ such that for some $m\in\Z_{\geq 0}[\Sg]\iso\Z_{\geq 0}[\Sg_p]$,
\[\phi(\rec_{\Sgbar_p}(x))=x^m\text{ for all $x\in(\cO_\cF\ot\Zp)^\x$ sufficiently close to $1$.}\]
Let $\phi$ be an anticyclotomic Hecke character of $p$-power conductor and of infinity type $(m,-m)$ with $m\in\Z_{\geq 0}[\Sg]$. Then $\phi$ is unramified outside $p$ and $\phi|_{\AF^\x}=1$.
The \padic avatar $\wh\phi$ of $\phi$ belongs to $\wtsp$. To be precise, let $\phi_{\Sgbar_p}:=\prod_{w\in\Sgbar_p}\phi_w$. Then we have \beq\label{E:4.W}\wh\phi(\rec_{\Sgbar_p}(x))=\phi_{\Sgbar_p}(x)x^{-m}\text{ for every }x\in (\cF\ot_\Q\Qp)^\x.\eeq Hereafter, we let $\lam$ be a Hecke character of $\cK^\x$ and assume that \hypref{H:HypA} and \eqref{sf} hold for $(\pi,\lam)$. Note that \hypref{H:HypA} and \eqref{sf} also hold for $(\pi,\lam\phi)$.
We will apply our calculations in \subsecref{S:toric} to the pair $(\pi,\chi)=(\pi,\lam\phi)$.
\begin{lm}\label{L:2.W} Let $\phi$ be as above. Then
\begin{mylist}
\item $\localK{\lam\phi}^\setp=\localK{\lam}^\setp$.
\item $C'(\pi,\lam\phi)=C'(\pi,\lam)\phi(\Csplit)$.
\end{mylist}
\end{lm}
\begin{proof}If $v\ndivides p$ is split, we have remarked that $\localW{\chi_v\phi_v^{-1},v}=\localW{\chi_v,v}$. If $v$ is inert or ramified, then $\phi_v=1$ as $\phi_v$ is unramified and $p>2$. Therefore, we have $\localW{\chi\phi,f}^\setp=\localW{\chi,f}^\setp$. Part (1) follows from the definition of $\localK{\lam\phi}^\setp$ \eqref{E:FC.W} immediately.
Next, recall that we have defined $C'(\pi,\chi)$ for a Hecke character $\chi$ in \eqref{E:constantC.W}. Since $\phi$ is anticyclotomic and unramified outside $p$, part (2) follows from the well-known fact that
\[\ep(\onehalf,\pi_v\ot\lam_w\phi_w,\addchar_v)=\ep(\onehalf,\pi_v\ot\lam_w,\addchar)\phi_w(\cD_\cF^2\Csplit)\quad(v=w\wbar,\,w|\Csplit).\qedhere\]
\end{proof}
Let $\OFp:=\cO_\cF\ot_\Z\Zp$ and let $\Gamma':=\rec_{\Sgbar_p}(1+p\OFp)$ be an open subgroup of $\Gamma^-$. Let $\stt{\theta(\sg)}_{\sg\in\Sg}$ be the Dwork-Katz \padic differential operators (\cite[Cor.\,(2.6.25)]{Katz:p_adic_L-function_CM_fields}) and let $\theta^m:=\prod_{\sg\in\Sg}\theta(\sg)^{m_\sg}$.
\begin{prop}\label{P:3.W}There exists a unique $V(\frakc,\opcpt,\Zbarp)$-valued \padic measure $\EucF_{\lam,\frakc}$ on $\Gamma^-$ such that
\begin{itemize}\item[(i)] $\EucF_{\lam,\frakc}$ is supported in $\Gamma'$,
\item[(ii)] for every $\wh\phi\in\wtsp$ of weight $(m,-m)$, we have
\[\int_{\IwGamma}\wh\phi d\EucF_{\lam,\frakc}=\theta^m\wh\bff_{\lam\phi,\frakc}^*.\]
\end{itemize}
\end{prop}
\begin{proof}
We denote by $\EucF_{\lam,\frakc}(q)$ the \padic measure with values in the space of formal $q$-expansions such that for every $\vp\in C(\Gamma^-,\Zbarp)$,
\[\int_{\IwGamma}\vp d\EucF_{\lam,\frakc}(q)=\sum_{\beta\in(\None^{-1}\frakc)_+}\mfFC{\lam}\vp(\rec_{\Sgbar_p}(\beta^{-1}))q^\beta.\]
Note that $\mfFC{\lam}=0$ unless $\beta\in\cO_{\cF,\setp}^\x$ and that $\EucF_{\lam,\frakc}$ has support in $\Gamma'$ by definition.

Let $\wh\phi$ be the \padic avatar of a Hecke character $\phi$ of infinity type $(m,-m)$. By \cite[(2.6.27)]{Katz:p_adic_L-function_CM_fields} (\cf\cite[\S 1.7 p.205]{HidaTilouine:KatzPadicL_ASENS}), the $q$-expansion of $\theta^m\wh\bff_{\lam\phi}^{*}$ is given by
\begin{align*}\theta^m\wh\bff_{\lam\phi}^{*}|_{(\OF,\frakc)}(q)=&\sum_{\beta\in (\None^{-1}\frakc)_+}\mfFC{\lam\phi}\phi_{\Sgbar_p}(\beta^{-1})\beta^m q^\beta. \end{align*}
Therefore, by \lmref{L:2.W} and \eqref{E:4.W} we find that
\beq\label{E:53.N}\int_{\IwGamma}\wh{\phi} d\EucF_{\lam,\frakc}(q)=\theta^m\wh\bff_{\lam\phi,\frakc}^*(q).\eeq
By the $q$-expansion principle, this measure descends to the \padic measure $\EucF_{\lam,\frakc}$ with values in the space of \padic modular forms $V(\frakc,K,\Zbarp)$.
\end{proof}
Let $\cP_\Sg(\pi,\lam)$ be the \padic measure on $\IwGamma$ such that for each $\vphi\in\cC(\IwGamma,\Zbarp)$,
\beq\label{E:period_measure.E}\int_{\IwGamma}\vphi d\cP_\Sg(\pi,\lam)=\sum_{a\in\cD_1}\wtd\lam(a)\int_{\IwGamma}\vphi|[a] d\EucF_{\lam,\frakc(a)}(x(a))\quad(\wtd\lam=\lam\cdot\Abs_{\AK}^{1-k_{mx}/2}).\eeq
Here $\vphi|[a](x):=\vphi(x\rec_{\cK^-_\infty/\cK}(a))$. Let $(\Omega_\infty,\Omega_p)\in(\C^\x)^\Sg\x(\Zbarp^\x)^\Sg$ be the complex and \padic CM periods of $(\cK,\Sg)$ introduced in \cite[(4.4 a,b) p.211]{HidaTilouine:KatzPadicL_ASENS} (\cf $(\Omega,c)$ in \cite[(5.1.46), (5.1.48)]{Katz:p_adic_L-function_CM_fields}) and let $\Omega_\cK=(2\pii)^{-1}\Omega_\infty$. We have the following evaluation formula of $\cP_\Sg(\pi,\lam)$.
\begin{thm}\label{T:padicL.W} Suppose that \hypref{H:HypA} and \eqref{sf} hold. Then for each \padic character $\wh\phi\in\wtsp$ of weight $(m,-m)$, we have the evaluation formula
\begin{align*}\left(\frac{1}{\Omega_p^{\wt+2m}}\int_{\IwGamma}\wh\phi d\cP_\Sg(\pi,\lam)\right)^2=&[\cO_\cK^\x:\cO_\cF^\x]^2\cdot \frac{\Gamma_\Sg(\wt+m)\Gamma_\Sg(m+1)}{(\Im\CMP)^{k+2m}(4\pi)^{\wt+2m+1}}\\
&\times E_{\Sg_p}(\pi,\lam\phi)\cdot \frac{L(\onehalf,\pi_\cK\ot\lam\phi)}{\Omega_\cK^{2(\wt+2m)}}\cdot \phi(\Csplit)C(\pi,\lam).\end{align*}
\end{thm}
\begin{proof}
From \cite[(2.4.6), (2.6.8), (2.6.33)]{Katz:p_adic_L-function_CM_fields} we can deduce that
\[\frac{1}{\Omega_p^{\wt+2m}}\theta^m\wh\bff_{\lam\phi}^{*}(x(a))=\frac{1}{\Omega_\cK^{\wt+2m}}\delta_\wt^m\bff^*_{\lam\phi}(x(a)).\]
Therefore, we have
\begin{align*}
\frac{1}{\Omega_p^{\wt+2m}}\cdot \int_{\IwGamma}\wh\phi d\cP_\Sg(\pi,\lam)=&\sum_{a\in\cD_1}\wtd\lam(a)\phi(a)\cdot\frac{1}{\Omega_p^{\wt+2m}}\theta^m\wh\bff_{\lam\phi}^{*}(x(a))\\
=&\frac{1}{\Omega_\cK^{\wt+2m}}\cdot\sum_{a\in\cD_1}\wtd\lam\phi(a)\delta_\wt^m\bff^*_{\lam\phi}(x(a))=\frac{1}{\Omega_\cK^{\wt+2m}}\cdot P_{\lam\phi}(\delta_\wt^m\bff^*_{\lam\phi}).
\end{align*}
Combined with \propref{P:1.W} and \lmref{L:2.W} (2), the above equation yields the proposition.
\end{proof}

\section{The $\mu$-invariant of \padic $L$-functions}\label{S:mu_invariant}

\def\Fadsu{\vp_{\pi,\lam,u}}
In this section, we use the explicit computation of Fourier coefficients of $\stt{\bff_{\lam,u}^*}_{u\in\cU_p}$ to study the $\mu$-invariant of the \padic measure $\cP_\Sg(\pi,\lam)$ by the approach of Hida \cite{Hida:mu_invariant}.
\subsection{The $t$-expansion of \padic modular forms} We begin with a brief review of the $t$-expansion of \padic modular forms. A functorial point in $\Ig_K(\frakc)$ can be written as $[(\ulA,\lp)]=[(A,\lam,\iota,\ol{\eta}^\setp,\lp)]$. Enlarging $\CMring$ if necessary, we let $\CMring$ be the \padic ring generated by the values of $\lam$ on finite ideles over the Witt ring $W(\Fpbar)$. Let $\frakm_\CMring$ be the maximal ideal of $\CMring$ and fix an isomorphism $\CMring/\frakm_\CMring\isoto\Fpbar$.
Let $T:=\OF^*\ot_\Z\bbmu_{p^\infty}$ and let $\wh T=\dirlim_m T_{/\CMring/\frakm_\CMring^m}=\OF^*\ot_\Z\formal{\bbG}_m$. Let $\stt{\xi_1\cdots,\xi_d}$ be a basis of $\OF$ over $\Z$ and let $t$ be the character $1\in \OF=X^*(\OF^*\ot_\Z\Gm)=\Hom(\OF^*\ot_\Z\Gm,\Gm)$. Then we have $\cO_{\wh T}\isoto \CMring\powerseries{t^{\xi_1}-1,\cdots t^{\xi_d}-1}$.
For $y=(\ul{A}_y,\eta_y)\in\Ig_K(\frakc)(\Fpbar)\subset \Ig_K(\Fpbar)$, it is well known that the deformation space $\wh S_{y}$ of $y$ is isomorphic to the formal torus $\wh T$ by the theory of Serre-Tate coordinate (\cite{Katz:ST}). The $p^\infty$-level structure $\lp_y$ of $A_y$ induces a canonical isomorphism $\vphi_{y}\colon\wh{T}_{/\baseR}\isoto \wh{S}_{y}=\Spf \wh\cO_{\Ig_\opcpt(\frakc),y}$ (\cf \cite[(3.15)]{Hida:mu_invariant}).

Now let $\bfx:=x(1)_{/\CMring}\in\Ig_K(\frakc)(\CMring)$ be a fixed CM point of type $(\cK,\Sg)$ and let $x_0=\bfx\ot_{\CMring}\Fpbar=(\ul{A}_0,\lp_0)$.
For a deformation $z=(\ul{A},\lp)_{/\cR}$ of $x_0$ over an artinian local ring $\cR$ with the maximal ideal $\frakm_\cR$ and the residue field $\Fpbar$, we let $t(\ul{A},\lp):=t(\vphi_{x_0}^{-1}((\ul{A},\lp)_{/\cR}))\in 1+\frakm_\cR$. Then $\bfx$ is the canonical lifting of $x_0$, \ie $t(\bfx)=1$. For $f\in V(\frakc,\opcpt,\CMring)$, we define \[f(t):=\vphi_{x_0}^*(f)\in\cO_{\wh T}=\CMring\powerseries{T_1,\cdots T_d}\quad(T_i=t^{\xi_i}-1).\]
The formal power series $f(t)$ is called the \emph{$t$-expansion} around $x_0$ of $f$. For each $u\in\OFp^\x$, let $uz:=(\ul{A},u\lp)$ is a deformation of $ux_0$. Then we have $t(uz)=t(z)^u$ and hence $\vphi_{ux_0}^*(f)(t)=\vphi_{x_0}^*(f)(t^u)=f(t^u)$.
\subsection{The vanishing of the $\mu$-invariant}
Let $\pi_-:(\AKf^{(p\None)})^\x\to \Gamma^-$ be the natural map induced by the reciprocity law. Let $Z'=\pi_-^{-1}(\Gamma')$ be a subgroup of $(\AKf^{(p\None)})^\x$ and let $Cl'_-\supset Cl^\alg_-$ be the image of $Z'$ in $Cl_-$. Let $\cD_1'$ (resp. $\cD_1''$) be a set of representative of $Cl_-'/Cl_-^\alg$ (resp. $Cl_-/Cl_-'$) in $(\AKf^{(p\None)})^\x$. Let $\cD_1:=\cD_1'\cD_1''$ be a set of representative of $Cl_-/Cl_-^\alg$. Recall that $\cU_p$ is the torsion subgroup of $\Op^\x$. Let $\cU$ be the torsion subgroup of $\cK^\x$ and let $\cU^\alg=(\cK^\x)^{1-c}\cap \OK^\x$ be a subgroup of $\cU$. We regard $\cU^\alg$ as a subgroup of $\Op$ by the imbedding induced by $\Sgbar_p$. Let $\cD_0$ be a set of representatives of $\cU_p/\cU^\alg$ in $\cU_p$.

Fix $\frakc:=\frakc(\OK)$ to be the polarization ideal of the CM point $x(1)$. The following theorem reduces the calculation of the $\mu$-invariant $\mu^-_{\pi,\lam,\Sg}$ to the determination of the $q$-expansion of $\bff^*_{\lam,u}$.
\begin{thm}\label{T:linear.V}Suppose that $p$ is unramified in $\cF$. Then
\[\mu^-_{\pi,\lam,\Sg}= \inf_{\substack{(a,u)\in \cD_1\x\cD_0 \\ \beta\in\cF_+}} v_p(\bfa_\beta(\bff^*_{\lam,u},\frakc(a))).\]
\end{thm}
\begin{proof}For every pair $(u,a)\in\cU_p\x \cD_1$, we let $\bff^*_{u,a}:=\bff^*_{\lam,u}|_{\frakc(a)}\in\cM_k(\frakc(a),K,\EucO)$. Let $\EucF_{u,a}$ be the \padic avatar of $\bff^*_{u,a}$. Fix a sufficient large finite extension $L$ over $\Qp$ so that $\ads$ and $\bff^*_{u,a}|[a]$ are defined over $\cO_L$ for all $(u,a)$, and hence $\EucF_{u,a}|[a]\in V(\frakc,\cO_L)$.  For each $z\in Z'$, let $\Dmd{z}$ be the unique element in $1+p\Op$ such that $\rec_{\Sgbar_p}(\Dmd{z})=\pi_-(z)\in\Gamma^-$. For $(a,b)\in\cD_1\x\cD_1''$, we define
\begin{align*}\wtd\EucF_a(t)&=\sum_{u\in \torsbgp}\EucF_{u,a}(t^{u^{-1}}),\\
\EucF^b(t)&=\sum_{a\in b\cD_1}\wtd\lam(ab^{-1})\wtd\EucF_a|[a](t^{\Dmd{ab^{-1}}}).\end{align*}
Let $\cP^b_{\Sg}(\pi,\lam)$ be the \padic measure on $1+p\OFp\iso\Gamma'$ obtained by the restriction of $\cP_\Sg(\pi,\lam)$ to $\pi_-(b)\Gamma'$. In other words,
\begin{align*}\int_{\Gamma'}\vp d\cP^b_\Sg(\pi,\lam):=&\int_{\Gamma^-}\bbI_{b.\Gamma'}\cdot\vp|[b^{-1}]d\cP_\Sg(\pi,\lam)\\
=&\sum_{a\in b\cD_1'}\wtd\lam(a)\int_{\Gamma^-}\vp|[ab^{-1}]d\EucF_{\lam,\frakc(a)}(x(a)).
\end{align*}
Here the second equality follows from the fact that $\EucF_{\lam,\frakc(a)}$ has support in $\Gamma'$ (\propref{P:3.W} (i)). The argument of \cite[Prop.\,5.2]{Hsieh:VMU} shows that $\EucF^b(t)$ is the power series expansion of the measure $\cP^b_\Sg(\pi,\lam)$ regarded as a \padic measure on $\OFp$ and that
\begin{align*}\mu^-_{\pi,\lam,\Sg}&=\inf_{b\in\cD_1''}\mu(\EucF^b),\text{ where}\\
&\mu(\EucF^b):=\inf\stt{r\in\Q_{\geq 0}\mid p^{-r}\EucF^b\not\con 0\pmod{\frakm_p}}.\end{align*}
By \eqref{E:17.W} we find that
\[\wtd\EucF_a(t)=\#(\cU^{\alg})\cdot \sum_{u\in\cD_0}\EucF_{u,a}(t^{u^{-1}}),\]
and hence
\[\EucF^b(t)=\#(\cU^{\alg})\cdot \sum_{(u,a)\in \cD_0\x b\cD_1'}\wtd\lam(ab^{-1})\EucF_{u,a}|[a](t^{\Dmd{ab^{-1}} u^{-1}}).\]
Proceeding along the same lines in \cite[Thm.\,5.5]{Hsieh:VMU}, we can deduce the theorem from the above equation by the linear independence of \padic modular forms modulo $p$ acted by the automorphisms in $\cD_0\x \cD_1'$ (\cite[Thm.\,3.20, Cor.\,3.21]{Hida:mu_invariant}) and the $q$-expansion principle for \padic modular forms.
\end{proof}

\begin{thm}\label{T:mu-invariant}In addition to \hypref{H:HypA} and \eqref{sf}, we suppose that $p$ is unramified in $\cF$ and \beqcd{ai${}_\cK$}\text{the residual Galois representation $\ol{\rho}_p(\pi_\cK):=\rho_p(\pi)|_{G_\cK}\pmod{\frakm_p}$ is absolutely irreducible}.\eeqcd
Then $\mu^-_{\pi,\lam,\Sg}=0$ if and only if
\[\sum_{v|\frakc^-_\lam}\mu_p(\chpi)=0,\]
where $\mu_p(\chpi)$ are the local invariants defined as in \eqref{E:9.W}.
\end{thm}
\begin{proof}
It is not difficult to deduce from the formula of $\bfa_\beta(\bff^*_{\lam,u},\frakc(a))$ in \propref{P:2.W} and \propref{P:4.W} that
\[\mu_p(\chpi)>0\text{ for some $v|\frakc_\lam^-$}\imply \bfa_\beta(\bff^*_{\lam,u},\frakc(a))\con 0\pmod{\frakm}\text{ for all }a\in \AFf^\x,\]
and hence $\mu^-_{\pi,\lam,\Sg}>0$ by \thmref{T:linear.V}.

Conversely, we suppose that $\mu_p(\chpi)=0$ for all $v|\frakc^-_\lam$. We are going to show $\mu^-_{\pi,\lam,\Sg}=0$ by contradiction. Assume that $\mu^-_{\pi,\lam,\Sg}>0$. By \propref{P:2.W} \thmref{T:linear.V}, for each $a\in\AKf^{(p\None)}$ we find that
\begin{align*}&\bfa_\beta(\bff_{\lam,u}^*,\frakc(a))\con 0\pmod{\frakm}\text{ for all }u\in\cU_p\text{ and }\beta\in\cF_+\\
\iff&\localK{\lam}^\setp(\beta \bfc^{-1}\rmN(a^{-1}))\con 0\pmod{\frakm}\text{ for all }\beta\in\cO_{\cF,\setp}^\x.\end{align*}
Therefore, as a function on $(\AFf^\setp)^\x$, we have
\beq\label{E:8.W}\begin{aligned}\localK{\lam}^\setp(a)&\con 0\pmod{\frakm}\text{ for all }\\
&a\in \cO_{\cF,\setp}^\x \bfc^{-1} \det(U(\None))\rmN((\AKf^{(p\None)})^\x)=\cF^\x\bfc^{-1}\rmN((\AKf^\setp)^\x).\end{aligned}\eeq
By \propref{P:4.W}, there exists $\eta=(\eta_v)\in\prod_{v|\frakc^-_\chi}\cF^\x_v$ such that $\localK{\lam,v}^*(\eta_v)\not \con 0\pmod{\frakm}$ for each $v|\frakc_\lam^-$. We extend $\eta$ to be the idele in $\AFf^\x$ such that $\eta_v=1$ at $v\ndivides\frakc_\chi$. Therefore,
\eqref{E:8.W} together with the factorization formula of $\localK{\lam}^\setp$ \eqref{E:FC.W} imply that for each uniformizer $\uf_v$ at $v\ndivides p\Bad$, we have
\beq\label{E:14.W}\begin{aligned}\localK{\lam}^\setp(\eta\uf_v)&\con 0\pmod{\frakm}\iff W_v^0(\DII{\uf_v}{1})\con 0\pmod{\frakm}\text{ whenever }\\
&\uf_v\in [\eta^{-1}\bfc^{-1}]:=\cF^\x\eta^{-1}\bfc^{-1}\rmN((\AKf^\setp)^\x).\end{aligned}\eeq
On the other hand, by \eqref{E:Galois.W}, we find that
\[\Tr\rho_p(\pi)(\Frob_v)=\om(\uf_v)^{-1}\abs{\uf_v}^{-k_{mx}/2}W_v^0(\DII{\uf_v}{1})\text{ for all }v\ndivide p\frakn.\]
Let $\rec_{\cK/\cF}:\AK^\x\to \Gal(\cK/\cF)$ be the surjection induced by the reciprocity law. Combined with \eqref{E:14.W}, the above equation yields that
\begin{align*}\Tr\rho_p(\pi)(\Frob_v)&\con 0\pmod{\frakm}\text{ whenever }\\
&\Frob_v|_{\cK}=\rec_{\cK/\cF}(\uf_v)=\rec_{\cK/\cF}(\eta^{-1}\bfc^{-1}).\end{align*}
This in particular implies that $\rec_{\cK/\cF}(\eta^{-1}\bfc^{-1})$ must be the complex conjugation $c$, and hence we arrive at a contradiction to \eqref{ai${}_\cK$} by the following \lmref{L:3.W}.
\end{proof}
\begin{lm}\label{L:3.W}Let $p>2$ be a prime. Let $G$ be a finite group and $H\subset G$ be a index two subgroup. Let $\rho: G\hookto \GL_2(\Fpbar)$ be a faithful irreducible representation of $G$. Let $\rmT=\Tr\rho: G\to\Fpbar$ be the trace function. Assume that
\begin{mylist}\item There exists an order two element $c\in G-H$,
\item $\rmT(hc)=0$ for all $h\in H$.
\end{mylist}
Then $\rho|_H$ is reducible.
\end{lm}
\begin{proof}The assumption (2) implies that $\rmT(c)=0$, and hence $\det\rho(c)=-1$. We may assume that $\rho(c)=\MX{0}{1}{1}{0}$. Suppose that $p\ndivide \#(G)$. By the usual representation theory of finite groups, we have
\[1=\pair{\rmT}{\rmT}:=\frac{1}{\#(G)}\sum_{g\in G}\rmT(g)\rmT(g^{-1})=\frac{1}{2\#(H)}\sum_{h\in H}\rmT(h)\rmT(h^{-1})=\onehalf\cdot \pair{\rmT|_H}{\rmT|_H}.\]
Since $\pair{\rmT|_H}{\rmT|_H}=2$, we conclude that $\rho|_H$ is not irreducible.

Now we assume that $p\mid \#(H)$. For each $b\in M_2(\Fpbar)$ with $b^2=0$, define the $p$-subgroup $P_b$ of $\rho(H)$ by
\[P_b=\stt{h\in \rho(H)\mid h=1+xb\text{ for some }x\in\Fpbar}.\]
Let $h\in H$ be an element of $p$-power order. It is well known that $(\rho(h)-1)^2=0$, and hence
$\rmT(h)=2$ and $\det\rho(h)=1$. Combined with $\rmT(hc)=0$, these equations imply that
\[\rho(h)\in P_{b_1}\text{ or }P_{b_2}\text{ with }b_1=\MX{1}{1}{-1}{-1},\,b_2=\MX{1}{-1}{1}{-1}.\]

Note that either $P_{b_1}$ or $P_{b_2}$ is trivial. Indeed, if $h_1\not =1\in P_{b_1}$ and $h_2\not =1\in P_{b_2}$. Then $h_1h_2\in H$ and $\Tr(h_1h_2c)\not =0$, which is a contradiction.
In particular, we conclude that elements of $p$-power order in $H$ are commutative to each other and that there is only one $p$-Sylow subgroup of $H$, which we denote by $P$. It is clear that $H$ normalizes $P$. Since $P\not =\stt{1}$, there is a unique line fixed by $\rho(P)$, which is an invariant subspace of $\rho(H)$. We find that $\rho|_H$ is reducible if $p\mid \#(H)$.
\end{proof}
\begin{Remark}The assumption (3) in \thmref{T:B.W} in the introduction implies the vanishing of $\mu_p(\chpi)$ for all $v|\frakc_\lam^-$.
\end{Remark}

%% file: Waldnv.tex
\def\frakLbar{{\ol{\frakL}}}
\def\setn{{(n)}}
\def\chpi{\varPsi_{\pi,\chi,v}}
\section{Non-vanishing of central $L$-values with anticyclotomic twists}\label{S:NV}
In this section, we consider the problem of non-vanishing of $L$-values modulo $p$ with anticyclotomic twists. Let $\ell\not =p$ be a rational prime and let $\frakl$ be a prime of $\cF$ above $\ell$. Let $\Gamma_{\frakl}^-:=\Gal(\cK^-_{\frakl^\infty}/\cK)$ be the Galois group of the maximal anticyclotomic pro-$\ell$ extension $\cK^-_{\frakl^\infty}$ in the ray class field of $\cK$ of conductor $\frakl^\infty$. Let $\frakX_\frakl^0$ be the set consisting of finite order characters $\phi:\Gamma_\frakl^-\to\mu_{\ell^\infty}$. Fix a Hecke character $\chi$ of infinity type $(k/2+m,-k/2-m)$. For each $\phi:\Gamma_\frakl^-\to\mu_{\ell^\infty}$ in $\frakX_\frakl^0$, we put\[L^{\alg}(\onehalf,\pi_\cK\ot\chi\phi):=[\cO_\cK^\x:\cO_\cF^\x]^2\cdot \frac{\Gamma_\Sg(m)\Gamma_\Sg (k+m)}{(\Im\CMP)^{k+2m}(4\pi)^{k+2m+1\cdot\Sg}}\cdot\frac{L(\onehalf,\pi_\cK\ot\chi\phi)}{\Omega_\cK^{2(k+2m)}}.\]
For simplicity, we assume
\[ (p\frakl,\frakn D_{\cK/\cF})=1.\]
In particular, $\pi$ is unramified at $\frakl$ and every place above $p$. It can be shown that $L^{\alg}(\onehalf,\pi_\cK\ot\chi\phi)\in \Zbar_\setp$ at least when $p\ndivides D_{\cK}$. This section is devoted to proving the following result:
\begin{thm}\label{T:main3.W}With the same assumptions in \thmref{T:mu-invariant}, we further assume that
\begin{mylist}\item $(p\frakl,\frakn\frakc_\chi D_{\cK/\cF})=1$.
\item $\mu_p(\chpi)=0$ for all $v|\frakc_\chi^-$.
\end{mylist} Then for almost all $\phi\in\frakX_\frakl^0$ we have \[L^{\alg}(\onehalf,\pi_\cK\ot\chi\phi)\not\con 0\pmod{\frakm}.\]  Here almost all means "except for finitely many $\phi\in\frakX_\frakl^0$" if $\dim_{\Q_\ell} F_\frakl=1$ and "for $\phi$ in a Zariski dense subset of $\frakX_\frakl^0$" if $\dim_{\Q_\ell}F_\frakl>1$ $($See \cite[p.737]{Hida:nonvanishingmodp}$)$.
\end{thm}
When $\cF=\Q$, a non-primitive version of the above result under different assumptions is treated in \cite{Miljan:2}.
\subsection{}
After introducing some notation, we outline the approach of Hida \cite{Hida:nonvanishingmodp} to study this problem. We shall take $\Bad=\frakc_\chi\frakn D_{\cK/\cF}\frakl$ to be the fixed ideal in \subsecref{SS:choiceofcmpt}. For every $n\in\Z_{\geq 0}$, let $R_{\frakl^n}:=\cO_\cF+\frakl^n\cO_{\cK}$ be the order in $\cK$ of conductor $\frakl^n$.
Let $U_{\frakl^n}=(\cK\ot_\Q\R)^\x (R_n\ot_\Z\Zhat)^\x$ and let $Cl^-_{\frakl^n}:=\cK^\x\AF^\x\bksl \AK^\x/U_{\frakl^n}$ be the anticyclotomic ideal class group of conductor $\frakl^n$. Denote by $[\cdot]_n:\AK^\x\to Cl^-_{\frakl^n}$ the quotient map. Let $Cl^-_{\frakl^\infty}=\prolim_n Cl^-_{\frakl^n}$. Let $I_\frakl$ be the $\frakl$-adic Iwahori subgroup of $K^0_\frakl$ given by
\[I_\frakl=\stt{g=\MX{a}{b}{c}{d}\in \opcpt^0_\frakl\mid c\in \uf_\frakl\cD_{\cF_\frakl}}.\]
Let $\opcpt_0(\frakl):=K^\frakl I_\frakl=\stt{g\in \opcpt\mid g_\frakl\in I_\frakl}$ be an open-compact subgroup of $\GL_2(\AFf)$. Recall that the $U_\frakl$-operator on $\bfM_k(\opcpt_0(\frakl),\C)$ is given by
\[F\mid U_\frakl(\tau,g_f)=\sum_{u\in\cO_\cF/\frakl\cO_\cF}F(\tau,g_f\MX{\uf_\frakl}{ud_{\cF_\frakl}^{-1}}{0}{1}).\]
We briefly outline the approach of Hida to prove \thmref{T:main3.W} as follows:
\begin{mylist}
\item Construct a suitable \padic modular form $\wh\bff^\dagger_\chi$ which is an eigenfunction of $U_\frakl$-operator with \padic unit eigenvalue.
\item Consider Hida's measure $\vphi^\dagger_\chi$ on $Cl^-_{\frakl^\infty}$ attached to $\wh\bff^\dagger_\chi$ \eqref{E:lmeasure.W} and show the evaluation formula of this measure is related to central values $L^\alg(\onehalf,\pi_\cK\ot\chi\phi)$ (\propref{P:evaluation.W}).
    \item The Zariski density of CM points in Hilbert modular varieties modulo $p$ reduces the proof of \thmref{T:main3.W} to the non-vanishing of certain Fourier coefficients of $\wh \bff^\dagger_\chi$ at some cusp (\cite[Thm.\,3.2 and Thm.\,3.3]{Hida:nonvanishingmodp}).
\end{mylist}
We remind the reader that the proof is very close to \thmref{T:padicL.W} and \thmref{T:mu-invariant}. The essential new inputs in this section are the choice of $U_\frakl$-eigenforms and the computation of local period integral at $\frakl$.
\subsection{CM points of conductor $\frakl^n$} Let $n\in\Z_{\geq 0}$. We choose $\cmpt_\frakl^\setn\in G_\frakl$ as follows.
If $\frakl=\frakL\ol{\frakL}$ splits in $\cK$, writing $\CMP=\CMP_\frakL e_\frakL+\CMP_\frakLbar e_\frakLbar$ (so $d_{\cF_\frakl}=\CMP_\frakL-\CMP_\frakLbar$ is a generator of $\cD_{\cF_\frakl}$), we put
\begin{align*}\cmpt_\frakl^{(n)}=&\MX{\CMP_{\frakL}}{-1}{1}{0}\DII{\uf_\frakl^n}{1}.
\intertext{If $\frakl$ is inert, then we put}
\cmpt_\frakl^\setn=&\MX{0}{1}{-1}{0}\DII{\uf_\frakl^n}{1}.\end{align*}
Let $\cmpt^\setn:=\cmpt^\setn_\frakl\prod_{v\not=\frakl}\cmpt_v$. According to this choice of $\cmpt_\frakl^\setn$, we have
\[\cmpt_f^{(n)}*(\sL\ot_\Z\Zhat)=q_{\CMP}(R_{\frakl^n}).\]
 Define $x_n:\AK^\x\to X^+\x G(\AFf)$ by
\[x_n(a):=(\CMP_\Sg,a_f\cmpt_f^\setn).\]
This collection $\stt{x_n(a)}_{a\in\AK^\x}$ of points is called CM points of conductor $\frakl^n$. As discussed in \subsecref{S:CMpoint}, $\stt{x_n(a))}_{a\in(\AKf^\setp)^\x}$ descend to CM points in $\Ig_{K}(\CMring)$.
\subsection{The measures associated to $U_\frakl$-eigenforms}
We construct the $U_\frakl$-eigenform $\bff^\dagger_\chi$ as follows. Write $\pi_\frakl=\pi(\mu_\frakl,\nu_\frakl)$. Define the local Whittaker function $\localW{\frakl}^\dagger\in\cW(\pi_\frakl,\addchar_\frakl)$ by
\[\localW{\frakl}^\dagger(g)=W^0_{\frakl}(g)-\mu_\frakl\Abs^\onehalf(\frakl)W^0_\frakl(g\DII{\uf_\frakl^{-1}}{1}).\]
It is not difficult to verify that 
\begin{itemize} \item $\localW{\frakl}^\dagger$ is invariant by $I_\frakl$,
\item $\localW{\frakl}^\dagger(\DII{a}{1})=\nu_\frakl\Abs^\onehalf(a)\bbI_{\cO_{F_\frakl}}(a)$,
\item $\localW{\frakl}^\dagger$ is an $U_\frakl$-eigenfunction with the eigenvalue $\nu_\frakl(\uf_\frakl)\abs{\uf_\frakl}^{-\onehalf}$.
\end{itemize}
Define the normalized global Whittaker function $\localW{\chi}^\dagger$ by  \[\localW{\chi}^\dagger:=N(\pi,\chi)^{-1}\abs{\det\cmpt_f}_{\AFf}^{1-k_{mx}/2}\cdot \prod_{\sg\in\Sg}W_{k_\sg}\cdot \prod_{v\in\bdh,v\not =\frakl}\localW{\chi,v}\cdot \localW{\frakl}^\dagger,\]
 where $N(\pi,\chi)$ is the normalization factor in \eqref{E:constantN.W}.  Let $\vp_\chi^\dagger$ be the automorphic form associated to $\localW{\chi}^\dagger$ as in \eqref{E:WhittakerAuto.W} and let $\bff_\chi^\dagger$ be the associated Hilbert modular form as in \defref{D:toricform}.

 The following lemma follows from the choices of our Whittaker function $\localW{\chi}^\dagger$ and the construction of $\bff_\chi^\dagger$.
 \begin{lm}\label{L:6.W}Recall that $\cR$ is the group generated by $\cK_v^\x$ for all ramified places $v$ in $\AK^\x$. We have
 \begin{mylist}
  \item $\bff^\dagger_\chi$ is toric of character $\chi$ outside $\frakl$, and
\[\bff^\dagger_\chi(x_n(ta))=\bff^\dagger_\chi(x_n(t))\chi^{-1}\Abs_{\AK}^{k_{mx}/2-1}(a)\text{ for all }a\in\cR\cdot (R_{\frakl^n}\ot_\Z\Zhat)^\x.\]
 \item $\bff_{\chi\phi}^\dagger=\bff^\dagger_{\chi}$ for every $\phi\in\frakX_\frakl^0$.
\end{mylist}
\end{lm}
\begin{proof}Part (1) follows immediately from the fact that $\localW{\chi}^\dagger$ is a toric Whittaker function outside $\frakl$ in view of \lmref{L:toric.W}. In addition, for every $\phi\in\frakX_\frakl^0$, $\phi$ is anticyclotomic and unramified outside $\frakl$. We thus have $\localW{\chi}^\dagger=\localW{\chi\phi}^\dagger$, which verifies part (2) (\cf \lmref{L:2.W}).
\end{proof}
Following \cite[(3.9)]{Hida:nonvanishingmodp}, we define a \padic $\Zbarp$-valued measure $\vp^\dagger_\chi$ on $Cl^-_{\frakl^\infty}$ attached to the \padic avatar $\wh\bff^\dagger_\chi$ of $\bff^\dagger_\chi$ as follows. For a locally constant function $\phi:Cl^-_{\frakl^\infty}\to\Zbarp$ factors through $Cl^-_{\frakl^n}$, we define
\beq\label{E:lmeasure.W}\int_{Cl_{\frakl^\infty}^-}\phi d\vp^\dagger_\chi=\al_\frakl^{-n}\sum_{[a]_n\in Cl_{\frakl^n}^-}\theta^m\wh\bff^\dagger_\chi(x_n(a))\wh\chi(a)\phi([a]_n),\eeq
 where $\al_\frakl=\nu_\frakl(\uf_\frakl)\abs{\uf_\frakl}^\frac{1-k_{mx}}{2}$ and $\wh\chi$ is the \padic avatar of $\chi\Abs^{1-k_{mx}/2}_{\AK}$. One checks that the right hand side does not depend on the choice of $n$ since $\bff^\dagger_\chi$ is an $U_\frakl$-eigenform with the eigenvalue $\al_\frakl$.

Let $\phi\in\frakX^0_\frakl$ be a character of conductor $\frakl^n$. We view $\phi$ as a character on $Cl^-_{\frakl^n}$ by the reciprocity law. Following the arguments in \propref{P:1.W} and \thmref{T:padicL.W}, we can write the measure as a toric period integral of $\wtd V_+^m\vp_{\chi\phi}^\dagger$:
\beq\label{E:19.W}\begin{aligned}
\frac{\Omega_\cK^{k+2m}}{\Omega_p^{k+2m}}\cdot \int_{Cl^-_{\frakl^\infty}}\phi d\vp_\chi^\dagger&=\al_\frakl^{-n}\vol(U_{\frakl^n},dt)^{-1}\frac{1}{(\Im\CMP)^{k/2+m}}\cdot P_{\chi\phi}(\pi(\cmpt^\setn)\wtd V^m_+\vp^\dagger_{\chi})\\
&=\vol(U_\cK,dt)^{-1}\frac{1}{(\Im\CMP)^{k/2+m}}\cdot\frac{\al_\frakl^{-n}}{L(1,\tau_{\cK_\frakl/\cF_\frakl})\abs{\uf_\frakl}^n}\cdot P_{\chi\phi}(\pi(\cmpt^\setn)\wtd V^m_+\vp^\dagger_{\chi\phi}).
\end{aligned}\eeq
Here we used the fact that
\[\vol(U_{\frakl^n},dt)=\vol(U_\cK,dt)\cdot L(1,\tau_{\cK_\frakl/\cF_\frakl})\abs{\uf_\frakl}^{n}.\]
We have the following evaluation formula.
\begin{prop}\label{P:evaluation.W}Suppose that $(\frakl,\frakc_\chi\frakn D_{\cK/\cF})=1$. For $\phi\in\frakX_\frakl^0$ of conductor $\frakl^n$ with $n\geq 1$, we have
\[\left(\frac{1}{\Omega_p^{k+2m}}\cdot \int_{Cl_{\frakl^\infty}^-}\phi d\vp_\chi^\dagger\right)^2=\frac{\abs{\uf_\frakl}^{-n}}{\al_\frakl^{2n}}\cdot L^{\alg}(\onehalf,\pi_\cK\ot\chi\phi)\cdot C(\pi,\chi)\phi(\Csplit).\]
\end{prop}
\begin{proof}In view of \eqref{E:19.W}, it remains to compute $P_{\chi\phi}(\pi(\cmpt^\setn)\wtd V^m_+\vp^\dagger_{\chi\phi})^2$, which can be written as a product of local toric period integrals as in the proof of \thmref{T:main1.W}. We have computed these local period integrals in \secref{SS:toricintegralI.W} and \subsecref{SS:toricII.W} except for the local integral at $\frakl$, which will be carried out in the following \lmref{L:5.W}. The desired formula is obtained by combining these calculations.
\end{proof}

\begin{lm}\label{L:5.W}Suppose that $\chi_\frakl$ is unramified and $\phi\in\frakX_\frakl^0$ has conductor $\frakl^n,\,n\geq 1$. Then
\[\localP{\pi(\cmpt_\frakl^\setn)W_\frakl^\dagger}{\chi\phi}=\abs{\cD_{\cK_\frakl}}_{\cK_\frakl}^\onehalf \cdot\om_\frakl(\uf_\frakl^n)\abs{\uf_\frakl^n} L(1,\tau_{\cK_\frakl/\cF_\frakl})^2.\]
\end{lm}
\begin{proof}
Write $F=\cF_\frakl$ (resp. $E=\cK_\frakl$) and $\uf=\uf_\frakl$. For $t\in E$, we put
\[\iota_\cmpt^\setn(t):=(\cmpt_\frakl^\setn)^{-1}\iota(t)\cmpt_\frakl^\setn.\]

First we suppose $\frakl$ is split. A direct computation shows that
\begin{align*}\iota_\cmpt^\setn(t)=&\MX{1}{\uf^{-n}d_{F}^{-1}}{0}{1}\MX{t_\frakL}{}{0}{t_\frakLbar}\MX{1}{-\uf^{-n}d_{F}^{-1}}{0}{1};\\
(\cmpt_\frakl^\setn)^{-1}\cmJ\cmpt_\frakl^\setn=&\MX{1}{0}{\uf^n d_F}{-1}.
\end{align*}
We find that
\begin{align*}
&\om^{-1}(\det\cmpt_\frakl^\setn)\localP{\pi(\cmpt_\frakl^\setn)W_\frakl^\dagger}{\chi\phi}\\
=&\int_{F^\x}\int_{F^\x}W_{\frakl}^\dagger(\DII{ax}{1}\MX{1}{-\uf^{-n}d_F^{-1}}{0}{1})
W_{\frakl}^\dagger(\DII{-a}{1}\MX{1}{-\uf^{-n}d_F^{-1}}{0}{1})\om_\frakl^{-1}(a)\chi_{\frakL}\phi_\frakL(x) \dx a\dx x\\
=&\int_{\cO_F}\addchar(-d_F^{-1}\uf^{-n}x)\nu_\frakl\chi_\frakL\phi_\frakL\Abs^\onehalf(x)\dx x\cdot \int_{\cO_F}\addchar(d_F^{-1}\uf^{-n}a)\nu_\frakl\om_\frakl^{-1}\chi_\frakL^{-1}\phi_\frakL^{-1}\Abs^\onehalf(a)\dx a\\
=&\int_{\cO_F^\x}\addchar(-d_F^{-1}\uf^{-n}x)\phi_\frakL(x)\dx x\cdot \int_{\cO_F}\addchar(d_F^{-1}\uf^{-n}a)\phi_\frakL^{-1}(a)\dx a\\
=&\ep(1,\phi_\frakL^{-1},\addchar)\phi_{\frakL}(-1)\ep(1,\phi_\frakL,\addchar)\cdot\zeta_F(1)^2=\abs{\uf^n\cD_F}L(1,\tau_{E/F})^2.
\end{align*}
This proves the formula in the split case.

Now we suppose that $\frakl$ is inert. We shall retain the notation in \subsecref{SS:toricII.W}. Define $\bfm^\dagger:G_\frakl\to\C$ by \[\bfm^\dagger(g):=\bfb_\frakl(\pi(g)W_\frakl^\dagger,W_\frakl^\dagger).\]
Then $\bfm^\dagger(g)$ only depends on the double coset $I_\frakl gI_\frakl$. Put
\[\Prd^*:=\int_{E^\x/F^\x}\bfm^\dagger(\iota_\cmpt^\setn(t))\chi\phi(t)dt.\]
It is clear that \beq\label{E:18.W}\localP{\pi(\cmpt_\frakl^\setn)W_\frakl^\dagger}{\chi\phi}=\Prd^*\cdot\frac{L(1,\tau_{E/F})\om(\det\cmpt_\frakl^\setn)}{\zeta_F(1)}.\eeq
For $y\in \uf^r\cO_F^\x$, it is easy to verify that $\iota_\cmpt^\setn(1+y\OKbasis)\in I_\frakl$ if $r\geq n$ and
\begin{align*}\iota_\cmpt^\setn(1+y\OKbasis)&\in I_\frakl\bfw\DII{\uf^{n-r}}{\uf^{r-n}}I_\frakl\text{ if }0\leq r<n\quad(\bfw=\MX{0}{-d_F^{-1}}{d_F}{0}).\end{align*}
If $x\in\uf\cO_F$, then
\[\iota_\cmpt^\setn(x+\OKbasis)\in I_\frakl\bfw\DII{\uf^n}{\uf^{-n}}I_\frakl.\]
Note that $n=c_\frakl(\phi)=c_\frakl(\chi\phi)$ as in \eqref{E:alv.W}. Combined with the above observations and \lmref{L:halfsum.W}, a direct computation shows that
\begin{align*}
\Prd^*&=X_n\cdot\bfm^\dagger(1)+(-X_n)\cdot \bfm^\dagger(\bfw\DII{\uf}{\uf^{-1}})\\
&=\bfb_\frakl(W^0_\frakl-\pi(\DII{\uf^{-1}}{\uf})W^0_\frakl,W_\frakl^\dagger)\cdot X_n\\
&=(\frac{\mu_\frakl(\uf)}{1-\abs{\uf}}-\frac{\nu_\frakl(\uf)}{1-\mu_\frakl^{-1}\nu_\frakl\Abs(\uf)})\cdot\frac{1-\mu_\frakl^{-1}\nu_\frakl\Abs(\uf)}{\mu_\frakl(\uf)-\nu_\frakl(\uf)}\cdot \abs{\cD_F}^\onehalf X_n\\
&=\frac{1}{1-\abs{\uf}}\cdot L(1,\tau_{E/F})\abs{\uf^n}\abs{\cD_E}_E^\onehalf.
\end{align*}
The formula in the inert case follows from \eqref{E:18.W} immediately.
\end{proof}

\subsection{Proof of \thmref{T:main3.W}}\label{SS:proof}
We prove \thmref{T:main3.W} in this subsection. By the evaluation formula \propref{P:evaluation.W}, it boils down to proving that
\beq\label{E:20}\int_{Cl_\infty^-}\phi d\vp^\dagger_\chi\not\con 0\pmod{\frakm}\text{ for almost all $\phi\in\frakX_\frakl^0$}.\eeq

By \cite[Thm.\,3.2, 3.3]{Hida:nonvanishingmodp} together with the toric property of $\bff_\chi^\dagger$ \lmref{L:6.W} (\cf \cite[Lemma 6.1 and Remark 6.2]{Hsieh:Hecke_CM}), the validity of \eqref{E:20} is reduced to verifying the following condition:\begin{description}[\breaklabel\setlabelstyle{\itshape}]\item [$(\mathrm{H'})$]
For every $u\in \cO_{\cF_\frakl}$ and a positive integer $r$, there exist $\beta\in\cO_{\cF,\setp}^\x$ and $a\in(\AKf^{(p\None)})^\x$ such that $\beta\con u\pmod{\frakl^r}$ and \[\bfa_\beta(\bff^\dagger_\chi,\frakc(a))\not\con 0\pmod{\frakm}.\]
\end{description}
The verification of (H$'$) under the assumptions \eqref{ai${}_\cK$} and $\mu_p(\chpi)=0$ for all $v|\frakc_\chi^-$ follows from the same argument in \thmref{T:mu-invariant}. Note that for a polarization ideal $\frakc(a)$ $(\frakc=\frakc(\cO_\cK),\,a\in(\AKf^{(p\frakl)})^\x)$ and a totally positive $\beta\in\cO_{\cF,\setp}^\x\cap \cO_{\cF_\frakl}$, we have $(\frakc(a),p\frakl)=1$ and
\[\bfa_\beta(\bff^\dagger_\chi,\frakc(a))=\beta^{k/2}\prod_{v\ndivides p\frakl}\localK{\chi,v}^*(\beta\bfc_v^{-1}\rmN(a_v^{-1}))\cdot \nu_\frakl\Abs_{\cF_\frakl}^\onehalf(\beta)\quad(\il_\cF(\bfc)=\frakc).\]
Let $u\in \cO_{\cF_\frakl}$ and let $\eta^u=(\eta^u_v)$ be the idele in $\AF^\x$ such that $\bfa^*_{\lam,v}(\eta^u_v)\not\con 0\pmod{\frakm}$ for all $v|\frakc_\chi^-$, $\eta^u_\frakl=u$ and $\eta^u_v=1$ for all $v\ndivides \frakl\frakc_\chi^-$. To verify (H$'$), we simply proceed the Galois argument in \thmref{T:mu-invariant}, replacing $\eta$ in by $\eta^u$ therein. This completes the proof of \thmref{T:main3.W}. 